\documentclass{aart}

\usepackage{titlesec}
\titleformat{\subsection}[runin]{\normalfont\bfseries}{\thesubsection.}{.5em}{}[.]\titlespacing{\subsection}{0pt}{2ex plus .1ex minus .2ex}{.8em}
\titleformat{\subsubsection}[runin]{\normalfont\itshape}{\thesubsubsection.}{.3em}{}[.]\titlespacing{\subsubsection}{0pt}{1ex plus .1ex minus .2ex}{.5em}

\usepackage[letterpaper, hmargin=1in, top=1.5in, bottom=1.9in, footskip=0.7in]{geometry}

\usepackage{amsmath} 
\usepackage{amssymb}
\usepackage{amsfonts}
\usepackage{latexsym}
\usepackage{amsthm}
\usepackage{amsxtra}
\usepackage{amscd}
\usepackage{bbm}
\usepackage{mathrsfs}
\usepackage{bm}

\usepackage{xspace}

\usepackage{color}

\numberwithin{equation}{section}

\newcommand{\rd}{{\rm d}}

\newcommand{\for}{\qquad \text{for} \quad}
\newcommand{\where}{\qquad \text{where} \quad}

\newcommand{\bx}{{\bf{x}}}

\newcommand{\bv}{{\bf{v}}}
\newcommand{\bw}{{\bf{w}}}

\newcommand{\be}{\begin{equation}}
\newcommand{\ee}{\end{equation}}

\newcommand{\hp}[2]{with $(#1,#2)$-high probability\xspace}

\newcommand{\e}{{\varepsilon}}

\newcommand{\om}{{\omega}}

\newcommand{\cN}{{\cal N}}

\newcommand{\f}[1]{\bm{\mathrm{#1}}} 

\renewcommand{\cal}{\mathcal} 
 
\newcommand{\fra}{\mathfrak} 
\newcommand{\ol}[1]{\overline{#1} \!\,} 
\newcommand{\wh}{\widehat}
\newcommand{\wt}{\widetilde}

\newcommand{\me}{\mathrm{e}} 
\newcommand{\ii}{\mathrm{i}} 
\newcommand{\dd}{\mathrm{d}}
\newcommand{\col}{\mathrel{\vcenter{\baselineskip0.75ex \lineskiplimit0pt \hbox{.}\hbox{.}}}}
\newcommand*{\deq}{\mathrel{\vcenter{\baselineskip0.65ex \lineskiplimit0pt \hbox{.}\hbox{.}}}=}

\newcommand{\umat}{\mathbbmss{1}} 
\renewcommand{\epsilon}{\varepsilon}
\renewcommand{\leq}{\leqslant}
\renewcommand{\geq}{\geqslant}

\newcommand{\ind}[1]{\f 1 (#1)}
\newcommand{\indb}[1]{\f 1 \pb{#1}}
\newcommand{\indB}[1]{\f 1 \pB{#1}}
\newcommand{\indbb}[1]{\f 1 \pbb{#1}}

\renewcommand{\le}{\leq}
\renewcommand{\ge}{\geq}

\renewcommand{\P}{\mathbb{P}}
\newcommand{\E}{\mathbb{E}}
\newcommand{\R}{\mathbb{R}}
\newcommand{\C}{\mathbb{C}}
\newcommand{\N}{\mathbb{N}}

\newcommand{\IE}{\mathbb{I} \mathbb{E}}

\newcommand{\pb}[1]{\bigl({#1}\bigr)}
\newcommand{\pB}[1]{\Bigl({#1}\Bigr)}
\newcommand{\pbb}[1]{\biggl({#1}\biggr)}
\newcommand{\pBB}[1]{\Biggl({#1}\Biggr)}
\newcommand{\pa}[1]{\left({#1}\right)}

\newcommand{\qb}[1]{\bigl[{#1}\bigr]}
\newcommand{\qB}[1]{\Bigl[{#1}\Bigr]}
\newcommand{\qbb}[1]{\biggl[{#1}\biggr]}
\newcommand{\qBB}[1]{\Biggl[{#1}\Biggr]}
\newcommand{\qa}[1]{\left[{#1}\right]}

\newcommand{\h}[1]{\{{#1}\}}
\newcommand{\hb}[1]{\bigl\{{#1}\bigr\}}

\newcommand{\hbb}[1]{\biggl\{{#1}\biggr\}}
\newcommand{\hBB}[1]{\Biggl\{{#1}\Biggr\}}

\newcommand{\abs}[1]{\lvert #1 \rvert}
\newcommand{\absb}[1]{\bigl\lvert #1 \bigr\rvert}
\newcommand{\absB}[1]{\Bigl\lvert #1 \Bigr\rvert}
\newcommand{\absbb}[1]{\biggl\lvert #1 \biggr\rvert}
\newcommand{\absBB}[1]{\Biggl\lvert #1 \Biggr\rvert}
\newcommand{\absa}[1]{\left\lvert #1 \right\rvert}

\newcommand{\norm}[1]{\lVert #1 \rVert}

\newcommand{\avg}[1]{\langle #1 \rangle}
\newcommand{\avgb}[1]{\bigl\langle #1 \bigr\rangle}
\newcommand{\avgB}[1]{\Bigl\langle #1 \Bigr\rangle}
\newcommand{\avgbb}[1]{\biggl\langle #1 \biggr\rangle}

\newcommand{\scalar}[2]{\langle{#1} \mspace{2mu}, {#2}\rangle}

\newcommand{\bra}[1]{\langle #1 |}

\newcommand{\ket}[1]{| #1 \rangle}

\DeclareMathOperator{\tr}{Tr}

\DeclareMathOperator{\re}{Re}
\DeclareMathOperator{\im}{Im}

\theoremstyle{plain} 
\newtheorem{theorem}{Theorem}[section]
\newtheorem*{theorem*}{Theorem}
\newtheorem{lemma}[theorem]{Lemma}
\newtheorem*{lemma*}{Lemma}

\newtheorem*{corollary*}{Corollary}
\newtheorem{proposition}[theorem]{Proposition}
\newtheorem*{proposition*}{Proposition}
\newtheorem{definition}[theorem]{Definition}
\newtheorem*{definition*}{Definition}

\newtheorem*{example*}{Example}
\theoremstyle{remark} 
\newtheorem{remark}[theorem]{Remark}

\newtheorem*{remark*}{Remark}
\newtheorem*{remarks*}{Remarks}

\DeclareMathOperator{\law}{law}

\title{Spectral Statistics of Erd{\H o}s-R\'enyi Graphs II:\\ Eigenvalue Spacing and the Extreme Eigenvalues}

\author{
L\'aszl\'o Erd\H os${}^1$\thanks{Partially supported
by SFB-TR 12 Grant of the German Research Council} \quad
Antti Knowles${}^2$\thanks{Partially supported by NSF grant DMS-0757425} \quad
Horng-Tzer Yau${}^2$\thanks{Partially supported
by NSF grants DMS-0757425, 0804279}  \quad
Jun Yin${}^2$\thanks{Partially supported by NSF grant DMS-1001655} \\\\
Institute of Mathematics, University of Munich, \\
Theresienstrasse 39, D-80333 Munich, Germany \\ lerdos@math.lmu.de ${}^1$ \\ \\
Department of Mathematics, Harvard University\\
Cambridge MA 02138, USA \\ knowles@math.harvard.edu \quad htyau@math.harvard.edu \quad  jyin@math.harvard.edu ${}^2$ \\ 
\\}

\begin{document}

\maketitle

\vspace{0.5cm}

\begin{abstract}

We consider the ensemble of adjacency matrices of Erd{\H o}s-R\'enyi random graphs, i.e.\ graphs on $N$ vertices where
 every edge is chosen independently and with probability $p \equiv p(N)$. We rescale the matrix so that its bulk 
eigenvalues are of order one.  Under the assumption  $p N \gg N^{2/3}$, we prove  the universality of eigenvalue 
distributions both in the bulk and at the edge of the spectrum. More precisely, we prove (1) that the eigenvalue spacing 
of the Erd{\H o}s-R\'enyi graph in the bulk of the spectrum has the same distribution as that of the Gaussian orthogonal 
ensemble; and (2) that the second largest eigenvalue of the Erd{\H o}s-R\'enyi graph has the same distribution as the 
largest eigenvalue of the Gaussian orthogonal ensemble.
As an application of our method, we prove the bulk universality of generalized Wigner matrices under the assumption that 
the matrix entries have at least $4 + \epsilon$ moments. 
\end{abstract}

\vspace{1cm}


\medskip

\medskip


\medskip

\newpage


\section{Introduction}

The Erd{\H o}s-R\'enyi ensemble \cite{ER1, ER2} is a law of a random graph on $N$ vertices, in which
each edge is chosen independently  with probability  $p \equiv p(N)$. The corresponding adjacency matrix is called the 
Erd{\H o}s-R\'enyi matrix.  Since each row and column has typically $pN$ nonzero entries, the matrix
is sparse as long as $p\ll 1$. We shall refer to $pN$ as the sparseness parameter of the matrix.
In the companion  paper \cite{EKYY}, we  established  the local semicircle law for
the Erd{\H o}s-R\'enyi matrix for $p N \ge (\log N)^C$, i.e.\ we showed that, assuming $p N \ge (\log N)^C$, the 
eigenvalue density is given by the Wigner semicircle law in any spectral window containing on average at least $(\log 
N)^{C'}$ eigenvalues. In this paper, we  use this result to prove both the bulk  and edge universalities  for the  
Erd{\H o}s-R\'enyi matrix   under the restriction that the sparseness parameter satisfies
\be\label{res}
p N \;\gg\; N^{2/3}.
\ee
More precisely, assuming that $p$ satisfies \eqref{res},   we prove that the eigenvalue spacing of the Erd{\H 
o}s-R\'enyi graph in the bulk of the spectrum has the same distribution as that of the Gaussian orthogonal ensemble 
(GOE). In order to outline the statement of the edge universality for the Erd\H{o}s-R\'enyi graph, we observe that, 
since the matrix elements of the Erd{\H o}s-R\'enyi ensemble are either $0$ or $1$, they do not satisfy the mean zero 
condition which typically appears in the random matrix literature.  In particular, the largest eigenvalue of the Erd{\H 
o}s-R\'enyi matrix is very large and lies far away from the rest of the spectrum.  We normalize the Erd{\H o}s-R\'enyi 
matrix so that the bulk of its spectrum lies in the interval $[-2, 2]$.  By the edge universality of the Erd{\H 
o}s-R\'enyi ensemble, we therefore mean that its second largest eigenvalue has the same distribution as the largest 
eigenvalue of the GOE, which is the well-known Tracy-Widom distribution. We prove the edge universality under the 
assumption \eqref{res}.

Neglecting the mean zero condition, the Erd{\H o}s-R\'enyi matrix becomes a Wigner random matrix with a Bernoulli 
distribution when $0<p<1$ is a constant independent of $N$. Thus for $p \ll 1$ we can view the Erd{\H o}s-R\'enyi 
matrix, up to a shift in the expectation of the matrix entries, as a singular Wigner matrix for which the probability 
distributions of the matrix elements are highly concentrated at zero. Indeed, the probability for a single entry to be 
zero is $1 - p$.
Alternatively, we can express the singular nature of the Erd{\H o}s-R\'enyi ensemble by the fact that the $k$-th moment 
of a matrix entry is bounded by
\be\label{dec}
N^{-1}  (p N)^{-(k- 2)/2}\,.
\ee
For $p \ll 1$ this decay in $k$ is much slower than in the case of Wigner matrices.

There has been spectacular progress in the understanding of the universality of eigenvalue distributions for invariant 
random matrix ensembles \cite{BI, DKMVZ1, DKMVZ2, PS2, PS}. The Wigner and Erd{\H o}s-R\'enyi matrices are not invariant 
ensembles, however. The moment method \cite{SS, Sosh, So2} is a powerful means for establishing edge universality. In 
the context of sparse matrices, it was applied in \cite{So2} to prove edge universality for the zero mean version of the 
$d$-regular graph, where the matrix entries take on the values $-1$ and $1$ instead of $0$ and $1$. The need for this 
restriction can be ascribed to the two following facts.  First, the moment method is suitable for treating the largest 
and smallest eigenvalues. But in the case of the Erd{\H o}s-R\'enyi matrix, it is the second largest eigenvalue, not the 
largest one, which behaves like the largest eigenvalue of the GOE.  Second, the modification of the moment method to 
matrices with non-symmetric distributions poses a serious technical challenge.

A general approach to proving the universality of Wigner matrices was recently developed in the series of papers 
\cite{ESY1, ESY2, ESY3,  ESY4,  ESYY, EYY, EYY2, EYYrigidity}.  In this paper, we further extend this method to cover 
sparse matrices such as the Erd\H{o}s-R\'enyi matrix in the range \eqref{res}. Our approach is based on the following 
three ingredients. (1) A local semicircle law -- a precise estimate of the local eigenvalue density down to energy 
scales containing around $(\log N)^C$ eigenvalues.  (2) Establishing universality of the eigenvalue distribution of 
Gaussian
 divisible ensembles, via an estimate on the rate of decay to local equilibrium of the Dyson Brownian motion \cite{Dy}.  
(3) A density argument which shows that for any probability distribution of the matrix entries there exists a Gaussian 
divisible distribution such that the two  associated Wigner ensembles have identical local eigenvalue statistics down to 
the  scale $1/N$.  In the case of Wigner matrices, the edge universality can also be obtained by a modification of (1) 
and (3) \cite{EYYrigidity}.  The class of ensembles to which  this method applies is  extremely general. So far it 
includes all (generalized) Wigner matrices under the sole assumption that the distributions of the matrix elements have 
a uniform subexponential decay.  In this paper we extend this method to the Erd{\H o}s-R\'enyi matrix, which in fact
represents a generalization in two unrelated directions: (a) the law of the matrix entries is much more singular, and 
(b) the matrix elements have nonzero mean.

%

As an application of the local semicircle law for sparse matrices proved in \cite{EKYY}, we also prove the bulk 
universality for generalized Wigner matrices under the sole assumption that the matrix  entries have $4 + \epsilon$ 
moments.  This  relaxes the subexponential decay condition on the tail of the distributions assumed in \cite{EYY, EYY2, 
EYYrigidity}. Moreover, we prove the edge universality of Wigner matrices under the assumption that the matrix entries 
have $12 + \epsilon$ moments. These results on Wigner matrices are stated and proved in Section \ref{sec4} below.
We note that in \cite{ABAP} it was proved that the distributions of the largest eigenvalues are Poisson if the entries 
have at most $4-\e$ moments. Numerical results \cite{BBP} predict that the existence of four moments corresponds to a 
sharp transition point, where the transition is from the Poisson process to the determinantal point process with Airy 
kernel.

We remark that  the bulk universality  for Hermitian Wigner matrices was also obtained in
 \cite{TV}, partly by using the result of \cite{J} and the local semicircle law from Step (1).  For  real symmetric 
Wigner matrices, the bulk universality in \cite{TV} requires  that the first four moments of every matrix element 
coincide with those of the standard Gaussian random variable.  In particular, this restriction rules out the real 
Bernoulli Wigner matrices, which may be regarded as the simplest kind of an Erd{\H o}s-R\'enyi matrix (again neglecting 
additional difficulties arising from the nonzero mean of the entries).

As a first step in our general strategy to prove universality, we proved, in the companion paper \cite{EKYY}, a local 
semicircle law stating   that the  eigenvalue distribution of the  Erd{\H o}s-R\'enyi  ensemble   in   any spectral 
window which on average contains at least $(\log N)^C$ eigenvalues  is given by the  Wigner semicircle law.  As a 
corollary, we proved that the eigenvalue locations are equal to those predicted by the semicircle law, up to an error of 
order $(p N)^{-1}$. The second step of the strategy outlined above for Wigner matrices is to estimate the local 
relaxation time of the Dyson Brownian motion \cite{ ESY4,  ESYY}. This is achieved by constructing a pseudo-equilibrium 
measure and estimating the global relaxation time to this measure. For models with nonzero mean, such as the 
Erd\H{o}s-R\'enyi matrix, the largest eigenvalue is located very far from its equilibrium position, and moves rapidly 
under the Dyson Brownian motion. Hence a uniform approach to equilibrium is impossible. We overcome this problem by 
integrating out the largest eigenvalue from the joint probability distribution of the eigenvalues, and consider the flow 
of the marginal distribution of the remaining $N - 1$ eigenvalues. This enables us to establish bulk universality for 
sparse matrices with nonzero mean under the restriction \eqref{res}. This approach trivially also applies to Wigner 
matrices whose entries have nonzero mean.

Since the eigenvalue locations are only established with accuracy $(p N)^{-1}$, the local relaxation time for the Dyson 
Brownian motion with the initial data given by the Erd{\H o}s-R\'enyi ensemble is only shown to be less than $1/ (p^2 N) 
\gg 1/N$. For Wigner ensembles, it was proved in \cite{EYYrigidity} that the local relaxation time is of order $1/N$.  
Moreover, the slow decay of the third moment of the Erd{\H o}s-R\'enyi matrix entries, as given in \eqref{dec}, makes 
the approximation in Step (3) above less effective. These two effects impose the restriction \eqref{res} in our proof of 
bulk universality.  At the end of Section \ref{sec:def} we give a more detailed account of how this restriction arises.  
The reason for the same restriction's being needed for the edge universality is different; see Section \ref{sect: proof 
of green fn comp}. We note, however, that both the bulk and edge universalities are expected to hold without this 
restriction, as long as the graphs are not too sparse in the sense that $p N \gg \log N$; for $d$-regular graphs this condition is conjectured to be the weaker $pN \gg 1$ \cite{Sa}. A discussion of related 
problems on $d$-regular graphs can be found in \cite{ MNS}.

\medskip \noindent
{\bf Acknowledgement.} We thank P.\ Sarnak for bringing the problem of universality of sparse matrices to our 
attention.

\section{Definitions and results}\label{sec:def}

We begin this section by introducing a class of $N \times N$ sparse random matrices $A \equiv A_N$.
 Here $N$ is a large 
parameter. (Throughout the following we shall often refrain from explicitly indicating $N$-dependence.)

The motivating example is the \emph{Erd\H{o}s-R\'enyi matrix}, or the
adjacency matrix of the \emph{Erd\H{o}s-R\'enyi random graph}. Its entries are independent (up to the constraint that 
the matrix be symmetric), and equal to $1$ with probability $p$ and $0$ with probability $1 - p$.  For our purposes it 
is convenient to replace $p$ with the new parameter $q \equiv q(N)$, defined through $p = q^2 / N$.  Moreover, we 
rescale the matrix in such a way that its bulk eigenvalues typically lie in an interval of size of order one.

Thus we are led to the following definition. Let $A = (a_{ij})$ be the symmetric $N\times N$ matrix whose entries 
$a_{ij}$ are independent
(up to the symmetry constraint $a_{ij}=a_{ji}$) and each element is distributed according to
\begin{equation}\label{sparsedef}
a_{ij} \;=\;
\frac{\gamma}{q}
\begin{cases}
1 & \text{with probability } \frac{q^2}{N}
\\
0 & \text{with probability } 1 - \frac{q^2}{N}\,.
\end{cases}
\end{equation}
Here $\gamma \deq (1 - q^2 / N)^{-1/2}$ is a scaling introduced for convenience.  The parameter $q \leq N^{1/2}$ 
expresses the sparseness of the matrix; it may depend on $N$.
Since $A$ typically has $q^2 N$ nonvanishing entries, we find that if $q \ll N^{1/2}$ then the matrix is sparse.

We extract the mean of each matrix entry and write
\begin{equation*}
A \;=\; H + \gamma q \, \ket{\f e} \bra{\f e}\,,
\end{equation*}
where the entries of $H$ (given by $h_{ij} = a_{ij} - \gamma q/N$) have mean zero, and we defined the vector
\begin{equation} \label{def of b e}
\f e \;\equiv\; \f e_N \;\deq\; \frac{1}{\sqrt{N}} (1, \dots, 1)^T\,.
\end{equation}
Here we use the notation $\ket{\f e} \bra{\f e}$ to denote the orthogonal projection onto $\f e$, i.e.\ $(\ket{\f e} 
\bra{\f e})_{ij} \deq N^{-1}$.


One readily finds that the matrix elements of $H$ satisfy the moment bounds
\begin{equation}
\E h_{ij}^2 \;=\; \frac{1}{N}\,, \qquad
\E \absb{h_{ij}}^p \;\leq\; \frac{1}{N q^{p - 2}}\,,
\end{equation}
where $p \geq 2$.

More generally, we consider the following class of random matrices
with non-centred entries characterized by two parameters $q$ and $f$,
which may be $N$-dependent. The parameter $q$ expresses how singular
the distribution of $h_{ij}$ is; in particular, it expresses
the sparseness of $A$ for the special case \eqref{sparsedef}. The parameter $f$ determines the nonzero expectation value 
of the matrix elements.

\begin{definition}[$H$] \label{definition of H}
We consider $N \times N$ random matrices $H = (h_{ij})$ whose entries are real and independent up to the symmetry 
constraint $h_{ij} = h_{ji}$. We assume that the elements of $H$ satisfy the moment conditions
\begin{equation} \label{moment conditions}
\E h_{ij} \;=\; 0\,, \qquad \E \abs{h_{ij}}^2 \;=\; \frac{1}{N}
\,, \qquad
\E \abs{h_{ij}}^p \;\leq\; \frac{C^p}{N q^{p - 2}}
\end{equation}
for $1 \leq i,j \leq N$ and $2 \leq p \leq (\log N)^{10 \log \log N}$, where $C$ is a positive constant. Here $q \equiv 
q(N)$ satisfies
\begin{equation} \label{lower bound on d}
(\log N)^{15 \log \log N} \;\leq\; q \;\leq\; C N^{1/2}
\end{equation}
for some positive constant $C$.
\end{definition}

\begin{definition}[$A$] \label{definition of A}
Let $H$ satisfy Definition \ref{definition of H}. Define the matrix $A = (a_{ij})$ through
\begin{align} \label{A = H + fee}
A \;\deq\; H + f \ket{\f e} \bra{\f e}\,,
\end{align}
where $f \equiv f(N)$ is a deterministic number that satisfies
\begin{align} \label{lower bound on f}
1 + \epsilon_0 \;\leq\; f \;\leq\; N^C\,,
\end{align}
for some constants $\epsilon_0 > 0$ and $C$.
\end{definition}



%
%
\begin{remark}
For definiteness, and bearing the Erd\H{o}s-R\'enyi matrix in mind, we restrict ourselves to real symmetric matrices 
satisfying Definition \ref{definition of A}. However, our proof applies equally to complex Hermitian sparse matrices.
\end{remark}

\begin{remark} \label{remark: GOE assumption}
As observed in \cite{EKYY}, Remark 2.5, we may take $H$ to be a Wigner matrix whose entries have subexponential decay 
$\E \abs{h_{ij}}^p \leq (Cp)^{\theta p} N^{-p/2}$ by choosing $q = N^{1/2} (\log N)^{-5 \theta \log \log N}$.
\end{remark}

We shall use $C$ and $c$ to denote generic positive constants which may only depend on the constants in assumptions such 
as \eqref{moment conditions}.
Typically, $C$ denotes a large constant and $c$ a small constant.
Note that the fundamental large parameter of our model is $N$, and the notations $\gg, \ll, O(\cdot), o(\cdot)$ always 
refer to the limit $N \to \infty$. Here $a \ll b$ means $a = o(b)$. We  write $a \sim b$ for $C^{-1} a \leq b \leq C a$.

After these preparations, we may now state our results. They concern the distribution of the eigenvalues of $A$, which 
we order in a nondecreasing fashion and denote by
$\mu_1 \leq \cdots \leq \mu_N$. We shall only consider the distribution of the $N - 1$ first eigenvalues $\mu_1, \dots, 
\mu_{N - 1}$. The largest eigenvalue $\mu_N$ lies far removed from the others, and its distribution is known to be 
normal with mean $f + f^{-1}$ and variance $N^{-1/2}$; see \cite{EKYY}, Theorem 6.2, for more details.

First, we establish the bulk universality of eigenvalue correlations. Let $p(\mu_1, \dots, \mu_N)$ be the probability 
density\footnote{Note that we use the density of the law of the eigenvalue density for simplicity of notation, but our 
results remain valid when no such density exists.} of the ordered eigenvalues $\mu_1 \leq \cdots \leq \mu_N$ of $A$.  
Introduce the marginal density
\begin{align*}
p_N^{(N - 1)}(\mu_1, \dots, \mu_{N - 1}) \;\deq\; \frac{1}{(N - 1)!} \sum_{\sigma \in S_{N - 1}} \int \dd
\mu_N \, p(\mu_{\sigma(1)}, \dots, \mu_{\sigma(N - 1)}, \mu_N)\,.
\end{align*}
In other words, $p_N^{(N - 1)}$ is the symmetrized probability density of the first $N - 1$ eigenvalues of $H$.  For $n 
\leq N-1$ we define the $n$-point correlation function (marginal) through
\begin{align} \label{definition of marginal}
p_N^{(n)}(\mu_1, \dots, \mu_n) \;\deq\; \int \dd \mu_{n + 1} \cdots \dd \mu_{N - 1} \, p_N^{(N - 
1)}(\mu_1, \dots, \mu_{N - 1})\,.
\end{align}
Similarly, we denote by $p_{{\rm GOE}, N}^{(n)}$ the $n$-point correlation function of the symmetrized eigenvalue density of an $N \times N$ GOE matrix.

\begin{theorem}[Bulk universality] \label{theorem: bulk universality}
Suppose that $A$ satisfies Definition \ref{definition of A} with $q \geq N^\phi$ for some $\phi$ satisfying $0 < \phi
\leq 1/2$, and that $f$ additionally satisfies $f \leq C N^{1/2}$ for some $C > 0$.  Let $\beta > 0$ and assume that
\begin{equation} \label{exponent condition for bulk universality}
\phi \;>\; \frac{1}{3} + \frac{\beta}{6}\,.
\end{equation}
Let $E \in (-2,2)$ and take a sequence $(b_N)$ satisfying $N^{\epsilon - \beta} \leq b_N \leq \abs{\abs{E} - 2} / 2$ for 
some $\epsilon > 0$.  Let $n \in \N$ and $O : \R^n \to \R$ be compactly supported and continuous.
Then
\begin{multline*}
\lim_{N \to \infty} \int_{E - b_N}^{E + b_N} \frac{\dd E'}{2 b_N} \int \dd \alpha_1 \cdots \dd \alpha_n\, O(\alpha_1, 
\dots, \alpha_n)
\\
{}\times{} \frac{1}{\varrho_{sc}(E)^n}\pb{p_{N}^{(n)} - p_{{\rm GOE}, N}^{(n)}} \pbb{E' +
\frac{\alpha_1}{N\varrho_{sc}(E)}, 
\dots, E' + \frac{\alpha_n}{N\varrho_{sc}(E)}} \;=\; 0\,,
\end{multline*}
where we abbreviated
\begin{equation} \label{def rho sc}
\varrho_{sc}(E) \;\deq\; \frac{1}{2 \pi}\sqrt{[4 - E^2]_+}
\end{equation}
for the density of the semicircle law.
\end{theorem}

\begin{remark}
Theorem \ref{theorem: bulk universality} implies bulk universality for sparse matrices provided that $1/3 < \phi \leq
1/2$. See
the end of this section for an account on the origin of the condition \eqref{exponent condition for bulk universality}.
\end{remark}

We also prove the universality of the extreme eigenvalues.

\begin{theorem}[Edge universality] \label{thm: edge}
Suppose that $A$ satisfies Definition \ref{definition of A} with $q \geq N^\phi$ for some $\phi$ satisfying $1/3 < \phi
\leq 1/2$. Let $V$ be an $N \times N$ GOE matrix whose eigenvalues we denote by $\lambda_1^V \leq \cdots \leq 
\lambda_N^V$.  Then there is a $\delta > 0$ such that for any $s$ we have
\begin{equation}\label{tw}
\P^V\pB{N^{2/3}(\lambda_N^V - 2) \leq s - N^{-\delta}} - N^{- \delta} \;\leq\; \P^A \pb{N^{2/3}(\mu_{N - 1} - 2) \leq s}
\;\leq\;
\P^V\pB{N^{2/3}(\lambda_N^V - 2) \leq s + N^{-\delta}} + N^{- \delta}
\end{equation}
as well as
\begin{equation}\label{tw01}
\P^V\pB{N^{2/3}(\lambda_1^V +2) \leq s - N^{-\delta}} - N^{- \delta} \;\leq\; \P^A
 \pb{N^{2/3}(\mu_1+ 2) \leq s}
\;\leq\;
\P^V\pB{N^{2/3}(\lambda_1^V+ 2) \leq s + N^{-\delta}} + N^{- \delta}\,,
\end{equation}
for $N \geq N_0$, where $N_0$ is independent of $s$. Here $\P^V$ denotes the law of the GOE matrix $V$, and $\P^A$ the
law of the sparse matrix $A$.
\end{theorem}

\begin{remark}
Theorem \ref{twthm} can be easily extended to correlation functions of a finite collection of extreme eigenvalues.
\end{remark}

\begin{remark}
The GOE distribution function $F_1(s) \deq \lim_N \P^V \pb{N^{2/3}(\lambda_N^V - 2) \leq s}$ of the largest
eigenvalue of $V$ has been identified by Tracy and Widom \cite{TW, TW2}, and can be computed in terms of Painlev\'e
equations. A similar result holds for the smallest
eigenvalue $\lambda_1^V$ of $V$.  \end{remark}

\begin{remark}
A result analogous to Theorem \ref{thm: edge} holds for the extreme eigenvalues of the centred sparse matrix $H$; see 
\eqref{tw0} below.
\end{remark}

We conclude this section by giving a sketch of the origin of the restriction
$\phi > 1/3$ in Theorem \ref{theorem: bulk universality}.
To simplify the outline of the argument, we set $\beta = 0$ in Theorem \ref{theorem: bulk universality} and ignore any
powers of $N^\epsilon$.  The proof of Theorem \ref{theorem: bulk universality} is based on an analysis of the local
relaxation properties of the marginal Dyson Brownian motion, obtained from the usual Dyson Brownian motion by integrating
out the largest eigenvalue $\mu_N$. As an input, we need the bound
\begin{equation} \label{Q in intro}
Q \;\deq\; \E \sum_{\alpha = 1}^{N - 1} \abs{\mu_\alpha - \gamma_\alpha}^2 \;\leq\; N^{1 - 4 \phi},
\end{equation}
where $\gamma_\alpha$ denotes the classical location of the $\alpha$-th eigenvalue (see \eqref{def of gamma} below). The
bound \eqref{Q in intro} was proved in \cite{EKYY}. In that paper we prove, roughly, that $\abs{\mu_\alpha -
\gamma_\alpha} \leq q^{-2} \leq N^{-2 \phi}$, from which \eqref{Q in intro} follows. The precise form is given in
\eqref{main estimate on Q}. We then take an arbitrary initial
sparse matrix ensemble $A_0$ and evolve it according to the Dyson Brownian motion up to a time $\tau = N^{- \rho}$, for
some $\rho > 0$. We
prove that the local spectral statistics, in the first $N - 1$ eigenvalues, of the evolved ensemble $A_\tau$ at time
$\tau$ coincide with those of a GOE
matrix $V$, provided that
\begin{equation}
Q \tau^{-1} \;=\; Q N^\rho \;\ll\; 1\,.
\end{equation}
The precise statement is given in \eqref{local ergodicity of DBM}. This gives us the condition
\begin{equation} \label{exponent condition 1}
1 - 4 \phi + \rho < 0\,.
\end{equation}
Next, we compare the local spectral statistics of a given Erd\H{o}s-R\'enyi matrix $A$ with those of the time-evolved
ensemble $A_\tau$ by constructing an appropriate initial $A_0$, chosen so that the first four moments of $A$ and $A_\tau$ are
close. More precisely, by comparing Green functions, we prove that the local spectral statistics of $A$ and $A_\tau$ coincide if the first three moments of the entries
of $A$ and $A_\tau$ coincide and their fourth moments differ by at most $N^{-2 - \delta}$ for some $\delta > 0$. (See Proposition
\ref{proposition: Green function comparison}.) Given $A$ we find, by explicit construction, a sparse matrix $A_0$ such
that the first three moments of the entries of $A_\tau$ are equal to those of $A$, and their fourth moments differ by at
most $N^{-1 - 2 \phi} \tau = N^{-1 - 2 \phi - \rho}$; see \eqref{moment matching result}. Thus the local spectral statistics
of $A$ and $A_\tau$ coincide provided that
\begin{equation} \label{exponent condition 2}
1 - 2 \phi - \rho  < 0\,.
\end{equation}
From the two conditions \eqref{exponent condition 1} and \eqref{exponent condition 2} we find that the local spectral
statistics of $A$ and $V$ coincide provided that $\phi > 1/3$.

\section{The strong local semicircle law and eigenvalue locations}

In this preliminary section we collect the main notations and tools from the companion paper \cite{EKYY} that we shall 
need for the proofs. Throughout this paper we shall make use of the parameter
\begin{align} \label{bounds on xi}
\xi \;\equiv\; \xi_N \;\deq\; 5 \log \log N\,,
\end{align}
which will keep track of powers of $\log N$ and probabilities of high-probability events.
Note that in \cite{EKYY}, $\xi$ was a free parameter. In this paper we choose the special form \eqref{bounds on xi} for 
simplicity.

We introduce the spectral parameter
\begin{equation*}
z \;=\; E + \ii \eta
\end{equation*}
where $E \in \R$ and $\eta > 0$. Let $\Sigma \geq 3$ be a fixed but arbitrary constant and define the domain
\begin{equation} \label{definitionD}
D_L \;\deq\; \hb{z \in \C \col \abs{E} \leq \Sigma \,,\, (\log N)^LN^{-1} \leq \eta \leq 3}\,,
\end{equation}
with a parameter $L \equiv L(N)$ that always satisfies
\begin{equation} \label{lower bound on L}
L \;\geq\; 8 \xi\,.
\end{equation}
For $\im z > 0$ we define the Stieltjes transform of the local semicircle law
\begin{equation} \label{def m sc}
m_{sc}(z) \;\deq\; \int_\R \frac{\varrho_{sc}(x)}{x - z} \, \dd x\,,
\end{equation}
where the density $\varrho_{sc}$ was defined in \eqref{def rho sc}.
The Stieltjes transform $m_{sc}(z) \equiv m_{sc}$ may also be characterized as the unique solution of
\begin{equation} \label{identity of msc}
m_{sc} + \frac{1}{z + m_{sc}} \;=\; 0
\end{equation}
satisfying $\im m_{sc}(z) > 0$ for $\im z > 0$. This implies that
\begin{equation} \label{formula for m sc}
m_{sc}(z) \;=\; \frac{-z + \sqrt{z^2 - 4}}{2}\,,
\end{equation}
where the square root is chosen so that $m_{sc}(z) \sim - z^{-1}$ as $z \to \infty$. We define the resolvent of $A$ 
through
\begin{align*}
G(z) \;\deq\; (A - z)^{-1}\,,
\end{align*}
as well as the Stieltjes transform of the empirical eigenvalue density
\begin{equation*}
m(z) \;\deq\; \frac{1}{N} \tr G(z)\,.
\end{equation*}
For $x \in \R$ we define the distance $\kappa_x$ to the spectral edge through
\begin{align} \label{def kappa}
\kappa_x \;\deq\; \absb{\abs{x} - 2}\,.
\end{align}
At this point we warn the reader that we depart from our conventions in \cite{EKYY}. In that paper, the quantities 
$G(z)$ and $m(z)$ defined above in terms of $A$ bore a tilde to distinguish them from the same quantities defined in 
terms of $H$. In this paper we drop the tilde, as we shall not need resolvents defined in terms of $H$.

We shall frequently have to deal with events of very high probability, for which the following definition is useful. It 
is characterized by two positive parameters, $\xi$ and $\nu$, where $\xi$ is given by \eqref{bounds on xi}.

\begin{definition}[High probability events]
We say that an $N$-dependent event $\Omega$ holds with \emph{$(\xi, \nu)$-high probability} if
\begin{equation} \label{high prob}
\P(\Omega^c) \;\leq\; \me^{-\fra \nu (\log N)^\xi}
\end{equation}
for $N \geq N_0(\nu)$.

Similarly, for a given event $\Omega_0$, we say that $\Omega$ \emph{holds with $(\xi, \nu)$-high probability on 
$\Omega_0$} if
\begin{equation*}
\P(\Omega_0 \cap \Omega^c) \;\leq\; \me^{- \nu (\log N)^\xi}
\end{equation*}
for $N \geq N_0(\nu)$.
\end{definition}

\begin{remark}
In the following we shall not keep track of the explicit value of $\nu$;
in fact we allow $\nu$ to decrease from one line to another
without introducing a new notation. All of our results will hold for $\nu \leq \nu_0$, where $\nu_0$ depends only on the 
constants $C$ in Definition \ref{definition of H} and the parameter $\Sigma$ in \eqref{definitionD}.
\end{remark}

\begin{theorem}[Local semicircle law \cite{EKYY}] \label{LSCTHM}
Suppose that $A$ satisfies Definition \ref{definition of A} with the condition \eqref{lower bound on f} replaced with
\begin{equation} \label{lax bounds on f}
0 \;\leq\; f \;\leq\; N^C\,.
\end{equation}
Moreover, assume that
\begin{align} \label{assumptions for SLSC}
q &\;\geq\; (\log N)^{120 \xi}\,,
\\
\label{def of L}
L &\;\geq\; 120 \xi\,.
\end{align}
Then there is a constant $\nu > 0$, depending on $\Sigma$ and the constants $C$ in \eqref{moment conditions} and 
\eqref{lower bound on d}, such that the following holds.

We have the local semicircle law: the event
\begin{align}\label{scm}
\bigcap_{z \in D_L} \hBB{\absb{m(z) - m_{sc}(z)} \leq (\log N)^{40 \xi} \pbb{
\min \hbb{\frac{ (\log N)^{40 \xi} }{\sqrt{\kappa_E+\eta}}\frac{1}{q^2}, \frac{1}{q}} + \frac{1}{N \eta}}}
\end{align}
holds \hp{\xi}{\nu}.
Moreover, we have the following estimate on the individual matrix elements of $G$. If instead of \eqref{lax bounds on f} 
$f$ satisfies
\begin{equation} \label{f upper bound}
0 \;\leq\; f \;\leq\; C_0 N^{1/2}\,,
\end{equation}
for some constant $C_0$, then the event
\begin{align}\label{Gij estimate}
\bigcap_{z \in D_L} \hBB{\max_{1 \leq i,j \leq N} \absb{G_{ij}(z) - \delta_{ij} m_{sc}(z)} \leq (\log N)^{40 \xi}  
\pbb{\frac{1}{q} + \sqrt{\frac{\im m_{sc}(z)}{N \eta}} + \frac{1}{N \eta}}}
\end{align}
holds \hp{\xi}{\nu}.

\end{theorem}

Next, we recall that the $N - 1$ first eigenvalues of $A$ are close the their classical locations predicted by the 
semicircle law. Let $n_{sc}(E) \deq \int_{-\infty}^E \varrho_{sc}(x) \, \dd x$ denote the integrated density of the 
local semicircle law. Denote by $\gamma_\alpha$ the classical location of the $\alpha$-th eigenvalue, defined through
\begin{equation} \label{def of gamma}
n_{sc}(\gamma_\alpha) \;=\; \frac{\alpha}{N} \for \alpha = 1, \dots, N\,.
\end{equation}
The following theorem compares the locations of the eigenvalues $\mu_1, \dots, \mu_{N - 1}$ to their classical locations 
$\gamma_1, \dots, \gamma_{N - 1}$.

\begin{theorem}[Eigenvalue locations \cite{EKYY}] \label{thm: eigenvalue locations}
Suppose that $A$ satisfies Definition \ref{definition of A}, and let $\phi$ be an exponent satisfying $0 < \phi \leq 
1/2$, and set $q = N^{\phi}$.
Then there is a constant $\nu > 0$
 -- depending on $\Sigma$ and the constants $C$ in \eqref{moment conditions}, \eqref{lower bound on d}, and \eqref{lower 
bound on f} -- as well as a constant $C > 0$
such that the following holds.

We have \hp{\xi}{\nu} that
\begin{equation} \label{main estimate on Q}
\sum_{\alpha = 1}^{N - 1} \abs{\mu_\alpha - \gamma_\alpha}^2 \;\leq\; (\log N)^{C \xi} \, \pb{N^{1 - 4 \phi} + N^{4/3 - 
8 \phi}}\,.
\end{equation}
Moreover, for all $\alpha = 1, \dots, N - 1$ we have \hp{\xi}{\nu} that
\begin{equation} \label{detailed estimate for large phi}
\abs{\mu_\alpha - \gamma_\alpha} \;\leq\; (\log N)^{C \xi} \pbb{N^{-2/3} \qB{\wh \alpha^{-1/3} + \indB{\wh \alpha \leq 
(\log N)^{C \xi} (1 + N^{1 - 3 \phi})}} + N^{2/3 - 4 \phi} \wh \alpha^{-2/3} + N^{-2 \phi}}\,,
\end{equation}
where we abbreviated $\wh \alpha \deq \min \h{\alpha, N - \alpha}$.
\end{theorem}

\begin{remark}
Under the assumption $\phi \geq 1/3$ the estimate \eqref{detailed estimate for large phi} simplifies to
\begin{equation} \label{rigidity for large phi}
\abs{\mu_\alpha - \gamma_\alpha} \;\leq\; (\log N)^{C \xi} \pB{N^{-2/3} \wh \alpha^{-1/3} + N^{-2 \phi}}\,,
\end{equation}
which holds \hp{\xi}{\nu}.
\end{remark}

Finally, we record two basic results from \cite{EKYY} for later reference.
From \cite{EKYY}, Lemmas 4.4 and 6.1, we get, \hp{\xi}{\nu},
\begin{equation} \label{bound on bulk eigenvalues}
\max_{1 \leq \alpha \leq N} \abs{\lambda_\alpha} \;\leq\; 2 + (\log N)^{C \xi} \pB{q^{-2} + N^{-2/3}}\,, \qquad
\max_{1 \leq \alpha \leq N - 1} \abs{\mu_\alpha} \;\leq\; 2 + (\log N)^{C \xi} \pB{q^{-2} + N^{-2/3}}\,.
\end{equation}
Moreover, from \cite{EKYY}, Theorem 6.2, we get, \hp{\xi}{\nu},
\begin{equation} \label{position of mu max}
\mu_N \;=\; f + \frac{1}{f} + o(1)\,.
\end{equation}
In particular, using \eqref{lower bound on f} we get, \hp{\xi}{\nu},
\begin{equation} \label{gap 1}
2 + \sigma \;\leq\; \mu_N \;\leq\; N^C\,,
\end{equation}
where $\sigma > 0$ is a constant spectral gap depending only on the constant $\epsilon_0$ from \eqref{lower bound on f}.

\section{Local ergodicity of the marginal Dyson Brownian motion} \label{section: DBM}

In Sections \ref{section: DBM} and \ref{sect: bulk universality} we give the proof of Theorem \ref{theorem: bulk 
universality}. Throughout Sections \ref{section: DBM} and \ref{sect: bulk universality} it is convenient to adopt a 
slightly different notation for the eigenvalues of $A$. In these two sections we shall consistently use $x_1 \leq \cdots 
\leq x_N$ to denote the ordered eigenvalues of $A$, instead of $\mu_1 \leq \cdots \leq \mu_N$ used in the rest of this 
paper.  We abbreviate the collection of eigenvalues by $\f x = (x_1, \dots, x_N)$.

The main tool in the proof of Theorem \ref{theorem: bulk universality} is the marginal Dyson Brownian motion, obtained 
from the usual Dyson Brownian motion of the eigenvalues $\f x$ by integrating out the largest eigenvalue $x_N$.  In this 
section we establish the local ergodicity of the marginal Dyson Brownian and derive an upper bound on its local 
relaxation time.

Let $A_0 = (a_{ij,0})_{ij}$ be a matrix satisfying Definition \ref{definition of A} with constants $q_0 \geq N^\phi$ and 
$f_0 \geq 1 + \epsilon_0$. Let $(B_{ij,t})_{ij}$ be a symmetric matrix of independent Brownian motions, whose 
off-diagonal entries have variance $t$ and diagonal entries variance $2t$. Let the matrix $A_t = (a_{ij,t})_{ij}$ 
satisfy the stochastic differential equation
\begin{align} \label{simple bm}
\dd a_{ij} \;=\; \frac{\dd B_{ij}}{\sqrt{N}} - \frac{1}{2} a_{ij} \, \dd t\,.
\end{align}
It is easy to check that the distribution of $A_t$ is equal to the distribution of
\begin{align} \label{law of Ht}
\me^{-t/2} A_0 + (1 - \me^{-t})^{1/2} V\,,
\end{align}
where $V$ is a GOE matrix independent of $A_0$.

Let $\rho$ be a constant satisfying $0 < \rho < 1$ to be chosen later. In the following we shall consider times $t$ in 
the interval $[t_0, \tau]$, where
\begin{equation*}
t_0 \;\deq\; N^{-\rho - 1}\,, \qquad \tau \;\deq\; N^{- \rho}\,.
\end{equation*}
One readily checks that, for any fixed $\rho$ as above, the matrix $A_t$ satisfies Definition \ref{definition of A}, 
with constants
\begin{equation*}
f_t \;=\; f (1 + O(N^{- \delta_0})) \;\geq\; 1 + \frac{\epsilon_0}{2}\,, \qquad
q_t \;\sim\; q_0 \;\geq\; N^\phi\,,
\end{equation*}
where all estimates are uniform for $t \in [t_0, \tau]$. Denoting by $x_{N,t}$ the largest eigenvalue of $A_t$, we get 
in particular from \eqref{gap 1} that
\begin{equation} \label{gap}
\P \pB{\exists \, t \in [t_0, \tau] \col x_{N,t} \notin [2 + \sigma, N^C]} \;\leq\; \me^{-\nu (\log N)^\xi}
\end{equation}
for some $\sigma > 0$ and $C > 0$.

From now on we shall never use the symbols $f_t$ and $q_t$ in their above sense. The only information we shall need 
about $x_N$ is \eqref{gap}. In this section we shall not use any information about $q_t$, and in Section \ref{sect: bulk 
universality} we shall only need that $q_t \geq c N^\phi$ uniformly in $t$. Throughout this section $f_t$ will denote 
the joint eigenvalue density evolved under the Dyson Brownian motion. (See Definition \ref{definition of f_t} below.)

It is well known that the eigenvalues $\f x_t$ of $A_t$ satisfy the stochastic differential equation (Dyson Brownian 
motion)
\begin{align} \label{DBM}
\dd x_i \;=\; \frac{\dd B_i}{\sqrt{N}} + \pbb{-\frac{1}{4} x_i + \frac{1}{2N} \sum_{j \neq i} \frac{1}{x_i - x_j}} \dd t 
\for i = 1, \dots, N\,,
\end{align}
where $B_1, \dots, B_N$ is a family of independent standard Brownian motions.

In order to describe the law of $V$, we define the equilibrium Hamiltonian
\begin{equation} \label{GOEH}
\cal H(\f x) \;\deq\; \sum_i \frac{1}{4} x_i^2 - \frac{1}{N} \sum_{i < j} \log \abs{x_i - x_j}
\end{equation}
and denote the associated probability measure by
\begin{equation} \label{GOE measure}
\mu^{(N)}(\dd \f x) \;\equiv\; \mu(\dd \f x) \;\deq\; \frac{1}{Z} \me^{-N \cal H(\f x)} \, \dd \f x\,,
\end{equation}
where $Z$ is a normalization. We shall always consider the restriction of $\mu$ to the domain
\begin{align*}
\Sigma_N \;\deq\; \h{\f x \col x_1 < \cdots < x_N}\,,
\end{align*}
i.e.\ a factor $\ind{\f x \in \Sigma_N}$ is understood in expressions like the right-hand side of \eqref{GOE measure}; 
we shall usually omit it. The law of the ordered eigenvalues of the GOE matrix $V$ is $\mu$.

Define the Dirichlet form $D_\mu$ and the associated generator $L$ through
\begin{equation} \label{DBMD}
D_\mu(f) \;=\; - \int f (L f) \, \dd \mu \;\deq\; \frac{1}{2N} \int \abs{\nabla f}^2 \, \dd \mu\,,
\end{equation}
where $f$ is a smooth function of compact support on $\Sigma_N$.
One may easily check that
\begin{align*}
L \;=\; \sum_i \frac{1}{2N} \partial_i^2 + \sum_i \pbb{-\frac{1}{4} x_i + \frac{1}{2N}\sum_{j \neq i} \frac{1}{x_i - 
x_j}} \partial_i\,,
\end{align*}
and that $L$ is the generator of the Dyson Brownian motion \eqref{DBM}. More precisely, the law of $\f x_t$ is given by 
$f_t(\f x) \, \mu(\dd \f x)$, where $f_t$ solves $\partial_t f_t = L f_t$ and $f_0(\f x) \mu(\dd \f x)$ is the law of 
$\f x_0$.

\begin{definition} \label{definition of f_t}
Let $f_t$ to denote the solution of $\partial_t f_t = L f_t$ satisfying $f_t |_{t = 0} = f_0$.  It is well known that 
this solution exists and is unique, and that $\Sigma_N$ is invariant under the Dyson Brownian motion, i.e.\ if $f_0$ is 
supported in $\Sigma_N$, so is $f_t$ for all $t \geq 0$. For a precise formulation of these statements and their proofs, 
see e.g.\ Appendices A and B in \cite{ESYY}. In Appendix \ref{appendix: DBM}, we present a new, simpler and more 
general, proof.
\end{definition}

\begin{theorem} \label{theorem: flow}
Fix $n \geq 1$ and let $\f m = (m_1, \dots, m_n) \in \N^n$ be an increasing family of indices. Let $G \col \R^n \to \R$ be a 
continuous function of compact support and set
\begin{align*}
\cal G_{i, \f m}(\f x) \;\deq\; G \pb{N(x_i - x_{i + m_1}), N(x_{i + m_1} - x_{i + m_2}), \dots, N(x_{i + m_{n - 1}} - 
x_{i + m_n})}\,.
\end{align*}
Let $\gamma_1, \dots, \gamma_{N - 1}$ denote the classical locations of the first $N - 1$ eigenvalues, as defined in 
\eqref{def of gamma}, and set
\begin{equation}
Q \;\deq\; \sup_{t \in [t_0, \tau]} \sum_{i = 1}^{N - 1} \int (x_i - \gamma_i)^2 f_t \, \dd \mu\,.
\end{equation}
Choose an $\epsilon > 0$. Then for any $\rho$ satisfying $0 < \rho < 1$ there exists a $\bar \tau \in [\tau/2, \tau]$ 
such that, for any $J \subset \{1, 2, \dots, N - m_n - 1\}$, we have
\begin{align} \label{local ergodicity of DBM}
\absbb{\int \frac{1}{\abs{J}} \sum_{i \in J} \cal G_{i, \f m} \, f_{\bar \tau} \, \dd \mu - \int \frac{1}{\abs{J}} 
\sum_{i \in J} \cal G_{i, \f m} \, \dd \mu^{(N - 1)}} \;\leq\; C N^\epsilon \sqrt{\frac{N^{1 + \rho} Q + 
N^{\rho}}{\abs{J}}}
\end{align}
for all $N \geq N_0(\rho)$.
Here $\mu^{(N - 1)}$ is the equilibrium measure of $(N - 1)$ eigenvalues (GOE).
\end{theorem}
Note that, by definition, the observables $\cal G_{i, \f m}$ in \eqref{local ergodicity of DBM} only depend on the 
eigenvalues $x_1, \dots, x_{N - 1}$.

The rest of this section is devoted to the proof of Theorem \ref{theorem: flow}. We begin by introducing a pseudo 
equilibrium measure. Abbreviate
\begin{align*}
R \;\deq\; \sqrt{\tau N^{-\epsilon}} \;=\; N^{-\rho/2 - \epsilon/2}
\end{align*}
and define
\begin{equation*}
W(\f x) \;\deq\; \sum_{i = 1}^N \frac{1}{2 R^2}(x_i - \gamma_i)^2\,.
\end{equation*}
Here we set $\gamma_N \deq  2 + \sigma$ for convenience, but one may easily check that the proof remains valid for any 
larger choice of $\gamma_N$.  Define the probability measure
\begin{align*}
\omega(\dd \f x) \;\deq\; \psi(\f x) \, \mu(\dd \f x) \where \psi(\f x) \;\deq\; \frac{Z}{\wt Z} \me^{- N W(\f 
x)}\,.
\end{align*}

Next, we consider marginal quantities obtained by integrating out the largest eigenvalue $x_N$. To that end we write
\begin{align*}
\f x \;=\; (\wh x,x_N) \,, \qquad \wh x \;=\; (x_1, \dots, x_{N - 1})
\end{align*}
and denote by $\wh \omega(\dd \wh x)$ the marginal measure of $\omega$ obtained by integrating out $x_N$. By a slight 
abuse of notation, we sometimes make use of functions $\mu$, $\omega$, and $\wh \omega$, defined as the densities (with 
respect to Lebesgue measure) of their respective measures. Thus,
\begin{align*}
\mu(\f x) \;=\; \frac{1}{Z} \, \me^{-N \cal H(\f x)}\,, \qquad \omega(\f x) \;=\; \frac{1}{\wt Z} \, \me^{- N \cal H(\f 
x) - N W(\f x)}\,, \qquad \wh \omega(\wh x) \;=\; \int_{x_{N - 1}}^\infty \, \omega(\wh x, x_N) \, \dd x_N\,.
\end{align*}

For any function $h(\f x)$ we introduce the conditional expectation
\begin{align*}
\avg h(\wh x) \;\deq\; \E^\omega [h | \wh x] \;=\; \frac{\int_{x_{N - 1}}^\infty h(\wh x, x_N) \, \omega(\wh x, x_N) \, 
\dd x_N}{\wh \omega(\wh x)}\,.
\end{align*}

Throughout the following, we write $g_t \deq f_t/\psi$.
In order to avoid pathological behaviour of the extreme eigenvalues, we introduce cutoffs. Let $\sigma$ be the spectral 
gap from \eqref{gap}, and choose $\theta_1, \theta_2, \theta_3 \in [0,1]$ to be smooth functions that satisfy
\begin{align*}
\theta_1(x_1) &\;=\;
\begin{cases}
0 & \text{if } x_1 \leq -4
\\
1 & \text{if } x_1 \geq -3
\end{cases}\,,
\\
\theta_2(x_{N - 1}) &\;=\;
\begin{cases}
1 & \text{if } x_{N - 1} \leq 2 + \frac{\sigma}{5}
\\
0 & \text{if } x_{N - 1} \geq 2 + \frac{2 \sigma}{5}
\end{cases}\,,
\\
\theta_3(x_N) &\;=\;
\begin{cases}
0 & \text{if } x_N \leq 2 + \frac{3 \sigma}{5}
\\
1 & \text{if } x_N \geq 2 + \frac{4 \sigma}{5}
\end{cases}\,.
\end{align*}
Define $\theta \equiv \theta(x_1, x_{N - 1}, x_N) = \theta_1(x_1) \, \theta_2(x_{N - 1}) \, \theta_3(x_N)$. One easily 
finds that
\begin{align} \label{nabla theta}
\frac{\abs{\nabla \theta}^2}{\theta} \;\leq\; C \ind{-4 \leq x_1 \leq -3} + C \indbb{\frac{\sigma}{2} \leq x_{N - 1} - 2 
\leq \frac{2 \sigma}{5}} + C \indbb{\frac{3 \sigma}{5} \leq x_N - 2 \leq \frac{4 \sigma}{5}}\,,
\end{align}
where the left-hand side is understood to vanish outside the support of $\theta$.

Define the density
\begin{align*}
h_t\;\deq\; \frac{1}{\wh Z_t} \, \theta g_t\,, \qquad \wh Z_t \;\deq\; \int \theta g_t \, \dd \omega\,.
\end{align*}

If $\nu$ is a probability measure and $q$ a density such that 
$q \nu$ is also a probability measure, we define the entropy
\begin{align*}
S_\nu(q) \;\deq\; \int q \log q \, \dd \nu\,.
\end{align*}

The following result is our main tool for controlling the local ergodicity of the marginal Dyson Brownian motion.

\begin{proposition} \label{proposition: dtS}
Suppose that
\begin{align} \label{SA1}
&\text{(i)} \quad S_\mu(f_{t_0}) \;\leq\; N^C\,,
\\ \label{SA2}
&\text{(ii)} \quad \sup_{t \in [t_0,\tau]} \int \qbb{\ind{x_1 \leq -3} + \indbb{x_{N - 1} \geq 2 + \frac{\sigma}{5}} + 
\indbb{x_N \leq 2 + \frac{4 \sigma}{5}}} \, f_t \, \dd \mu \;\leq\; \me^{- \nu(\log N)^\xi}\,,
\\ \label{SA3}
&\text{(iii)} \quad \sup_{t \in [t_0, \tau]} \sup_{\wh x \in \Sigma_{N - 1}} (\theta_1 \theta_2)(\wh x) \absb{\log 
\avg{\theta g_t}(\wh x)}^2 \;\leq\; N^C\,.
\end{align}
Then for $t \in [t_0, \tau]$ we have
\begin{align} \label{dtS}
\partial_t S_{\wh \omega}(\avg h) \;\leq\;
- D_{\wh \omega}\pb{\sqrt{\avg h}} +
S_{\wh \omega}(\avg h) \, \me^{-c (\log N)^\xi} +
C N Q R^{-4} + C\,.
\end{align}
\end{proposition}
\begin{proof}
First we note that
\begin{align} \label{Z = 1}
\wh Z_t \;=\; \int \theta f_t \, \dd \mu \;=\; 1 - O \pb{\me^{- \nu (\log N)^\xi}}
\end{align}
uniformly for $t \in [t_0, \tau]$, by \eqref{SA2}.
Dropping the time index to avoid cluttering the notation, we find
\begin{align*}
\partial_t S_{\wh \omega}(\avg h) \;=\; \partial_t \int \frac{\avg{\theta g}}{\wh Z} \log \avg{\theta g} \, \dd \wh 
\omega
- \partial_t \log \wh Z \;=\; \frac{1}{\wh Z} \partial_t \int \theta g \, \log \avg{\theta g} \, \dd \omega - \pb{1 + 
\log \wh Z + S_{\wh \omega}(\avg h)} \, \partial_t \log \wh Z\,.
\end{align*}
We find that
\begin{align*}
\partial_t \wh Z \;=\; \int \theta (Lf) \, \dd \mu \;=\; - \frac{1}{2N} \int \nabla \theta \cdot \nabla f \, \dd \mu 
\;\leq\;
\pbb{\frac{1}{N} \int \abs{\nabla \theta}^2 \, f \, \dd \mu}^{1/2} D_\mu(\sqrt{f})^{1/2}\,.
\end{align*}
Bounding the Dirichlet form in terms of the entropy (see e.g.\ \cite{Enotes}, Theorem 3.2), we find that
\begin{align} \label{bound on D mu}
D_\mu(\sqrt{f_t}) \;\leq\; \frac{2}{t} S_\mu(f_{t_0}) \;\leq\; N^C\,,
\end{align}
by \eqref{SA1}.
Using \eqref{nabla theta} we therefore find
\begin{align} \label{Sh 1}
\partial_t \wh Z \;\leq\; N^C \me^{-c (\log N)^\xi}\,.
\end{align}
Thus we have
\begin{align} \label{Sh 2}
\partial_t S_{\wh \omega}(\avg h)  \;\leq\; 2 \partial_t \int \theta g \, \log \avg{\theta g} \, \dd \omega + \pb{1 + 
S_{\wh \omega}(\avg h)} N^C \me^{-c (\log N)^\xi}\,.
\end{align}
We therefore need to estimate
\begin{align} \label{Sh 3}
\partial_t \int \theta g \, \log \avg{\theta g} \, \dd \omega \;=\; \int \theta (L f) \log \avg{\theta g} \, \dd \mu + 
\int \avg{\theta g} \, \frac{\partial_t \avg{\theta g}}{\avg{\theta g}} \, \dd \wh \omega\,.
\end{align}
The second term of \eqref{Sh 3} is given by
\begin{align*}
\int \partial_t \avg{\theta g} \, \dd \wh \omega \;=\; \int \theta (Lf) \, \dd \mu \;=\; \partial_t \wh Z\,.
\end{align*}
Therefore \eqref{Sh 2} yields
\begin{align} \label{Sh 4}
\partial_t S_{\wh \omega}(\avg h) \;\leq\; 2 \int \theta (L f) \log \avg{\theta g} \, \dd \mu + (1 + S_{\wh \omega}(\avg 
h)) N^C \me^{-c (\log N)^\xi}\,.
\end{align}
The first term of \eqref{Sh 4} is given by
\begin{align} \label{Sh 5}
-\frac{1}{N} \int \nabla f \cdot \nabla \pb{\theta \log \avg{\theta g}} \, \dd \mu \;=\; - \frac{1}{N} \int \nabla 
(\theta f) \cdot \nabla \pb{\log \avg{\theta g}} \, \dd \mu + \cal E_1 + \cal E_2\,,
\end{align}
where we defined
\begin{align*}
\cal E_1 \;\deq\; \frac{1}{N} \int \nabla \theta \cdot \nabla (\log \avg{\theta g}) \, f \, \dd \mu\,, \qquad \cal E_2 
\;\deq\; -\frac{1}{N} \int \nabla \theta \cdot \nabla f \, \log \avg{\theta g} \, \dd \mu\,.
\end{align*}

Next, we estimate the error terms $\cal E_1$ and $\cal E_2$. Using \eqref{nabla theta} we get
\begin{align*}
\cal E_1 \;=\; \frac{1}{N} \int \frac{\nabla \theta}{\sqrt{\theta}} \cdot \nabla (\log \avg{\theta g}) \, \sqrt{\theta} 
f \, \dd \mu
\;\leq\; \pbb{\int \frac{\abs{\nabla \theta}^2}{\theta} f \, \dd \mu}^{1/2} \pbb{\int \frac{\abs{\nabla \avg{\theta 
g}}^2}{\avg{\theta g}^2} \, \theta f \, \dd \mu}^{1/2}
\\
\leq\; \me^{- \nu(\log N)^\xi}
\pbb{\int \frac{\abs{\nabla \avg{\theta g}}^2}{\avg{\theta g}^2} \, \avg{\theta g} \, \dd \wh \omega}^{1/2} \;\leq\; 
\me^{-c (\log N)^\xi}  +  \me^{-c (\log N)^\xi} D_{\wh \omega}\pb{\sqrt{\avg h}}\,,
\end{align*}
where we used \eqref{Z = 1}. Similarly, we find
\begin{align*}
\cal E_2 \;\leq\; \pbb{\int \abs{\nabla \theta}^2 \, \absb{\log \avg{\theta g}}^2 \, f \, \dd \mu}^{1/2} \pbb{\int 
\frac{\abs{\nabla f}^2}{f} \, \dd \mu}^{1/2}\,.
\end{align*}
Using \eqref{nabla theta}, \eqref{SA3}, and \eqref{bound on D mu} we therefore get
\begin{align*}
\cal E_2 \;\leq\; N^C \pbb{\int \frac{\abs{\nabla \theta}^2}{\theta} \, \theta \absb{\log \avg{\theta g}}^2 f \, \dd 
\mu}^{1/2} \;\leq\; N^C \me^{-c (\log N)^\xi}\,.
\end{align*}

Having dealt with the error terms $\cal E_1$ and $\cal E_2$, we compute the first term on the right-hand side of 
\eqref{Sh 5},
\begin{align} \label{Sh 6}
- \frac{1}{N} \int \nabla (\theta f) \cdot \nabla \pb{\log \avg{\theta g}} \, \dd \mu
\;=\;
- \frac{1}{N} \int \nabla_{\wh x} (\theta g) \cdot \nabla_{\wh x} \pb{\log \avg{\theta g}} \, \psi \, \dd \mu
- \frac{1}{N} \int \nabla_{\wh x} (\log \psi) \cdot \nabla_{\wh x} \pb{\log \avg{\theta g}} \, \theta g \psi \, \dd 
\mu\,.
\end{align}
The second term of \eqref{Sh 6} is bounded by
\begin{align*}
\frac{\eta^{-1}}{N} \int \abs{\nabla_{\wh x} \log \psi}^2 \, f \, \dd \mu + \frac{\eta}{N} \int \frac{\abs{\nabla \, 
\avg{\theta g}}^2}{\avg{\theta g}^2} \, \avg{\theta g} \, \dd \wh \omega &\;\leq\;
\eta^{-1}N \int \frac{1}{R^4} \sum_{i = 1}^{N - 1} (x_i - \gamma_i)^2\, f \, \dd \mu + 4 \eta D_{\wh \omega}(\sqrt{\avg 
h})
\\
&\;\leq\; \frac{N Q}{\eta R^4} + 4 \eta D_{\wh \omega}(\sqrt{\avg h})\,,
\end{align*}
where $\eta > 0$.

The first term of \eqref{Sh 6} is equal to
\begin{align*}
- \frac{1}{N} \int \avg{\nabla_{\wh x} (\theta g)} \cdot \nabla_{\wh x} \pb{\log \avg{\theta g}} \, \dd \wh \omega\,.
\end{align*}
A simple calculation shows that
\begin{align*}
\avg{\nabla_{\wh x} (\theta g)} \;=\; \nabla_{\wh x} \avg{\theta g} - \avgb{\theta g \nabla_{\wh x} \log \omega} + 
\avg{\theta g} \, \avg{\nabla_{\wh x} \log \omega}\,,
\end{align*}
so that the first term of \eqref{Sh 6} becomes
\begin{align*}
- \frac{1}{N} \int \nabla_{\wh x} \avg{\theta g} \cdot \nabla_{\wh x} \pb{\log \avg{\theta g}} \, \dd \wh \omega
+ \frac{1}{N} \int \pB{\avgb{\theta g \nabla_{\wh x} \log \omega} - \avg{\theta g} \, \avg{\nabla_{\wh x} \log \omega}}
\cdot \nabla_{\wh x} \pb{\log \avg{\theta g}} \, \dd \wh \omega
\\
\;\leq\;
- 4 (1 - \eta) D_{\wh \omega}\pb{\sqrt{\avg h}}
+ \frac{1}{N \eta} \int \frac{\absb{\avgb{\theta g \nabla_{\wh x} \log \omega} - \avg{\theta g} \, \avg{\nabla_{\wh x} 
\log \omega}}^2}{\avg{\theta g}} \, \dd \wh \omega
\,.
\end{align*}
Using the Cauchy-Schwarz inequality $\avg{ab}^2 \leq \avg{a^2} \, \avg{b^2}$ we find that the second term 
is bounded by
\begin{align*}
\frac{1}{N \eta} \int \frac{\absB{\avgb{\theta g \pb{\nabla_{\wh x} \log \omega- \avg{\nabla_{\wh x} \log 
\omega}}}}^2}{\avg{\theta g}} \, \dd \wh \omega
&\;\leq\;
\frac{1}{N \eta} \int \avgB{\theta g \absb{\nabla_{\wh x} \log \omega- \avg{\nabla_{\wh x} \log \omega}}^2} \, \dd \wh 
\omega
\\
&\;=\;
\frac{1}{N \eta} \int \absb{\nabla_{\wh x} \log \omega- \avg{\nabla_{\wh x} \log \omega}}^2 \, \theta f \, \dd \mu
\\
&\;=\; \frac{1}{N \eta} \int \sum_{i = 1}^{N - 1} \pbb{\frac{1}{x_N - x_i} - \avgbb{\frac{1}{x_N - x_i}}}^2\, \theta f 
\, \dd \mu\,.
\end{align*}
Thus, we have to estimate
\begin{align*}
\cal E_3 \;\deq\; \frac{1}{N \eta} \int \sum_{i = 1}^{N - 1} \pbb{\frac{1}{x_N - x_i}}^2\, \theta f \, \dd \mu
\,, \qquad
\cal E_4 \;\deq\; \frac{1}{N \eta} \int \sum_{i = 1}^{N - 1} \avgbb{\frac{1}{x_N - x_i}}^2\, \theta f \, \dd \mu\,.
\end{align*}
Since $x_N - x_i \geq \sigma/5$ on the support of $\theta f \mu$, one easily gets from \eqref{bound on bulk eigenvalues} 
that
\begin{align*}
\cal E_3 \;\leq\; \frac{C}{\eta}\,.
\end{align*}
In order to estimate $\cal E_4$, we write
\begin{align*}
\avgbb{\frac{1}{x_N - x_i}} \;=\; \pbb{\frac{\int \dd x_N \, (x_N - x_i) \, w_i(x_N)}{\int \dd x_N \, w_i(x_N)}}^{-1}\,,
\end{align*}
where
\begin{align*}
w_i(x_N) \;\deq\; \ind{x_N \geq x_{N - 1}} \, \me^{-\frac{N}{4} x_N^2 - \frac{N}{2R^2} (x_N - \gamma_N)^2} \prod_{j \neq 
i,N} (x_N - x_j)\,.
\end{align*}
We now claim that on the support of $\theta$, in particular for $-4 \leq x_1 < x_{N - 1} \leq 2 + 2 \sigma / 5$, we have
\begin{align} \label{rough bound on w}
\frac{\int \dd x_N \, (x_N - x_i) \, w_i(x_N)}{\int \dd x_N \, w_i(x_N)} \;\geq\; c \, \gamma_N\,,
\end{align}
uniformly for $\wh x \in \Sigma_{N - 1}$. Indeed, writing $\wt \gamma_N \deq \gamma_N(1 + R^{-2})$, we have on the support of $\theta$
\begin{align*}
\frac{\int \dd x_N \, (x_N - x_i) \, w_i(x_N)}{\int \dd x_N \, w_i(x_N)} \;\geq\; \wt \gamma_N/2 + \frac{\int \dd x_N \,
(x_N - \wt \gamma_N) \, w_i(x_N)}{\int \dd x_N \, w_i(x_N)}\,.
\end{align*}
Moreover, the second term is nonnegative: 
\begin{align*}
\int \dd x_N \, (x_N - \wt \gamma_N) \, w_i(x_N) &\;=\; - C_N(\wh x) \int_{x_{N - 1}}^\infty \dd
x_N \, \pbb{\frac{\partial}{\partial
x_N} \me^{- \frac{N}{R^2} (x_N - \wt \gamma_N)^2}} \prod_{j \neq i,N}(x_N - x_j)
\\
&\;=\; C_N(\wh x) \, \me^{- \frac{N}{R^2} (x_{N - 1} - \wt \gamma_N)^2} \prod_{j \neq i,N}(x_{N - 1} - x_j)
\\
&\qquad {}+{}
C_N(\wh x) \int_{x_{N - 1}}^\infty \dd
x_N \, \me^{- \frac{N}{R^2} (x_N - \wt \gamma_N)^2} \sum_{k \neq i,j} \prod_{j \neq i,k,N}(x_N - x_j)
\\
&\;\geq\; 0\,,
\end{align*}
where $C_N(\wh x)$ is nonnegative.
This proves \eqref{rough bound on w}.
Using \eqref{rough bound on w} we get
\begin{align*}
\cal E_4 \;\leq\; \frac{C}{\eta} \int \gamma_N^{-2} \, f \, \dd \mu \;=\; \frac{C}{\eta}\,.
\end{align*}

Summarizing, we have proved that
\begin{align*}
\partial_t S_{\wh \omega}(\avg h) \;\leq\;
- \pb{4 - 8 \eta -\me^{-c (\log N)^\xi}} D_{\wh \omega}\pb{\sqrt{\avg h}}
+ \pb{1 + S_{\wh \omega}(\avg h)} \me^{-c (\log N)^\xi} +
\frac{N Q}{\eta R^4} + \frac{C}{\eta}\,.
\end{align*}
Choosing $\eta$ small enough completes the proof.
\end{proof}

Next, we derive a logarithmic convexity bound for the marginal measure $\wh \omega$.

\begin{lemma} \label{lemma: log convex}
We have that
\begin{align*}
\wh \omega(\wh x) \;=\; \frac{1}{\wt Z} \, \me^{- N \wh {\cal H}(\wh x)}\,,
\end{align*}
where
\begin{align} \label{property of V}
\wh {\cal H}(\wh x) \;=\; - \frac{1}{N} \sum_{i < j < N} \log \abs{x_i - x_j} + V(\wh x)\,,
\end{align}
and $\nabla^2 V(\wh x) \geq R^{-2}$.
\end{lemma}
\begin{proof}
Write $\cal H(\wh x, x_N) = \cal H'(\wh x) + \cal H''(\wh x, x_N)$ where
\begin{align*}
\cal H'(\wh x) \;\deq\; -\frac{1}{N} \sum_{i < j < N} \log \abs{x_i - x_j} \,,
\qquad \cal H''(\wh x, x_N) \;\deq\; -\frac{1}{N} \sum_{i < N} \log \abs{x_N - x_i} + \sum_i \frac{1}{2R^2}(x_i -
\gamma_i)^2\,.
\end{align*}
By definition, we have
\begin{align*}
\wh \omega(\wh x) \;=\; \frac{1}{\wt Z} \me^{-N \cal H'(\wh x)} \, \int_{x_{N - 1}}^\infty \me^{- N \cal H''(\wh x, 
x_N)} \, \dd x_N\,.
\end{align*}

The main tool in our proof is the Brascamp-Lieb inequality \cite{BrascampLieb}. In order to apply it, we need to extend
the integration over $x_N$ to $\R$ and replace the singular logarithm with a $C^2$-function. To that end, we introduce the approximation parameter $\delta > 0$ and define, for
$\wh x \in \Sigma_{N - 1}$,
\begin{align*}
V_\delta(\wh x) \;\deq\; -\frac{1}{N} \log \int \exp \qBB{\sum_{i < N} \log_\delta(x_N - x_i) - \frac{N}{2 R^2} \sum_i
(x_i - \gamma_i)^2} \, \dd x_N\,,
\end{align*}
where we defined
\begin{align*}
\log_\delta(x) \;\deq\; \ind{x \geq \delta} \log x + \ind{x < \delta} \pbb{\log \delta + \frac{x - \delta}{\delta} -
\frac{1}{2 \delta^2} (x - \delta)^2}\,.
\end{align*}
It is easy to check that $\log_\delta \in C^2(\R)$, is concave, and satisfies
\begin{align*}
\lim_{\delta \to 0} \log_\delta(x) \;=\;
\begin{cases}
\log x &\text{if } x > 0
\\
-\infty &\text{if } x \leq 0\,.
\end{cases}
\end{align*}
Thus we find that $V_\delta \in C^2(\Sigma_{N - 1})$ and that we have the pointwise convergence, for all $\wh x \in
\Sigma_{N - 1}$,
\begin{align*}
\lim_{\delta \to 0} V_\delta(\wh x) \;=\; V(\wh x) \;\deq\; - \frac{1}{N} \log \int_{x_{N - 1}}^\infty \me^{- N \cal 
H''(\wh x, x_N)} \, \dd x_N\,,
\end{align*}
where $V \in C^2(\Sigma_{N - 1})$ satisfies \eqref{property of V}.

Next, we claim that if $\varphi = \varphi(x,y)$ satisfies $\nabla^2 \varphi(x,y) \geq K$
then $\psi(x)$, defined
by
\begin{align*}
\me^{- \psi(x)} \;\deq\; \int \me^{- \varphi(x,y)} \, \dd y\,,
\end{align*}
satisfies $\nabla^2 \psi(x) \geq K$. In order to prove the claim, we use subscripts to denote partial derivatives and 
recall the Brascamp-Lieb inequality for log-concave functions (Equation 4.7 in \cite{BrascampLieb})
\begin{align*}
\psi_{xx} \;\geq\; \frac{\int \pb{\varphi_{xx} - \varphi_{xy} \varphi_{yy}^{-1} \varphi_{yx}} \, \me^{- \varphi} \, \dd
y}{\int \me^{-\varphi} \, \dd y}\,.
\end{align*}
Then the claim follows from
\begin{equation*}
\begin{pmatrix}
\varphi_{xx} & \varphi_{xy}
\\
\varphi_{yz} & \varphi_{yy}
\end{pmatrix}^{-1}
\;\leq\; \frac{1}{K}
\qquad \Longrightarrow \qquad
\pb{\varphi_{xx} - \varphi_{xy} \varphi_{yy}^{-1} \varphi_{yx}} \;\geq\; K\,.
\end{equation*}


Using this claim, we find that $\nabla^2 V_\delta(\wh x) \geq R^{-2}$ for all $\wh x \in \Sigma_{N - 1}$. In order to 
prove that $\nabla^2 V(\wh x) \geq R^{-2}$ -- and hence complete the proof -- it suffices to consider directional 
derivatives and prove the following claim.
If $(\zeta_\delta)_{\delta > 0}$ is a family of functions on a neighbourhood $U$ that converges pointwise to a 
$C^2$-function $\zeta$ as $\delta \to 0$, and if $\zeta_\delta''(x) \geq
K$ for all $\delta > 0$ and $x \in U$, then $\zeta''(x) \geq K$ for all $x \in U$. Indeed, taking $\delta \to 0$ in
\begin{align*}
\zeta_\delta(x+h) + \zeta_\delta(x - h) - 2 \zeta_\delta(x) \;=\; \int_0^h \pb{\zeta_\delta''(x + \xi) +
\zeta_\delta''(x - \xi)}(h - \xi) \, \dd \xi \;\geq\; K h^2
\end{align*}
yields $\pb{\zeta(x+h) + \zeta(x - h) - 2 \zeta(x)} h^{-2} \geq K$, from which the claim follows by taking the limit $h 
\to 0$.
\end{proof}

As a first consequence of Lemma \ref{lemma: log convex}, we derive an estimate on the expectation of observables 
depending only on eigenvalue differences.

\begin{proposition} \label{proposition: q omega}
Let $q \in L^\infty(\dd \wh \omega)$ be probability density. Then for any $J \subset \{1, 2, \dots, N - m_n - 1\}$ and 
any $t > 0$ we have
\begin{align*}
\absbb{\int \frac{1}{\abs{J}} \sum_{i \in J} \cal G_{i, \f m} \, q \, \dd \wh \omega - \int \frac{1}{\abs{J}} \sum_{i 
\in J} \cal G_{i, \f m} \, \dd \wh \omega} \;\leq\; C \sqrt{\frac{D_{\wh \omega}(\sqrt{q}) \, t}{\abs{J}}} + C 
\sqrt{S_{\wh \omega}(q)} \, \me^{- c t / R^2}\,.
\end{align*}
\end{proposition}
\begin{proof}
Using Lemma \ref{lemma: log convex}, the proof of Theorem 4.3 in \cite{ESYY} applies with merely cosmetic changes.
\end{proof}

Another, standard, consequence of Lemma \ref{lemma: log convex} is the logarithmic Sobolev inequality
\begin{align} \label{log Sobolev}
S_{\wh \omega}(q) \;\leq\; C R^2 D_{\wh \omega}(\sqrt{q})\,.
\end{align}
Using \eqref{log Sobolev} and Proposition \ref{proposition: dtS}, we get the following estimate on the Dirichlet form.

\begin{proposition} \label{proposition: D leq}
Under the assumptions of Proposition \ref{proposition: dtS}, there exists a $\bar \tau \in [\tau/2, \tau]$ such that
\begin{align*}
S_{\wh \omega}(\avg{h_{\bar \tau}}) \;\leq\; CN R^{-2} Q + C R^2\,,
\qquad
D_{\wh \omega}(\sqrt{\avg{h_{\bar \tau}}}) \;\leq\; C N R^{-4} Q + C\,.
\end{align*}
\end{proposition}

\begin{proof}
Combining \eqref{log Sobolev} with \eqref{dtS} yields
\begin{align}
\partial_t S_{\wh \omega}(\avg{h_t}) \;\leq\; - C R^{-2} S_{\wh \omega}(\avg{h_t})
+
C N Q R^{-4} + C\,,
\end{align}
which we integrate from $t_0$ to $t$ to get
\begin{align} \label{S S0}
S_{\wh \omega}(\avg{h_t}) \;\leq\; \me^{-C R^{-2} (t - t_0)} S_{\wh \omega}(\avg{h_{t_0}}) + CN Q R^{-2} + C R^2\,.
\end{align}
Moreover, \eqref{Z = 1} yields
\begin{multline*}
S_{\wh \omega}(\avg{h_{t_0}}) \;\leq\; C S_{\wh \omega}(\avg{g_{t_0}}) + \me^{- \nu (\log N)^\xi}
\;\leq\; C S_{\omega}(g_{t_0}) + \me^{- \nu (\log N)^\xi}
\\
=\; C S_\mu(f_{t_0}) - C \int \log \psi \, f_{t_0} \, \dd \mu + \me^{- \nu (\log N)^\xi}\,,
\end{multline*}
where the second inequality follows from the fact that taking marginals reduces the relative entropy; see the proof of 
Lemma \ref{lemma: initial entropy} below for more details. Thus we get
\begin{align*}
S_{\wh \omega}(\avg{h_{t_0}}) \;\leq\; N^C + N R^{-2} Q \;\leq\; N^C\,.
\end{align*}
Thus \eqref{S S0} yields
\begin{align}
S_{\wh \omega}(\avg{h_t}) \;\leq\; N^C \me^{-C R^{-2} (t - t_0)} + CN R^{-2} Q + C R^2
\end{align}
for $t \in [t_0, \tau]$.
Integrating \eqref{dtS} from $\tau/2$ to $\tau$ therefore gives
\begin{align*}
\frac{2}{\tau} \int_{\tau/2}^\tau D_{\wh \omega}(\sqrt{\avg{h_t}}) \, \dd t \;\leq\; C N R^{-4} Q + C\,,
\end{align*}
and the claim follows.
\end{proof}

We may finally complete the proof of Theorem \ref{theorem: flow}.

\begin{proof}[Proof of Theorem \ref{theorem: flow}]
The assumptions of Proposition \ref{proposition: dtS} are verified in Subsection \ref{subsect: verifying dtS} below.  
Hence Propositions \ref{proposition: q omega} and \ref{proposition: D leq} yield
\begin{align*}
\absbb{\int \frac{1}{\abs{J}} \sum_{i \in J} \cal G_{i, \f m} \, h_{\bar \tau} \, \dd \omega - \int \frac{1}{\abs{J}} 
\sum_{i \in J} \cal G_{i, \f m} \, \dd \omega} \;\leq\; C N^\epsilon \sqrt{\frac{N^{1 + \rho} Q}{\abs{J}}} + C 
\sqrt{\frac{N^{- 2 \phi - \rho}}{\abs{J}}}\,.
\end{align*}
Using \eqref{Z = 1} and \eqref{SA2} we get
\begin{align} \label{ft mu omega}
\absbb{\int \frac{1}{\abs{J}} \sum_{i \in J} \cal G_{i, \f m} \, f_{\bar \tau} \, \dd \mu - \int \frac{1}{\abs{J}} 
\sum_{i \in J} \cal G_{i, \f m} \, \dd \omega} \;\leq\; C N^\epsilon \sqrt{\frac{N^{1 + \rho} Q}{\abs{J}}} + C 
\sqrt{\frac{N^{- 2 \phi - \rho}}{\abs{J}}}\,.
\end{align}

In order to compare the measures $\wh \omega$ and $\mu^{(N - 1)}$, we define the density
\begin{align*}
q(\f x) \;\deq\; \frac{1}{Z'} \exp \hBB{\sum_{i < N} \frac{1}{4} x_i^2 + \sum_{i < N} \frac{N}{2 R^2} (x_i - 
\gamma_i)^2 - \sum_{i < N} \log \abs{x_N - x_i}}\,,
\end{align*}
where $Z'$ is a normalization chosen so that $\theta q \, \dd \omega$ is a probability measure. It is easy to see that
\begin{align*}
q \, \dd \omega \;=\; \dd \mu^{(N - 1)} \otimes \dd g\,,
\end{align*}
where $\dd g = C \me^{- \frac{N}{4} x_N^2 - \frac{N}{2R^2} (x_N - \gamma_N)^2} \dd x_N$ is a Gaussian measure.  
Similarly to Proposition \ref{proposition: q omega}, we have
\begin{align*}
\absbb{\int \frac{1}{\abs{J}} \sum_{i \in J} \cal G_{i, \f m} \, \theta q \, \dd \omega - \int \frac{1}{\abs{J}} \sum_{i 
\in J} \cal G_{i, \f m} \, \dd \omega} \;\leq\; C \sqrt{\frac{D_{\omega}(\sqrt{\theta q}) \tau}{\abs{J}}} + C 
\sqrt{S_{\omega}(\theta q)} \, \me^{- c \tau / R^2}\,.
\end{align*}
Thus we have to estimate
\begin{align*}
D_\omega(\sqrt{\theta q}) &\;\leq\; \frac{C}{N} \int \abs{\nabla \log q}^2 \, \theta q \, \dd \omega + \frac{C}{N} \int 
\frac{\abs{\nabla \theta}^2}{\theta} \, q \, \dd \omega
\\
&\;\leq\; \frac{C}{N} \sum_{i < N} \int \pbb{\frac{1}{4} x_i^2 + \frac{N^2}{R^4} (x_i - \gamma_i)^2 + \frac{1}{(x_N - 
x_i)^2}} \, \theta q \, \dd \omega + \frac{1}{N}
\\
&\;\leq\; C + N R^{-4} \int \sum_{i < N} (x_i - \gamma_i)^2 \, \dd \mu^{(N - 1)}
\end{align*}
where the second inequality follows from standard large deviation results for GOE. Since $\int \sum_{i < N} (x_i - 
\gamma_i)^2 \, \dd \mu^{(N - 1)} \leq C N^{-1 + \epsilon'}$ for arbitrary $\epsilon'$ is known to hold for GOE (see 
\cite{EYYrigidity} where this is proved for more general Wigner matrices), we find
\begin{align*}
\absbb{\int \frac{1}{\abs{J}} \sum_{i \in J} \cal G_{i, \f m} \, \theta q \, \dd \omega - \int \frac{1}{\abs{J}} \sum_{i 
\in J} \cal G_{i, \f m} \, \dd \omega} \;\leq\; C \sqrt{\frac{N^{- \rho}}{\abs{J}}} + C \sqrt{\frac{N^{\rho + 2 
\epsilon + \epsilon'}}{\abs{J}}}\,.
\end{align*}
The cutoff $\theta$ can be easily removed using the standard properties of $\dd \mu^{(N - 1)}$. Choosing $\epsilon' = 
\epsilon$, replacing $\epsilon$ with $\epsilon / 2$, and recalling \eqref{ft mu omega} completes the proof.
\end{proof}

\subsection{Verifying the assumptions of Proposition \ref{proposition: dtS}} \label{subsect: verifying dtS}
The estimate \eqref{SA1} is an immediate consequence of the following lemma.

\begin{lemma} \label{lemma: initial entropy}
Let the entries of $A_0$ have the distribution $\zeta_0$.
Then for any $t > 0$ we have
\begin{align*}
S_\mu(f_t) \;\leq\; N^2 (N m_2(\zeta_0) - \log \pb{1 - \me^{-t})}\,,
\end{align*}
where $m_2(\zeta_0)$ is the second moment of $\zeta_0$.
\end{lemma}
\begin{proof}
Recall that the relative entropy is defined, for $\nu \ll \mu$, as $S(\nu | \mu) \;\deq\; \int \log \frac{\dd \nu}{\dd 
\mu} \, \dd \nu$. If $\wh \nu$ and $\wh \mu$ are marginals of $\nu$ and $\mu$ with respect to the same variable, it is 
easy to check that $S(\wh \nu | \wh \mu) \leq S(\nu | \mu)$. Therefore
\begin{align*}
S_\mu(f_t) \;=\; S(f_t \mu | \mu) \;\leq\; S(A_t | V) \;=\; N^2 S(\zeta_t | g_{2/N})\,,
\end{align*}
where $\zeta_t$ denotes the law of the off-diagonal entries of $A_t$, and $g_\lambda$ is a standard Gaussian with 
variance $\lambda$ (the diagonal entries are dealt with similarly).  Setting $\gamma = 1 - \me^{-t}$, we find from 
\eqref{law of Ht} that $\zeta_t$ has probability density $\varrho_\gamma * g_{2 \gamma/N}$, where $\varrho_\gamma$ is 
the probability density of $(1 - \gamma)^{1/2} \zeta_0$.  Therefore Jensen's inequality yields
\begin{align*}
S(\zeta_t | g_{2/N}) \;=\; S \pbb{\int \dd y \, \varrho_\gamma(y) \, g_{2 \gamma/N}(\cdot - y) \,\bigg\vert\, g_{2/N}} 
\;\leq\; \int \dd y \, \varrho_\gamma(y) S\pb{g_{2 \gamma/N}(\cdot - y) | g_{2/N}}\,.
\end{align*}
By explicit computation one finds
\begin{align*}
S\pb{g_{2 \gamma/N}(\cdot - y) | g_{2 /N}} \;=\; \frac{1}{2} \pbb{\frac{N}{2} y^2 - \log \gamma + \gamma - 1}\,.
\end{align*}
Therefore
\begin{align*}
S(\zeta_t | g_{2/N}) \;\leq\; N m_2(\zeta_0) - \log \gamma\,,
\end{align*}
and the claim follows.
\end{proof}

The estimate \eqref{SA2} follows from \eqref{gap} and \eqref{bound on bulk eigenvalues}.
It only remains to verify \eqref{SA3}.

\begin{lemma}\label{lemma: pointwise bound on f}
For any $t \in [t_0, \tau]$ we have
\begin{align} \label{bounds on log g}
(\theta_1 \theta_2)(\wh x) \absb{\log \avg{\theta g_t}(\wh x)}^2 \;\leq\; N^C\,.
\end{align}
\end{lemma}

\begin{proof}
Let $\zeta_t$ be the law of an off-diagonal entry $a$ of $A_t$ (the diagonal entries are treated similarly). From 
\eqref{law of Ht} we find
\begin{align*}
\zeta_t \;=\; \varrho_\gamma * g_{2 \gamma/N}\,,
\end{align*}
where $\gamma = 1 - \me^{-t}$, $\varrho_\gamma$ is the law of $(1 - \gamma)^{1/2} \zeta_0$, and $g_\lambda$ is a 
standard Gaussian with variance $\lambda$. Using $\dd a$ to denote Lebesgue measure, we find by explicit calculation that
\begin{align*}
\me^{-N^C - N^C a^2} \;\leq\; \frac{\dd \zeta_t}{\dd a} \;\leq\; \me^{N^C - \frac{N}{4} a^2}\,,
\end{align*}
which gives
\begin{align*}
\me^{-N^C - N^C a^2} \;\leq\; \frac{\dd \zeta_t}{\dd g_{2 \gamma/N}} \;\leq\; \me^{N^C}\,.
\end{align*}
Therefore, the density $F_t(A)$ of the law of $A$ with respect to the GOE measure satisfies
\begin{align*}
\me^{-N^C - N^C \tr A^2} \;\leq\; F_t(A) \;\leq\; \me^{N^C}\,.
\end{align*}
Parametrizing $A = A(\f x, \f v)$ using the eigenvalues $\f x$ and eigenvectors $\f v$, the GOE measure can be written 
in the factorized form $\mu(\dd \f x) P(\dd \f v)$, where $\mu$ is defined in \eqref{GOE measure} and $P$ is a 
probability measure. Thus we get that the density
\begin{align*}
f_t(\f x) \;=\; \int F_t(\f x, \f v) \, P(\dd \f v)
\end{align*}
satisfies
\begin{align} \label{bounds on f}
\me^{-N^C - N^C \sum_i x_i^2} \;\leq\; f_t(\f x) \;\leq\; \me^{N^C}\,.
\end{align}

Next, it is easy to see that
\begin{align} \label{bounds on psi}
\me^{-N^C - N^C \sum_i x_i^2} \;\leq\; \psi(\f x) \;\leq\; \me^{N^C}\,.
\end{align}
Using \eqref{bounds on psi} we may now derive an upper bound on $\avg{\theta_3 g_t}$:
\begin{align*}
\avg{\theta_3 g_t}(\wh x) &\;=\; \frac{\int \dd x_N \, \theta_3(x_N) f_t(\wh x, x_N) \mu(\wh x, x_N)}{\int \dd x_N \, 
\psi(\wh x, x_N) \mu(\wh x, x_N)} \\
&\;\leq\; \me^{N^C + N^C \sum_{i < N} x_i^2} \frac{\int \dd x_N \, \mu(\wh x, x_N)}{\int \dd x_N \, \me^{- N^C x_N^2} 
\mu(\wh x, x_N)}\,.
\end{align*}
Since
\begin{align} \label{lower bound on mu exp}
\frac{\int \dd x_N \, \me^{- N^C x_N^2} \mu(\wh x, x_N)}{\int \dd x_N \, \mu(\wh x, x_N)} \;=\; \frac{\int_{x_{N - 
1}}^\infty \dd x_N \, \me^{-N^C x_N^2} \prod_{i < N} (x_N - x_i) \me^{-\frac{N}{4} x_N^2}}{\int_{x_{N - 1}}^\infty \dd 
x_N \, \prod_{i < N} (x_N - x_i) \me^{-\frac{N}{4} x_N^2}} \;\geq\; \me^{-N^C -N^C \sum_{i < N} x_i^2}
\end{align}
by a straightforward calculation, we get
\begin{align*}
\avg{\theta_3 g_t}(\wh x) \;\leq\; \me^{N^C + N^C \sum_{i < N} x_i^2}\,.
\end{align*}

We now derive a lower bound on $\avg{\theta_3 g_t}$. Using \eqref{bounds on psi} and \eqref{bounds on f} we find
\begin{align*}
\avg{\theta_3 g_t}(\wh x) &\;\geq\; \me^{-N^C} \frac{\int \dd x_N \, \theta_3(x_N) f_t(\wh x, x_N) \mu(\wh x, x_N)}{\int 
\dd x_N \, \mu(\wh x, x_N)} \\
&\;\geq\;
\me^{-N^C - N^C \sum_{i < N} x_i^2} \frac{\int_{2 + \sigma/2}^\infty \dd x_N \, \me^{-N^C x_N^2} \mu(\wh x, 
x_N)}{\int_{x_{N - 1}}^\infty \dd x_N \, \mu(\wh x, x_N)}
\\
&\;\geq\; \me^{-N^C - N^C \sum_{i < N} x_i^2}\,,
\end{align*}
by a calculation similar to \eqref{lower bound on mu exp}. The claim follows from
\begin{equation*}
(\theta_1 \theta_2)(\wh x) \abs{\log \avg{\theta g_t}(\wh x)}^2 \;\leq\; 2 (\theta_1 \theta_2)(\wh x) \abs{\log \theta_1 
\theta_2}^2 + 2 (\theta_1 \theta_2)(\wh x) \abs{\log \avg{\theta_3 g_t}(\wh x)}^2 \;\leq\; 2 + N^C\,. \qedhere
\end{equation*}
\end{proof}

\section{Bulk universality: proof of Theorem \ref{theorem: bulk universality}} \label{sect: bulk universality}
Similarly to \eqref{definition of marginal}, we define $p_{t, N}^{(N - 1)}(x_1, \dots, x_{N - 1})$ as the probability 
density obtained by symmetrizing (in the variables $x_1, \dots, x_{N - 1}$) the function
$\int \dd x_N \, f_t(\f x) \, \mu(\f x)$,
and set, for $n \leq N - 1$,
\begin{align*}
p_{t, N}^{(n)}(x_1, \dots, x_n) \;\deq\; \int \dd x_{n + 1} \cdots \dd x_{N - 1} \, p_{t,N}^{(N - 1)}(x_1, \dots, x_{N - 
1})\,.
\end{align*}
We begin with a universality result for sparse matrices with a small Gaussian convolution.
\begin{theorem} \label{theorem: universality of Gaussian divisibles}
Let $E \in [-2 + \kappa,2 - \kappa]$ for some $\kappa > 0$ and let $b \equiv b_N$ satisfy $\abs{b} \leq \kappa/2$. Pick 
$\epsilon, \beta > 0$, and set $\tau \deq N^{-2 \alpha + \beta}$, where
\begin{equation} \label{def of alpha phi}
\alpha \;\equiv\; \alpha(\phi) \;\deq\; \min \hbb{2 \phi - \frac{1}{2} \,,\, 4 \phi - \frac{2}{3}}\,.
\end{equation}
Let $n \in \N$ and $O : \R^n \to \R$ be compactly supported and continuous.  Then there is a $\bar \tau \in [\tau/2, 
\tau]$ such that
\begin{multline}
\absBB{\int_{E - b}^{E + b} \frac{\dd E'}{2 b} \int \dd \alpha_1 \cdots \dd \alpha_n\, O(\alpha_1, \dots, \alpha_n) \, 
\frac{1}{\varrho_{sc}(E)^n}\pb{p_{\bar \tau, N}^{(n)} - p_{{\rm GOE}, N}^{(n)}} \pbb{E' + 
\frac{\alpha_1}{N\varrho_{sc}(E)}, \dots, E' + \frac{\alpha_n}{N\varrho_{sc}(E)}}}
\\
\leq\; C_n N^\epsilon \qB{b^{-1} N^{- 2 \phi} + b^{-1/2} N^{-\beta / 2}}\,.
\end{multline}
\end{theorem}

\begin{proof}
The claim follows from Theorem \ref{theorem: flow} and Theorem \ref{thm: eigenvalue locations}, similarly to the proof 
of Theorem 2.1 in \cite{ESYY}. We use that $Q \leq (\log N)^{C \xi} N^{- 2 \alpha}$, as follows from \eqref{main 
estimate on Q}; the contribution of the low probability complement event to \eqref{main estimate on Q} may be easily 
estimated using
Cauchy-Schwarz and the estimate $\sum_i \E^t x_i^4 = \E^t \tr A^4 \leq N^C$, uniformly for $t \geq 0$. The assumption IV 
of \cite{ESYY} is a straightforward consequence of the local semicircle law, Theorem \ref{LSCTHM}.
\end{proof}

\begin{proposition} \label{proposition: Green function comparison}
Let $A^{(1)} = (a_{ij}^{(1)})$ and $A^{(2)} = (a_{ij}^{(2)})$ be sparse random matrices, both satisfying Definition 
\ref{definition of A} with
\begin{align*}
q^{(1)} \;\sim\; q^{(2)} \;\geq\; N^\phi
\end{align*}
(in self-explanatory notation). Suppose that, for each $i,j$, the first three moments of $a_{ij}^{(1)}$ and 
$a_{ij}^{(2)}$ are the same, and that the fourth moments satisfy
\begin{align} \label{assumption on fourth moments}
\absb{\E \pb{a_{ij}^{(1)}}^4 - \E \pb{a_{ij}^{(2)}}^4} \;\leq\; N^{-2 - \delta}\,,
\end{align}
for some $\delta > 0$.

Let $n \in \N$ and let $F \in C^5(\C^n)$. We assume that, for any multi-index $\alpha \in \N^{n}$ with $1 \leq 
\abs{\alpha} \leq 5$ and any sufficiently small $\epsilon' > 0$, we have
\begin{align*}
\max \hbb{\absb{\partial^\alpha F(x_1, \dots, x_n)} \col \sum_i \abs{x_i} \leq N^{\epsilon'}} \;\leq\; N^{C_0 \epsilon'}
\,, \quad
\max \hbb{\absb{\partial^\alpha F(x_1, \dots, x_n)} \col \sum_i \abs{x_i} \leq N^2} \;\leq\; N^{C_0}\,,
\end{align*}
where $C_0$ is a constant.

Let $\kappa > 0$ be arbitrary. Choose a sequence of positive integers $k_1, \dots, k_n$ and real parameters $E_j^m \in 
[-2 + \kappa, 2 - \kappa]$, where $m = 1, \dots, n$ and $j = 1, \dots, k_m$.
Let $\epsilon > 0$ be arbitrary and choose $\eta$ with $N^{-1 -\epsilon} \leq \eta \leq N^{-1}$.
Set $z_j^m \deq E_j^m \pm \ii \eta$ with an arbitrary choice of the $\pm$ signs.

Then, abbreviating $G^{(l)}(z) \deq (A^{(l)} - z)^{-1}$, we have
\begin{multline*}
\absa{\E F \pa{\frac{1}{N^{k_1}} \tr \qa{\prod_{j = 1}^{k_1} G^{(1)}(z_j^1)}, \dots, \frac{1}{N^{k_n}} \tr \qa{\prod_{j 
= 1}^{k_n} G^{(1)}(z_j^n)}} - \E F \pb{G^{(1)} \to G^{(2)}}}
\\
\leq\; C N^{1 - 3 \phi + C \epsilon} + CN^{-\delta + C
\epsilon} \,.
\end{multline*}
\end{proposition}

\begin{proof}
The proof of Theorem 2.3 in  \cite{EYY} may be reproduced almost verbatim; the rest term in the Green function 
expansion is estimated by an $L^\infty$-$L^1$ bound using $\E \abs{a_{ij}^{(l)}}^5 \leq C N^{-1 - 3 \phi}$.
\end{proof}

As in  \cite{EYY} (Theorem 6.4), Proposition \ref{proposition: Green function comparison} readily implies the 
following correlation function comparison theorem.

\begin{theorem} \label{theorem: correlation function comparison}
Suppose the assumptions of Proposition \ref{proposition: Green function comparison} hold. Let $p^{(n)}_{(1),N}$ and 
$p^{(n)}_{(2),N}$ be $n$-point correlation functions of the eigenvalues of $A^{(1)}$ and $A^{(2)}$ respectively. Then 
for any $\abs{E} < 2$, any $n \geq 1$ and any compactly supported test function $O : \R^n \to \R$ we have
\begin{align*}
\lim_{N \to \infty} \int \dd \alpha_1 \cdots \dd \alpha_n \, O(\alpha_1, \dots, \alpha_n) \pB{p^{(n)}_{(1),N} - 
p^{(n)}_{(2),N}} \pbb{E + \frac{\alpha_1}{N}, \dots, E + \frac{\alpha_n}{N}} \;=\; 0\,.
\end{align*}
\end{theorem}

We may now complete the proof of Theorem \ref{theorem: bulk universality}.

\begin{proof}[Proof of Theorem \ref{theorem: bulk universality}]
In order to invoke Theorems \ref{theorem: universality of
Gaussian divisibles} and \ref{theorem: correlation function comparison}, we construct a sparse matrix $A_0$, satisfying
Definition \ref{definition of A}, such that its time evolution $A_{\bar \tau}$ is close to $A$ in the sense of the 
assumptions of Proposition \ref{proposition: Green function comparison}. For definiteness, we concentrate on 
off-diagonal elements (the diagonal elements are dealt with similarly).

For the following we fix $i < j$; all constants in the following are uniform in $i,j$, and $N$.
Let $\xi, \xi', \xi_0$ be random variables equal in distribution to $a_{ij}, (a_{\bar \tau})_{ij}, (a_0)_{ij}$
respectively. For any random variable $X$ we use the notation $\wt X \deq X - \E X$.
Abbreviating $\gamma \deq 1 - \me^{- \bar \tau}$, we have
\begin{align*}
\xi' \;=\; \sqrt{1 - \gamma} \, \xi_0 + \sqrt{\gamma} \, g\,,
\end{align*}
where $g$ is a centred Gaussian with variance $1/N$, independent of $\xi_0$. We shall construct a random variable $\xi_0$, supported on at most three points,
such that $A_0$ satisfies Definition \ref{definition of A} and the first four moments of $\xi'$ are sufficiently close 
to
those of $\xi$. For $k =
1,2,\dots$ we denote by $m_k(X)$ the $k$-th moment of a random variable $X$.
We set
\begin{align}
\xi_0 \;=\; \frac{f}{\sqrt{1 - \gamma} \, N} + \wt \xi_0\,,
\end{align}
where $m_1(\wt \xi_0) = 0$ and $m_2(\wt \xi_0) = N^{-1}$. It is easy to see that $m_k(\xi) = m_k(\xi')$ for $k = 1,2$.

We take the law of $\wt \xi_0$ to be of the form
\begin{align*}
p \delta_a + q \delta_{-b} + (1 - p - q) \delta_0
\end{align*}
where $a,b,p,q \geq 0$ are parameters satisfying $p+q \leq 1$.
The conditions $m_1(\wt \xi_0) = 0$ and $m_2(\wt \xi_0) = N^{-1}$ imply
\begin{align*}
p \;=\; \frac{1}{aN(a+b)} \,, \qquad q \;=\; \frac{1}{bN(a+b)}\,.
\end{align*}
Thus, we parametrize $\xi_0$ using $a$ and $b$; the condition $p+q \leq 1$ reads $ab \geq N^{-1}$. Our aim
is to determine $a$ and $b$ so that $\xi_0$ satisfies \eqref{moment conditions}, and so that the third and
fourth moments of $\xi'$ and $\xi$ are close.
By explicit computation we find
\begin{align} \label{m in terms of ab}
m_3(\wt \xi_0) \;=\; \frac{a - b}{N}\,, \qquad m_4(\wt \xi_0) \;=\; N m_3(\wt \xi_0)^2 + \frac{ab}{N}\,.
\end{align}
Now we require that $a$ and $b$ be chosen so that $ab \geq N^{-1}$ and
\begin{align*}
m_3(\wt \xi_0) \;=\; (1 - \gamma)^{-3/2} m_3(\wt \xi)\,, \qquad m_4(\wt \xi_0) \;=\; N m_3(\wt \xi_0)^2 + m_4(\wt \xi) -
N m_3(\wt \xi)^2\,.
\end{align*}
Using \eqref{m in terms of ab}, it is easy to see that such a pair $(a,b)$ exists provided that $m_4(\wt \xi) - N m_3(\wt \xi)^2 \geq
N^{-2}$. This latter estimate is generally valid for any random variable with $m_1 = 0$; it follows from the elementary
inequality $m_4 m_2 - m_3^2 \geq m_2^3$ valid whenever $m_1 = 0$.

Next, using \eqref{m in terms of ab} and the estimates $m_3(\wt \xi) = O(N^{-1 -\phi})$, $m_4(\wt \xi) = O(N^{-1 - 2
\phi})$, we find
\begin{align*}
a - b \;=\; O(N^{-\phi})\,, \qquad ab \;=\; O(N^{-2 \phi})\,,
\end{align*}
which implies $a,b = O(N^{-\phi})$. We have hence proved that $A_0$ satisfies Definition \ref{definition of A}.

One readily finds that $m_3(\xi') = m_3(\xi)$. Moreover, using
\begin{align*}
m_4(\wt \xi_0) - m_4(\wt \xi) \;=\; N m_3(\wt \xi)^2 \qb{(1 - \gamma)^{-3} - 1} \;=\; O(N^{-1 - 2 \phi} \gamma)\,,
\end{align*}
we find
\begin{align*}
m_4(\wt \xi') - m_4(\wt \xi) \;=\; (1 - \gamma)^2 m_4(\wt \xi_0) + \frac{6 \gamma}{N^2} - \frac{3 \gamma^2}{N^2} -
m_4(\wt \xi) \;=\; O(N^{-1 - 2 \phi}\gamma)\,.
\end{align*}
Summarizing, we have proved
\begin{align} \label{moment matching result}
m_k(\xi') \;=\; m_k(\xi)  \quad (k = 1,2,3), \qquad \abs{m_4(\xi') - m_4(\xi)} \;\leq\; C N^{- 1 - 2 \phi} \bar \tau\,.
\end{align}

The claim follows now by setting $\delta = 2 \alpha(\phi) + 2 \phi - 1 - \beta$ in \eqref{assumption on fourth moments}, 
and invoking Theorems \ref{theorem: universality of Gaussian divisibles} and \ref{theorem: correlation function 
comparison}.
\end{proof}

\section{Edge universality: proof of Theorem \ref{thm: edge}}

\subsection{Rank-one perturbations of the GOE} \label{sect: GOE perturbation}
We begin by deriving a simple, entirely deterministic, result on the eigenvalues of rank-one perturbations of matrices.  
We choose the perturbation to be proportional to $\ket{\f e}\bra{\f e}$, but all results of this subsection hold 
trivially if $\f e$ is replaced with an arbitrary $\ell^2$-normalized vector.

\begin{lemma}[Monotonicity and interlacing] \label{lemma: perturbation}
Let $H$ be a symmetric $N \times N$ matrix. For $f \geq 0$ we set
\begin{equation*}
A(f) \;\deq\; H + f \ket{\f e} \bra{\f e}\,.
\end{equation*}
Denote by $\lambda_1 \leq \cdots \leq \lambda_N$ the eigenvalues of $H$, and by $\mu_1(f) \leq \cdots \leq \mu_N(f)$ the 
eigenvalues of $A(f)$. Then for all $\alpha = 1, \dots, N - 1$ and $f \geq 0$ the function $\mu_\alpha(f)$ is nondecreasing, satisfies
$\mu_\alpha(0) = \lambda_\alpha$, and has the interlacing property
\begin{equation} \label{interlacing property}
\lambda_\alpha \;\leq\; \mu_\alpha(f) \;\leq\; \lambda_{\alpha + 1}\,.
\end{equation}
\end{lemma}
\begin{proof}
From \cite{EKYY}, Equation (6.3), we find that $\mu$ is an eigenvalue of $H + f \ket{\f e}\bra{\f e}$ if and only if
\begin{equation} \label{interlacing equation}
\sum_\alpha \frac{\abs{\scalar{\f u_\alpha}{\f e}}^2}{\mu - \lambda_\alpha} \;=\; \frac{1}{f}\,,
\end{equation}
where $\f u_\alpha$ is the eigenvector of $H$ associated with the eigenvalue $\lambda_\alpha$. The right-hand side of 
\eqref{interlacing equation} has $N$ singularities at $\lambda_1, \dots, \lambda_N$, away from which it is decreasing. All 
claims now follow easily.
\end{proof}

Next, we establish the following ``eigenvalue sticking'' property for GOE. Let $\alpha$ label an eigenvalue close to the 
right (say) spectral edge. Roughly we prove that, in the case where $H = V$ is a GOE matrix and $f > 1$, the eigenvalue 
$\mu_\alpha$ of $V + f \ket{\f e}\bra{\f e}$ ``sticks'' to $\lambda_{\alpha + 1}$ with a precision $(\log N)^{C \xi} 
N^{-1}$.  This behaviour can be interpreted as a form of long-distance level repulsion, in which the eigenvalues 
$\mu_\beta$, $\beta < \alpha$, repel the eigenvalue $\mu_\alpha$ and push it close to its maximum possible value, 
$\lambda_{\alpha + 1}$.

\begin{lemma}[Eigenvalue sticking] \label{16SD}
Let $V$ be an $N\times N$ GOE matrix. Suppose moreover that $\xi$ satisfies \eqref{assumptions for SLSC} and that $f$ 
satisfies $f \geq 1 + \epsilon_0$. Then there is a $\delta \equiv \delta(\epsilon_0) > 0$ such that for all $\alpha$ 
satisfying $N (1 - \delta) \leq \alpha \leq N - 1$ we have \hp{\xi}{\nu}
\begin{equation} \label{pp38}
\abs{\lambda_{\alpha + 1} - \mu_\alpha} \;\leq\; \frac{(\log N)^{C \xi}}{N}\,.
\end{equation}
Similarly, if $\alpha$ instead satisfies $\alpha \leq N \delta$ we have \hp{\xi}{\nu}
\begin{equation} \label{pp39}
\abs{\mu_\alpha - \lambda_\alpha} \;\leq\; \frac{(\log N)^{C \xi}}{N}\,.
\end{equation}
\end{lemma}

For the proof of Lemma \ref{16SD} we shall need the following result about Wigner matrices, proved in 
\cite{EYYrigidity}.
\begin{lemma}
Let $H$ be a Wigner matrix with eigenvalues $\lambda_1 \leq \cdots \leq \lambda_N$ and associated eigenvectors $\f u_1, 
\dots, \f u_N$. Assume that $\xi$ is given by \eqref{bounds on xi}. Then the following two statements hold 
\hp{\xi}{\nu}:
\begin{equation} \label{deloc of u alpha}
\max_\alpha \norm{\f u_\alpha}_\infty \;\leq\; \frac{(\log N)^{C \xi}}{\sqrt{N}}\,,
\end{equation}
and
\begin{equation} \label{rigidity for Wigner}
\abs{\lambda_\alpha - \gamma_\alpha} \;\leq\; (\log N)^{C \xi} N^{-2/3} \pb{\min \h{\alpha, N + 1 - \alpha}}^{-1/3}\,.
\end{equation}
Moreover, let $L$ satisfy \eqref{def of L} and write $G^H_{ij}(z) \deq \qb{(H - z)^{-1}}_{ij}$. Then we have, 
\hp{\xi}{\nu},
\begin{equation} \label{sc law for Wigner}
\bigcap_{z \in D_L} \hBB{\max_{1 \leq i,j \leq N} \absb{G^H_{ij}(z) - \delta_{ij} m_{sc}(z)} \leq (\log N)^{C \xi}  
\pbb{\sqrt{\frac{\im m_{sc}(z)}{N \eta}} + \frac{1}{N \eta}}}\,,
\end{equation}
where $D_L$ was defined in \eqref{definitionD}
\end{lemma}

\begin{proof}[Proof of Lemma \ref{16SD}]
We only prove \eqref{pp38}; the proof of \eqref{pp39} is analogous.
By orthogonal invariance of $V$, we may replace $\f e$ with the vector $(1,0, \dots, 0)$. Let us abbreviate $\zeta_\beta 
\deq \abs{u_\beta(1)}^2$. Note that \eqref{deloc of u alpha} implies
\begin{equation} \label{bound on zeta}
\max_\beta \zeta_\beta \;\leq\; (\log N)^{C \xi} N^{-1}
\end{equation}
\hp{\xi}{\nu}.
Now from we \eqref{interlacing equation} we get
\begin{equation*}
\frac{\zeta_\alpha}{\mu_\alpha - \lambda_{\alpha + 1}} + \sum_{\beta \neq \alpha + 1} \frac{\zeta_\beta}{\mu_\alpha - 
\lambda_\beta} \;=\; \frac{1}{f}\,,
\end{equation*}
which yields
\begin{equation} \label{lambda mu step 1}
\abs{\lambda_{\alpha + 1} - \mu_\alpha} \;=\; \zeta_\alpha \absbb{\sum_{\beta \neq \alpha + 1} 
\frac{\zeta_\beta}{\lambda_\beta - \mu_\alpha} + \frac{1}{f}}^{-1}\,.
\end{equation}
We estimate from below, introducing an arbitrary $\eta > 0$,
\begin{align}
- \sum_{\beta \neq \alpha + 1} \frac{\zeta_\beta}{\lambda_\beta - \mu_\alpha} &\;=\; \sum_{\beta < \alpha + 1} 
\frac{\zeta_\beta}{\mu_\alpha - \lambda_\beta} - \sum_{\beta > \alpha + 1} \frac{\zeta_\beta}{\lambda_\beta - 
\mu_\alpha}
\notag \\
&\;\geq\;
\sum_{\beta < \alpha + 1} \frac{\zeta_\beta (\mu_\alpha - \lambda_\beta)}{(\mu_\alpha - \lambda_\beta)^2 + \eta^2} - 
\sum_{\beta > \alpha + 1} \frac{\zeta_\beta}{\lambda_\beta - \mu_\alpha}
\notag \\
&\;=\;
- \re G^V_{11}(\mu_\alpha + \ii \eta) + \sum_{\beta > \alpha + 1} \frac{\zeta_\beta (\lambda_\beta - 
\mu_\alpha)}{(\lambda_\beta - \mu_\alpha)^2 + \eta^2} - \sum_{\beta > \alpha + 1} \frac{\zeta_\beta}{\lambda_\beta - 
\mu_\alpha}
\notag \\ \label{lower bound on zeta sum}
&\;\geq\;
- \re G^V_{11}(\mu_\alpha + \ii \eta) - \sum_{\beta > \alpha + 1} \frac{\zeta_\beta \eta^2}{(\lambda_\beta - 
\mu_\alpha)^3}\,,
\end{align}
where in the third step we used that $\lambda_{\alpha + 1} \geq \mu_\alpha$ by \eqref{interlacing property}.

We now choose $\eta = (\log N)^{C_1 \log \log N} N^{-1}$. For $C_1$ large enough, we get from \eqref{sc law for Wigner} 
that $G^V_{11}(\mu_\alpha + \ii \eta) = m_{sc}(\mu_\alpha + \ii \eta) + o(1)$. Therefore \eqref{formula for m sc} yields
\begin{equation} \label{first bound on G11}
- \re G^V_{11}(\mu_\alpha + \ii \eta) \;\geq\; 1 - 2 \sqrt{\abs{2 - \mu_\alpha}} + o(1)\,.
\end{equation}
From \eqref{rigidity for Wigner} and \eqref{interlacing property} we get that $\abs{\mu_\alpha - \gamma_\alpha} \leq 
(\log N)^{C \xi} N^{-2/3}$ \hp{\xi}{\nu}. Moreover, the definition \eqref{def of gamma} and $\alpha \geq N (1 - \delta)$ 
imply $\abs{\gamma_\alpha - 2} \leq C \delta^{2/3}$. Thus we get, \hp{\xi}{\nu}, that $\abs{2 - \mu_\alpha} = o(1) + C 
\delta^{2/3}$.  Therefore \eqref{first bound on G11} yields, \hp{\xi}{\nu},
\begin{equation*}
- \re G^V_{11}(\mu_\alpha + \ii \eta) \;\geq\; 1 + o(1) - C \delta^{1/3}\,.
\end{equation*}
Recalling \eqref{bound on zeta}, we therefore get from \eqref{lower bound on zeta sum}, \hp{\xi}{\nu},
\begin{equation} \label{zeta estimate step 2}
\absBB{\sum_{\beta \neq \alpha + 1} \frac{\zeta_\beta}{\lambda_\beta - \mu_\alpha}} \;\geq\; 1 + o(1) - C \delta^{1/3}- 
\frac{m (\log N)^{C \xi}}{N^3 \abs{\lambda_{\alpha + 1} - \mu_\alpha}^3} - \frac{(\log N)^{C \xi}}{N^3} \sum_{\beta > 
\alpha + m} \frac{1}{\abs{\lambda_\beta - \mu_\alpha}^3}\,,
\end{equation}
for any $m \in \N$.

Next, from \eqref{rigidity for Wigner} we find that, provided $C_2$ is large enough, $m \deq (\log N)^{C_2 \xi}$, and 
$\beta > \alpha + m$, then we have \hp{\xi}{\nu}
\begin{equation*}
\abs{\lambda_\beta - \lambda_{\alpha + 1}} \;\geq\; \abs{\gamma_{\beta} - \gamma_{\alpha + 1}} - \frac{(\log N)^{C 
\xi}}{N^{2/3} (N + 1 - \beta)^{1/3}}
\;\geq\; c \abs{\gamma_\beta - \gamma_{\alpha + 1}}\,.
\end{equation*}
Then for $C_2$ large enough we have, \hp{\xi}{\nu},
\begin{equation*}
\sum_{\beta > \alpha + m} \frac{1}{\abs{\lambda_\beta - \mu_\alpha}^3} \;\leq\; C \sum_{\beta > \alpha + m} 
\frac{1}{(\gamma_\beta - \gamma_{\alpha + 1})^3} \;\leq\; \frac{C N^3}{(\log N)^{3 C_2 \xi}}\,.
\end{equation*}
Thus we get from \eqref{zeta estimate step 2}, \hp{\xi}{\nu},
\begin{equation*}
\absBB{\sum_{\beta \neq \alpha + 1} \frac{\zeta_\beta}{\lambda_\beta - \mu_\alpha}} \;\geq\; 1 + o(1) - C \delta^{1/3} - 
\frac{(\log N)^{C \xi}}{N^3 \abs{\lambda_{\alpha + 1} - \mu_\alpha}^3}\,.
\end{equation*}
Plugging this into \eqref{lambda mu step 1} and recalling that $f \geq 1 + \epsilon_0 > 1$ yields, \hp{\xi}{\nu},
\begin{equation*}
\abs{\lambda_{\alpha + 1} - \mu_\alpha} \;\leq\; \frac{(\log N)^{C \xi}}{N} \pBB{\epsilon_0 - C \delta^{1/3} - o(1) - 
\frac{(\log N)^{C \xi}}{N^3 \abs{\lambda_{\alpha + 1} - \mu_\alpha}^2}}^{-1}\,,
\end{equation*}
from which the claim follows.
\end{proof}

\subsection{Proof of Theorem \ref{thm: edge}}
In this section we prove Theorem \ref{thm: edge} by establishing the following comparison result for sparse
matrices. Throughout the following we shall abbreviate the lower bound in \eqref{lower bound on f} by
\begin{equation} \label{def f star}
f_* \;\deq\; 1 + \epsilon_0\,.
\end{equation}

\begin{proposition}\label{twthm}
Let $\P^\bv$ and $\P^\bw$ be laws on the symmetric $N \times N$ random matrices $H$, each satisfying Definition 
\ref{definition of H} with $q \geq N^\phi$ for some $\phi$ satisfying $1/3 < \phi \leq 1/2$. In particular, we have the 
moment matching condition
\begin{equation} \label{2m}
\E^\bv h_{ij} \;=\; \E^\bw h_{ij} \;=\; 0\,, \qquad \E^\bv h_{ij}^2 \;=\; \E^\bw h_{ij}^2 \;=\; \frac{1}{N}\,.
\end{equation}
Set $f \deq f_*$ in Definition \ref{definition of A}: $A \equiv A(f_*) = (a_{ij}) \deq H + f_* \ket{\f e} \bra{\f e}$.  
As usual, we denote the ordered
eigenvalues of $H$ by $\lambda_1 \leq \cdots \leq \lambda_N$ and the ordered eigenvalues of $A$ by $\mu_1 \leq \cdots
\leq \mu_N$.

Then there is a $\delta>0$ such that 
for any $s \in \R$ we have
\begin{multline}
\label{tw0}
 \P^\bv \pB{ N^{2/3} ( \lambda_{N } -2) \le s- N^{-\delta}}- N^{-\delta}
\\
\le\;   \P^\bw \pb{ N^{2/3} ( \lambda_{N } -2) \le s }   \;\le\; \P^\bv \pB{ N^{2/3} ( \lambda_{N } -2) \le s+ 
N^{-\delta} } + N^{-\delta}
\end{multline}
as well as
\begin{multline}\label{tw1}
 \P^\bv \pB{ N^{2/3} ( \mu_{N - 1} -2) \le s- N^{-\delta}}- N^{-\delta}
\\
\le\;   \P^\bw \pb{ N^{2/3} ( \mu_{N - 1} -2) \le s }   \;\le\; \P^\bv \pB{ N^{2/3} ( \mu_{N - 1} -2) \le s+ N^{-\delta} 
} + N^{-\delta}
\end{multline}
%
for $N \ge N_0$ sufficiently  large, where $N_0$ is independent of $s$. 
%
%
%
\end{proposition}

Assuming Proposition \ref{twthm} is proved, we may easily complete the proof of Theorem \ref{thm: edge} using the 
results of Section
\ref{sect: GOE perturbation}.

\begin{proof}[Proof of Theorem \ref{thm: edge}]
Choose $\P^\bv$ to be the law of GOE (see Remark \ref{remark: GOE assumption}), and choose $\P^\bw$ to be the law of a 
sparse matrix satisfying Definition
\ref{definition of H} with $q \geq N^{\phi}$. We prove \eqref{tw}; the proof of \eqref{tw01} is similar.

For the following we write $\mu_\alpha(f) \equiv \mu_\alpha$ to emphasize the $f$-dependence of the eigenvalues of 
$A(f)$. Using first \eqref{interlacing property} and then \eqref{tw0} we get
\begin{equation*}
\P^\bw \pB{N^{2/3}(\mu_{N - 1}(f) -2) \leq s} \;\geq\;
\P^\bw \pB{N^{2/3}(\lambda_N -2) \leq s} \;\geq\;
\P^\bv \pB{N^{2/3}(\lambda_N -2) \leq s - N^{-\delta}} - N^{- \delta}\,,
\end{equation*}
for some $\delta > 0$.
Next, using first the monotonicity of $\mu_\alpha(f)$ from Lemma \ref{lemma: perturbation}, then \eqref{tw1}, and 
finally \eqref{pp38}, we get
\begin{multline*}
\P^\bw \pB{N^{2/3}(\mu_{N - 1}(f) -2) \leq s} \;\leq\;
\P^\bw \pB{N^{2/3}(\mu_{N - 1}(f_*) -2) \leq s}
\\
\leq\;
\P^\bv \pB{N^{2/3}(\mu_{N - 1}(f_*) -2) \leq s + N^{-\delta}} + N^{-\delta}
\;\leq\;
\P^\bv \pB{N^{2/3}(\lambda_N -2) \leq s + 2N^{-\delta}} + 2N^{-\delta}\,,
\end{multline*}
for some $\delta > 0$. This concludes the proof of \eqref{tw}, after a renaming of $\delta$.
\end{proof}

The rest of this section is devoted to the proof of Proposition \ref{twthm}. We shall only prove \eqref{tw1}. The proof 
of \eqref{tw0} is similar (in fact easier), and relies on the local semicircle law, Theorem \ref{LSCTHM}, with $f = 0$; 
if $f = 0$ some of the following analysis simplifies (e.g.\ the proof of Lemma \ref{lem: 52} below may be completed 
without the estimate from Lemma \ref{non-centred Z-lemma}.)

From now on we always assume the setup of Proposition \ref{twthm}. In particular, $f$ will always be equal to $f_*$.


We begin by outlining the proof of Proposition \ref{twthm}.
The basic strategy is similar to the one used for Wigner matrices in \cite{EYYrigidity} and \cite{KY}.
For any $E_1\le E_2$, let
$$
	\cN(E_1, E_2) \;\deq\; \absb{\{\alpha \col E_1 \leq \mu_\alpha \le E_2\}}
$$
denote the number of eigenvalues of $A$ in the interval $[E_1, E_2]$. In the first step, we express the distribution 
function in terms of Green functions according to
\begin{equation} \label{heuristic first step}
\P^{\f u}\pb{\mu_{N - 1} \geq E} \;=\; \E^{\f u} K\pb{\cal N(E, \infty) - 2} \;\approx\; \E^{\f u} K\pb{\cal N(E, E_*) - 
1} \;\approx\; \E^{\f u} K \pBB{\int_{E}^{E_*} \dd y \, N\, \im m(y + \ii \eta) - 1}\,.
\end{equation}
Here $\f u$ stands for either $\f v$ or $\f w$, $\eta \deq N^{-2/3 - \epsilon}$ for some $\epsilon > 0$ small enough, 
$K:\R \to\R_+$ is a smooth cutoff function satisfying
\begin{equation} \label{definition of K}
K(x) = 1      \quad  \text{if} \quad |x| \le 1/9   \qquad \text{and} \qquad K(x) = 0   \quad  \text{if} \quad |x| \ge 
2/9\,,
\end{equation}
and
\be\label{defEL}
E_* \;\deq\; 2 + 2(\log N)^{C_0 \xi} N^{-2/3}
\ee
for some $C_0$ large enough.
The first approximate identity in \eqref{heuristic first step} follows from Theorem \ref{thm: eigenvalue locations}
which guarantees that $\mu_{N - 1} \leq E_*$ \hp{\xi}{\nu}, and from \eqref{gap 1} which guarantees that $\mu_N \geq 2 + 
\sigma$ \hp{\xi}{\nu}. The second approximate identity in \eqref{heuristic first step} follows from the approximation
\begin{equation*}
\int_{E_1}^{E_2} \dd y \, N \, \im m(y + \ii \eta) \;=\; \sum_{\alpha} \int_{E_1}^{E_2} \dd y \, \frac{\eta}{(y - 
\mu_\alpha)^2 + \eta^2} \;\approx\; \cal N(E_1, E_2)\,,
\end{equation*}
which is valid for $E_1$ and $E_2$ near the spectral edge, where the typical eigenvalue separation is $N^{-2/3} \gg 
\eta$.

The second step of our proof is to compare expressions such as the right-hand side of \eqref{heuristic first step} for 
$\f u = \f v$ and $\f u = \f w$. This is done using a Lindeberg replacement strategy and a resolvent expansion of the 
argument of $K$. This step is implemented in Section \ref{sect: proof of green fn comp}, to which we also refer for a 
heuristic discussion of the argument.

Now we give the rigorous proof of the steps outlined in \eqref{heuristic first step}.
We first collect the tools we shall need. From \eqref{rigidity for large phi} and \eqref{gap 1} we get that there is a 
constant $C_0 > 0$ such that, under both $\P^\bv$ and $\P^\bw$, we have \hp{\xi}{\nu}
\be\label{6-1}
 \abs{N^{2/3} ( \mu_{N-1} -2) } \;\leq\; (\log N)^{C_0 \xi}\,, \qquad \mu_N \;\geq\; 2 + \sigma\,,
\ee
and
\be\label{6-2}
  \cN\left(2-\frac{2(\log N)^{C_0 \xi}}{ N^{2/3}},\;2+\frac{2(\log N)^{C_0 \xi}}{ N^{2/3}}\right) \;\leq\;
  (\log N)^{2 C_0 \xi}\,.
\ee
Therefore in \eqref{tw1} we can assume that $s$ satisfies
\be\label{25}
-(\log N)^{C_0 \xi} \;\le\; s \;\le\; (\log N)^{C_0 \xi}.
\ee
Recall the definition \eqref{defEL} of $E_*$ and
and introduce, for any $E \leq E_*$, the characteristic function on the interval $[E, E_*]$,
\[
\chi_E \;\deq\; {\bf 1}_{[E, E_*]}\,.
\]
For
any $\eta>0$ we define
\be\label{thetam}
\theta_\eta(x) \;\deq\; \frac{\eta }{\pi(x^2+\eta^2)} \;=\; \frac{1}{\pi} \im \frac{1}{x-\ii\eta}
\ee
to be an approximate delta function on scale $\eta$. 

The following result allows us to replace the sharp counting function
$\cN(E, E_*) = \tr \chi_E(H)$ with its approximation smoothed on scale $\eta$. 

\begin{lemma} \label{23} Suppose that $E$ satisfies
\be\label{condE-}
|E-2|N^{2/3} \;\leq\; (\log N)^{C_0 \xi}.
\ee
Let $\ell \deq \frac{1}{2}N^{-2/3 - \e}$ and $\eta \deq N^{-2/3-9\e}$, and recall the definition of the function $K$ 
from \eqref{definition of K}. Then the following statements hold for both ensembles $\P^\bv$ and $\P^\bw$. For some 
$\e>0$ small enough the inequalities
\be\label{41new}
\tr (\chi_{E+ \ell  }  \ast \theta_\eta) (H)  -  N^{-\e} \;\le\;  \cN (E, \infty)-1  \;\le\;  \tr  (\chi_{E- \ell  }  
\ast \theta_\eta) (H)  +  N^{-\e}
\ee
hold \hp{\xi}{\nu}. Furthermore, we have
\be\label{67E}
\E \, K \pB{\tr (\chi_{E-\ell}  \ast \theta_\eta)  (H)}
\;\le\; \P (\cN(E,  \infty)  = 1 )  \;\le\; \E \, K \pB{\tr ( \chi_{E+\ell}  \ast \theta_\eta ) (H)} + \me^{- \nu (\log 
N)^\xi}
\ee for sufficiently large $N$ independent of $E$, as long as \eqref{condE-} holds.
\end{lemma} 

\begin{proof}
The proof of Corollary 6.2 in \cite{EYYrigidity} can be reproduced almost verbatim. In the estimate (6.17) of 
\cite{EYYrigidity}, we need the bound, \hp{\xi}{\nu},
\begin{equation*}
\absb{m(E + \ii \ell) - m_{sc}(E + \ii \ell)} \;\leq\; \frac{(\log N)^{C \xi}}{N \ell}
\end{equation*}
for $N^{-1 + c} \leq \ell \leq N^{-2/3}$. This is an easy consequence of the local semicircle law \eqref{scm} and the 
assumption $q \geq
N^{1/3}$.

Note that, when compared to Corollary 6.2 in \cite{EYYrigidity}, the quantity $\cal N(E, \infty)$ has been incremented 
by one; the culprit is the single eigenvalue $\mu_N \geq 2 + \sigma$.
\end{proof}

Recalling that $\theta_\eta(H)= \frac{1}{\pi}\im G(\ii\eta)$,
Lemma  \ref{23}  bounds the probability of $\cN(E,\infty)=1$ in terms of  expectations
 of functionals of Green functions. We now show that the difference between the expectations of these functionals, with 
respect to the two probability distributions $\P^\bv$ and $\P^\bw$, is negligible assuming their associated second 
moments of $h_{ij}$ coincide. The precise statement is the following Green function
 comparison theorem at the edge. All statements are formulated
for the upper spectral edge 2, but with the same proof they hold for the lower
spectral edge $-2$ as well. 

For the following it is convenient to introduce the shorthand
\be \label{def Ie}
I_\e \;\deq\; \{x \col |x-2| \leq N^{-2/3+\e}\}
\ee
where $\epsilon > 0$.

\begin{proposition} [Green function comparison theorem on the edge] \label{GFCT}
Suppose that the assumptions of Proposition \ref{twthm}  hold.
Let    $F:\R\to \R$ be a function whose derivatives satisfy  
\be\label{gflowder}
\sup_{x}|F^{(n)}(x)| (1 + |x|)^{-C_1} \;\leq\; C_1 \for n \;=\; 1,2,3,4\,,
\ee
with some constant $C_1>0$.
Then there exists a constant $\wt \e>0$, depending only on $C_1$, such that for any $\e<\wt \e$
and for any real numbers $E, E_1, E_2 \in I_\epsilon$,
and setting $\eta \deq N^{-2/3-\e}$,
we have
\begin{align}\label{maincomp}
\Bigg|\E^\bv  F  
\left (   
N \eta \im m (z)    \right )  &  -
\E^\bw  F  
\left (  N \eta \im m (z)      \right )  \Bigg|
\;\le\;  C N^{1/3+C \e}q^{-1} \for z \;=\; E + \ii \eta\,,
\end{align}
and 
\be\label{c1}
\left|\E^\bv  F\left(N \int_{E_1}^{E_2} \rd y \;  \im m (y +\ii\eta)\right)-\E^\bw F \left(N \int_{E_1}^{E_2}  \rd y \; 
\im m (y+\ii\eta ) \right)\right| \;\leq\; N^{1/3+C \e}q^{-1}
\ee for some constant $C$ and large enough $N$.
\end{proposition}

We postpone the proof of Proposition \ref{GFCT} to the next section. Assuming it proved, we now have all the ingredients 
needed to complete the proof of Proposition \ref{twthm}.

\begin{proof}[Proof of Proposition \ref{twthm}]
As observed after \eqref{6-1} and \eqref{6-2}, we may assume that \eqref{25} holds.
We define $E:=2+sN^{-2/3}$ that satisfies \eqref{condE-}.
We define $E_*$ as in \eqref{defEL} with the  $C_0$ such that \eqref{6-1} and \eqref{6-2} hold.
From \eqref{67E} we get, for any sufficiently small $\e>0$,
\be
\E^\bw \,  K \pB {\tr (\chi_{E-\ell}  \ast \theta_\eta)  (H)}
\;\le\; \P^\bw (\cN(E, \infty)  = 1 )
\ee
where we set
$$
	\ell \;\deq\; \frac{1}{2}N^{-2/3-\e}\,, \qquad \eta \;\deq\; N^{-2/3-9\e}.
$$
Now \eqref{c1}
applied to the case $E_1=E-\ell$ and $E_2=E_*$
shows that  there exists a $\delta>0$ such that for sufficiently small $\e>0$ we have
\be\label{645}
\E^\bv  K \pB{\tr (\chi_{E-\ell}  \ast \theta_\eta)  (H)} \;\leq\; \E^\bw\, K \pB{\tr (\chi_{E-\ell}  \ast \theta_\eta)  
(H)} + N^{-\delta}
\ee
(note that here $9\e$  plays the role of $\e$ in the Proposition \ref{GFCT}).
Next, the second bound of \eqref{67E} yields
\be
\P^\bv (\cN(E-2\ell, \infty)  = 1 ) \;\leq\; \E^\bv\,  K \pB {\tr (\chi_{E-\ell}  \ast \theta_\eta)  (H)} + \me^{- \nu 
(\log N)^\xi}
\ee

Combining these inequalities, we have
\be\label{Pbv}
\P^\bv (\cN(E- 2\ell, \infty)  = 1 ) \;\leq\; \P^\bw (\cN(E, \infty)  = 1 ) + 2N^{-\delta}
\ee
for sufficiently small $\e>0$ and  sufficiently large $N$. Setting $E = 2+sN^{-2/3}$ proves the first inequality  of 
\eqref{tw1}. Switching the roles of $\bv$ and $\bw$ in \eqref{Pbv} yields the second inequality of \eqref{tw1}.
\end{proof}

\subsection{ Proof of Proposition \ref{GFCT}} \label{sect: proof of green fn comp}
All that remains is the proof of Proposition \ref{GFCT}, to which this section is devoted. Throughout this section we 
suppose that the assumptions of Proposition \ref{twthm} hold, and in particular that $f = 1 + \epsilon_0$.

We now set up notations to  replace the matrix elements one by one. This step is identical for the proof of both  
\eqref{maincomp} and \eqref{c1}; we use the notations of the case \eqref{maincomp}, for which they are less involved. 

For the following it is convenient to slightly modify our notation. We take two copies of our probability space, one of 
which carries the law $\P^\bv$ and the other the law $\P^\bw$. We work on the product space and write $H^\bv$ for the 
copy carrying the law $\P^\bv$ and $H^\bw$ for the copy carrying the law $\P^\bw$. The matrices $A^\bv$ and $A^\bw$ are 
defined in the obvious way, and we use the notations $A^\bv = (v_{ij})$ and $A^\bw = (w_{ij})$ for their entries.  
Similarly, we denote by $G^\bv(z)$ and $G^\bw(z)$ the Green functions of the matrices $A^\bv$ and $A^\bw$.

Fix a bijective ordering map on the index set of
the independent matrix elements,
\[
\phi \col \{(i, j) \col 1 \le i\le  j \le N \} \;\to\; \Big\{0, \ldots, \gamma_{\rm max}\Big\} \where \gamma_{\rm max} 
\;\deq\; \frac{N(N+1)}{2} - 1\,,
\] and denote by  $A_\gamma$  the generalized Wigner matrix whose matrix elements $a_{ij}$ follow
the $v$-distribution if $\phi(i,j)\le \gamma$ and they follow the $w$-distribution
otherwise; in particular $A_0 = A^\bv$ and $ A_{\gamma_{\rm max}} = A^\bw$.

Next, set $\eta \deq N^{-2/3 - \epsilon}$. We use the identity
\be\label{imscbound}
\im m_{sc}(E+\ii \eta) \;\le\; \sqrt{|E-2|+\eta} \;\le\; CN^{-1/3+\e/2}\,.
\ee
Therefore Theorem 2.9 of \cite {EKYY} yields, \hp{\xi}{\nu},
\be\label{basic}
\max_{0 \le \gamma \le \gamma_{\rm max}} \max_{1 \le k,l \le N}  \max_{E\in I_\e}\left |  \left (\frac 1 {  
A_{\gamma}-E- \ii \eta} \right )_{k l }
 -\delta_{kl}m_{sc}(E+ \ii \eta)
\right | \;\le\; \frac{1}{p}
\ee
where we defined
\be \label{def of p}
\frac{1}{p} \;\deq\; N^{\e}\left(q^{-1}+\frac1{N\eta}\right) \;\leq\; N^{-1/3+2\e} \,.
\ee

We set $z=E+\ii\eta$ where $E\in I_\e$ and $\eta=N^{-2/3-\e}$.
Using \eqref{basic}, \eqref{def of p}, and the identity
$$
  \im m \;=\; \frac{1}{N}\im \tr G \;=\; \frac{\eta}{N}\sum_{ij}G_{ij}\overline{G_{ij}}\,,
$$
we find, as in (6.36) of \cite{EYYrigidity}, that in order to prove \eqref{maincomp} it is enough to prove
\be\label{c1new}
\left|
\E  F\left(\eta^2 \sum_{i\neq j}G_{ij}^\bv\overline {G_{ji}^\bv} \right) -
\E F\left ( G^\bv \to  G^\bw\right )
\right| \;\leq\; C N^{1/3+C \e}q^{-1}
\ee
at $z = E + \ii \eta$.
We write the quantity in the absolute value on the left-hand side of \eqref{c1new} as a telescopic sum,
\begin{multline}\label{tel}
\E \, F \left (\eta^2 \sum_{i\neq j}\left(\frac{1}{A^\bv-z}\right)_{ij}
\overline{\left(\frac{1}{A^\bv-z}\right)}_{ji} \right )  - \E \, F \left  ( A^\bv\to A^\bw\right )  \\
=\; - \sum_{\gamma=2}^{\gamma_{\rm max}} \pB{  \E \, F ( A^\bv\to A_{\gamma }) -  \E \, F ( A^\bv\to A_{\gamma-1 })}\,.
\end{multline}

Let $E^{(ij)}$ denote the matrix whose matrix elements are zero everywhere except
at the $(i,j)$ position, where it is 1, i.e.\ $E^{(ij)}_{kl} \deq \delta_{ik}\delta_{jl}$.
Fix $\gamma\ge 1$ and let $(b,d)$ be determined by  $\phi (b, d) = \gamma$. For definiteness, we assume the off-diagonal 
case $b \neq d$; the case $b = d$ can be treated similarly. Note that the number of diagonal terms is $N$ and 
the number of off-diagonal terms is $O(N^2)$. We shall compare $A_{\gamma-1}$ with $A_\gamma$ for each $\gamma$
and then sum up the differences in \eqref{tel}.

Note that these two matrices differ only in the entries $(b,d)$ and $(b,d)$, and they can be written as
\be\label{defHg1}
	 A_{\gamma-1} \;=\; Q + V \where V \;\deq\; (v_{bd}-\E v_{bd})E^{(bd)}+ (v_{db}-\E v_{db}) E^{(db)}\,,
\ee
and
$$
	 A_\gamma \;=\; Q + W \where W \;\deq\; (w_{bd}-\E w_{bd})E^{(bd)}+ (w_{db}-\E w _{db}) E^{(db)}\,,
$$
where the matrix $Q$ satisfies
$$
Q_{bd} \;=\; Q_{db} \;=\; f/N \;=\; \E v_{bd} \;=\; \E v_{db} \;=\; \E w _{bd} \;=\; \E w_{db}\,,
$$
where, we recall $f = 1 + \epsilon_0$.
It is easy to see that
\be\label{364d}
\max_{i,j} |v_{ij}| +  \max_{i,j} |w_{ij}| \;\leq\; (\log N)^{C\xi} q^{-1}
\ee
\hp{\xi}{\nu}, and that
\be \label{moment bounds on v and w}
\E v_{ij} \;=\; \E w_{ij} \;=\; 0\,, \qquad \E (v_{ij})^2 \;=\; \E (w_{ij})^{2} \;\leq\; C/N\,,\qquad \E |v_{ij}|^k + \E 
|w_{ij}|^k \;\leq\; C N^{-1}q^{2-k}
\ee
for $k = 2,3,4,5,6$.

We define the  Green functions
\be\label{defG}
		R \;\deq\; \frac{1}{Q-z}\,, \qquad S \;\deq\; \frac{1}{A_{\gamma-1}-z}\,, \qquad T \;\deq\; 
\frac{1}{A_{\gamma}-z}\,.
\ee
We now claim that the estimate \eqref{basic} holds for the Green function $R$ as well, i.e.
\be
	\max_{1 \le k,l \le N}  \max_{E\in I_\e}\big| R_{k l}(E+\ii\eta) -\delta_{kl}m_{sc}(E+\ii \eta)
\big| \;\le\; p^{-1 }
\label{defOMR}
\ee
holds \hp{\xi}{\nu}. 
To see this, we use the 
resolvent expansion 
\be\label{relRS}
 R \;=\; S  +  SV S +  (SV)^2 S+ \ldots +  (SV)^9S+
 (SV)^{10} R.
\ee
Since $V$ has only at most two nonzero elements, when
computing the entry $(k,l)$ of this matrix identity,
each term is a sum of finitely many terms (i.e.\ the number of summands
is independent of $N$) that  involve
matrix elements of $S$ or $R$ and $v_{ij}$, e.g.\ of the form  $(SVS)_{kl} =S_{ki} v_{ij} S_{jl}
+ S_{kj} v_{ji} S_{il}$.  Using the bound \eqref{basic} for the $S$ matrix elements,
the bound \eqref{364d}    for $v_{ij}$ and the trivial bound $|R_{ij}| \le  \eta^{-1}\leq N$, we get \eqref{defOMR}.

Having introduced these notations, we may now give an outline of the proof of Proposition \ref{GFCT}. We have to 
estimate each summand of the telescopic sum \eqref{tel} with $b \neq d$ (the generic case) by $o(N^{-2})$; in the 
non-generic case $b = d$, a bound of size $o(N^{-1})$ suffices. For simplicity, assume that we are in the generic case 
$b \neq d$ and that $F$ has only one argument.  Fix $z = E + \ii \eta$, where $E \in I_\e$ (see \eqref{def Ie}) and 
$\eta \deq N^{-2/3 - \epsilon}$.  Define
\begin{equation} \label{def of yS}
y^S \;\deq\; \eta^2 \sum_{i \neq j} S_{ij}(z) \ol{S_{ji}(z)}\,;
\end{equation}
the random variable $y^R$ is defined similarly. We shall show that
\begin{equation} \label{outline step 1}
\E F(y^S) \;=\; B + \E F(y^R) + O(N^{-1/3 + C \epsilon} p^{-4} q^{-1})\,,
\end{equation}
for some deterministic $B$ which depends only on the law of $Q$ and the first two moments of $v_{bd}$.  From 
\eqref{outline step 1} we immediately conclude that \eqref{maincomp} holds. In order to prove \eqref{outline step 1}, we 
expand
\begin{equation} \label{outline step 2}
F(y^S) - F(y^R) \;=\; F'(y^R) (y^S - y^R) + \frac{1}{2} F''(y^R) (y^S - y^R)^2 + \frac{1}{6} F'''(\zeta) (y^S - 
y^R)^3\,,
\end{equation}
where $\zeta$ lies between $y^S$ and $y^R$. Next, we apply the resolvent expansion
\begin{equation} \label{resolvent exp in sketch}
S \;=\; R  +  RV R +  (RV)^2 R+ \ldots +  (RV)^m R+
 (RV)^{m+1} S
\end{equation}
to each factor $S$ in \eqref{outline step 2} for some $m \geq 2$. Here we only concentrate on the linear term in 
\eqref{outline step 2}. The second term is dealt with similarly. (The rest term in \eqref{outline step 2} requires a 
different treatment because $F'''(\zeta)$ is not independent of $v_{bd}$. It may however by estimated cheaply using a 
naive power counting.) By definition, $Q$ is independent of $v_{bd}$, and hence $F'(y^R)$ and $R$ are independent of 
$h_{bd}$.  Therefore the expectations of the first and second order terms (in the variable $v_{bd}$) in $\E F'(y^R) 
(y^S - y^R)$ are put into $B$. The third order terms in $\E F'(y^R) (y^S - y^R)$ are bounded, using a naive power 
counting, by
\begin{equation} \label{naive power counting bound}
\eta^2 N^2 p^{-3} \E \abs{v_{bd}}^3 \;\leq\; N^{-4/3} N^2 p^{-3} N^{-1} q^{-1}\,.
\end{equation}
Here we used that, thanks to the assumption $i \neq j$, in the generic terms $\{i,j\} \cap \{b,d\} = \emptyset$ there 
are at least three off-diagonal matrix elements $R$ in the resolvent expansion of \eqref{def of yS}. Indeed, since $b,d 
\notin \{i,j\}$, the terms of order greater than one in \eqref{resolvent exp in sketch} have at least two off-diagonal 
resolvents matrix elements, and other factor in \eqref{def of yS} has at least one off-diagonal resolvent matrix element 
since $i \neq j$.  Thus we get a factor $p^{-3}$ by \eqref{basic} (the non-generic terms are suppressed by a factor 
$N^{-1}$).  Note that the bound \eqref{naive power counting bound} is still too large compared to $N^2$, since $p \geq 
N^{-1/3}$.  The key observation to solve this problem is that the expectation of the leading term is much smaller than 
its typical size; this allows us to gain an additional factor $p^{-1}$. A similar observation was used in 
\cite{EYYrigidity}, but in the present case this estimate (performed in Lemma \ref{lem: 52} below) is substantially 
complicated by the non-vanishing expectation of the entries of $A$. Much of the heavy notation in the following argument 
arises from the need to keep track of the
non-generic terms, which have fewer off-diagonal elements than the generic terms, but have a smaller entropy factor.
The improved bound on the difference $\E F'(y^R)(y^S - y^R)$ is
\begin{equation*}
N^{-4/3} N^2 p^{-4} N^{-1} q^{-1} \;=\; N^{-1/3} p^{-4} q^{-1}\,,
\end{equation*}
which is much smaller than $N^{-2}$ provided that $q \geq N^{\phi}$ for $\phi > 1/3$ and $\epsilon$ is small enough.

The key step to the proof of Proposition \ref{GFCT} is the following lemma.

\begin{lemma}\label{lemGamma} Fix an index $\gamma = \phi(b,d)$ and recall the definitions of $Q$, $R$ and $S$ from 
\eqref{defG}.  For any small enough $\e>0$ and
under the assumptions in Proposition \ref{GFCT}, there exists a constant $C$ depending on $F$ (but independent of 
$\gamma$) and constants $B_N$ and $D_N$, depending  on the law $\law(Q)$ of the Green function $Q$ and on the second 
moments $m_2(v_{bd})$ of $v_{bd}$, such that, for large enough $N$ (independent of $\gamma$) we have
\begin{multline}\label{c3f} \Bigg|\E  F\left( \eta  \int_{E_1}^{E_2} \rd y  \sum_{i\neq 
j}S_{ij}\overline{S}_{ji}(y+\ii\eta)\right)
- \E  F\Bigg( \eta  \int_{E_1}^{E_2} \rd y  \sum_{i\neq j}R_{ij}\overline{R}_{ji}(y+\ii\eta)\Bigg)
- B_N\big(m_2(v_{bd}), \law(Q) \big)  \Bigg|
\\
\leq\; N^{{\bf1}(b=d)-5/3+C\e}q^{-1}\,,
\end{multline}
where, we recall, $\eta=N^{-2/3-\e}$, as well as
\be\label{new612}
\left|\E \, F \left (\eta^2 \sum_{i\neq j}S_{ij}\overline{S}_{ji}(z) \right ) -
\E \, F \left (\eta^2 \sum_{i\neq j}R_{ij}\overline{R}_{ji}(z) \right )-
 D_N\big(m_2(v_{bd}), \law(Q)\big)\right| \;\leq\; N^{{\bf1}(b=d)-5/3+C\e}q^{-1}\,,
\ee
where $z=E+\ii\eta$. The constants $B_N$ and $D_N$ may also depend on $F$, but they depend on the centered random 
variable $v_{bd}$ only through its second moments.

\end{lemma}

Assuming Lemma \ref{lemGamma}, we now complete the proof of Proposition \ref{GFCT}.

\begin{proof}[Proof of Proposition \ref{GFCT}]
Clearly, Lemma \ref{lemGamma} also holds if $S$ is replaced by $T$. Since $Q$ is independent of $v_{bd}$ and
$w_{bd}$, and $m_2(v_{bd})=m_2(w_{bd})=1/N$, we have $D_N\big(m_2(v_{bd}), \law(Q)\big)
= D_N\big(m_2(w_{bd}), \law(Q)\big)$.
Thus we get from Lemma \ref{lemGamma} that 
\be\label{RSest}
  \left|\E\, F \left (\eta^2 \sum_{i\neq j}S_{ij}\overline{S}_{ji}(z) \right ) -
  \E\, F \left (\eta^2 \sum_{i\neq j}T_{ij}\overline{T}_{ji}(z) \right ) \right|
  \le CN^{{\bf1}(b=d)-5/3+C\e}q^{-1}.
\ee
Recalling the definitions of $S$ and $T$ from \eqref{defG},
the bound \eqref{RSest} compares
the expectation of a function of the resolvent of $A_\gamma$ and that of $A_{\gamma-1}$.
 The telescopic summation in \eqref{tel} then implies
  \eqref{c1new},
since the number of summands with $b \ne d$ is of order $N^2$ but the number
of summands with $b=d$ is only $N$. Similarly, \eqref{c3f} implies \eqref{c1}. This completes the proof.
\end{proof}

\begin{proof}[Proof of Lemma \ref{lemGamma}]
Throughout the proof we abbreviate $A_{\gamma - 1} = A = (a_{ij})$ where $a_{ij} = h_{ij} + f / N$.
We shall only prove the more complicated case \eqref{c3f}; the proof of \eqref{new612} is similar. In fact, we shall 
prove the bound
\begin{multline}\label{c3}
\Bigg|\E  F\left( \eta  \int_{E_1}^{E_2} \rd y  \sum_{i\neq j}S_{ij}\overline{S}_{ji}(y+\ii\eta)\right)
- \E  F\Bigg( \eta  \int_{E_1}^{E_2} \rd y  \sum_{i\neq j}R_{ij}\overline{R}_{ji}(y+\ii\eta)\Bigg)
- B_N\big(m_2(h_{bd}), \law(Q) \big)  \Bigg|
\\
\leq\; CN^{{\bf1}(b=d)-1/3+C\e}p^{-4}q^{-1}\,,
\end{multline}
from which \eqref{c3f} follows by \eqref{def of p}.

From \eqref{basic} we get
\be\label{defOMS}
  \max_{1 \le k,l \le N}  \max_{E\in I_\e}\big| S_{k l}(E+\ii\eta) -\delta_{kl}m_{sc}(E+\ii \eta)
\big|
\;\le\;
p^{-1}
\ee
\hp{\xi}{\nu}. Define $\Omega$ as the event on which \eqref{defOMS}, \eqref{defOMR}, and \eqref{364d} hold. We have 
proved that $\Omega$ holds \hp{\xi}{\nu}.  Since the arguments of $F$ in \eqref{c3}
are bounded by $CN^{2+2\e}$ and $F(x)$ increases at most
polynomially, it is easy to see that the contribution
of the event $\Omega^c$
to the expectations  in \eqref{c3} is negligible.

Define $x^S$ and $x^R$ by 
\be\label{defxs}
x^{S} \;\deq\;  \eta  \int_{E_1}^{E_2} \rd y \sum_{i\neq j}S_{ij}\overline{S}_{ji}(y+\ii\eta)\,, \qquad x^{R} \;\deq\; 
\eta  \int_{E_1}^{E_2} \rd y \sum_{i\neq j}R_{ij}\overline{R}_{ji}(y+\ii\eta),
\ee
and decompose  $x^S$ into three parts
\be
x^{S} \;=\; x^{S}_2+x^S_1+x^S_0 \where x^S_k \;\deq\;  \eta  \int_{E_1}^{E_2} \dd y \, \sum_{i\neq j} \indb{|\{i,j\}\cap 
\{b,d\}|=k } \, S_{ij}\overline{S}_{ji}(y+\ii\eta)\,;
\ee
$x^R_k$ is defined similarly.  Here $k=|\{i,j\}\cap \{b,d\}|$ is the number of times the indices $b$ and $d$ appear 
among the summation indices $i,j$. Clearly $k=0$,  $1$ or $2$.  The number of the  terms in the sum of the definition of 
$x^S_k$ is $O(N^{2-k})$. A resolvent expansion yields
\be\label{SR-N}
	 S \;=\;  R -   RVR+  (RV)^2R -    (RV)^3R+   (RV)^4R-  (RV)^5R+   (RV)^6S\,.
\ee
In the following formulas we shall, as usual, omit the spectral parameter from the notation
of the resolvents. The spectral parameter is always $y+\ii\eta$
with $y\in [E_1,E_2]$; in particular, $y \in I_\e$.

If $|\{i,j\}\cap \{b,d\}|=k$, recalling that $i \neq j$ we find that there are at least $2 - k$ off-diagonal resolvent 
elements in $\big[(RV)^m R\big]_{ij}$, so that \eqref{basic} yields in $\Omega$
\be\label{51}
\big| \big[(RV)^m R\big]_{ij}\big| \;\leq\; C_m  \left(N^{\e}q^{-1}\right)^m p^{- (2-k)} \where  m \in \N_+\,, \quad 
m\leq 6\,, \quad k = 0,1, 2\,.
\ee
Similarly, we have in $\Omega$
\be\label{666}
\big| \big[(RV)^m S\big]_{ij}\big| \;\leq\; C_m  \left(N^{\e}q^{-1}\right)^m p^{- (2-k)} \where  m \in \N_+\,, \quad 
m\leq 6\,, \quad k = 0,1, 2\,.
\ee
Therefore we have in $\Omega$ that
\be
 |x^S_{k}-x^R_{k}| \;\leq\; C N^{ 2/3-k}p^{-(3-k)} N^\e q^{-1} \for k = 0,1, 2\,,
 \ee
where the factor $N^{ 2/3-k}$ comes from $\sum_{i\neq j}$, $\eta$ and $\int \dd E$. Inserting these bounds into  the 
Taylor expansion of $F$, using \be
\label{pqrel}
q \;\geq\; N^\phi \;\geq\; N^{1/3+C\e} \;\geq\; p \;\geq\; N^{1/3-2\e}
\ee
and keeping only the terms larger than $O(N^{-1/3+C\e}p^{-4}q^{-1})$,
we obtain
\begin{multline} \label{FF123}
\left|\E \pb{ F(x^S)- F(x^R)} -\E\left( F'(x^R) (x^S_0-x^R_0) +\frac12 F''(x^R)  (x^S_0-x^R_0 )^2+ 
F'(x^R)(x^S_1-x^R_1)\right) \right|
\\
\leq\; CN^{-1/3+C\e}p^{-4}q^{-1}\,,
\end{multline}
where we used the remark after \eqref{defOMS} to treat the contribution on the
event $\Omega$.
Since there is no $x_2$ appearing in \eqref{FF123}, we can focus on the cases $k=0$ and $k = 1$. 

To streamline the notation, we introduce
\be\label{defRMij}
R ^{(m)}_{ij} \;\deq\; (-1)^m \big[(RV)^m R\big]_{ij}\,.
\ee
Then using \eqref{51} and the estimate $\max_{i\neq j}|R_{ij}|\leq p^{-1}$ we get
\be\label{382m}
\absb{R ^{(m)}_{ij}} \;\leq\; C_m   \left(N^{\e}q^{-1}\right)^m p^{- (2-k)+\delta_{0m}\delta_{0k}}\,.
\ee

Now we decompose the sum $x^S_k - x^R_k$ according to the number of matrix elements $h_{bd}$ and $h_{db}$. To that end, 
for $k \in \{0,1\}$ and $s,t \in \{0,1,2,3,4\}$ and $s + t \geq 1$, we define
\begin{align}\label{defqst}
Q^{(s,t)}_{k} \;\deq\; \eta  \int_{E_1}^{E_2} \rd y  \, \sum_{i \neq j} \indb{|\{i,j\}\cap \{b,d\}|=k} \, 
R^{(s)}_{ij}\overline{R ^{(t)}_{ji}}\,,
\end{align}
and set
\be
Q^{(\ell)}_{k} \;\deq\; \sum_{s+t=\ell} Q^{(s,t)}_{k}\,.
\ee
Using \eqref{382m} we get the estimates, valid on $\Omega$,
\be\label{byi1232}
\absb{Q^{(s,t)}_{k}} \;\leq\; C_{st} \left(N^{\e}q^{-1}\right)^{s+t} N^{2/3-k}p^{- 
(4-2k)+\delta_{0s}\delta_{0k}+\delta_{0t}\delta_{0k}} \,,
\qquad
\absb{Q^{(\ell)}_{k}} \;\leq\; C_{\ell} \left(N^{\e}q^{-1}\right)^{\ell}N^{2/3-k}p^{- (3-k)}\,,
\ee
where $\ell \geq 1$.
Using \eqref{pqrel}, \eqref{51}, and \eqref{666}, we find the decomposition
\be\label{yi1232}
x^S_k-x^R_k \;=\;  \sum_{1\leq s+t\leq 4}Q^{(s,t)}_{k}  + O(N^{-1/3+C\e}p^{-4}q^{-1}
)\,,
\ee
where $s$ and  $t$ are non-negative. By \eqref{byi1232} and \eqref{defOMR} we have for $s+t\geq 1$
\be\label{387h}
\left|\E_{bd}\,Q^{(s,t)}_{k}  \right| \;\leq\;  C_{st} q ^{2-s-t} N^{-1/3-k}p^{- 
(4-2k)+\delta_{0s}\delta_{0k}+\delta_{0t}\delta_{0k}}\,,
\ee
where $\E_{bd}$ denotes partial expectation with respect to the variable $h_{bd}$. Here we used that only terms with at
least two elements $h_{bd}$ or $h_{db}$ survive. Recalling \eqref{moment bounds on v and w}, we find that taking the
partial expectation $\E_{bd}$ improves the bound \eqref{byi1232} by a factor $q^2 / N$.
Thus we also have
\be\label{388h}
 \left|\E_{bd} \, Q^{(\ell)}_{k}  \right|\leq  C_{\ell} \, q^{2-\ell} N^{-1/3-k}p^{- (3-k)}
 \ee
Similarly, for $s+t\geq 1$ and $u+v\geq 1$ we have
\be
\left|\E_{bd}\, Q^{(s,t)}_{k} Q^{(u,v)}_{k} \right| \;\leq\; q^{2-s-t-u-v}N^{ 1/3-2k}p^{- 
(8-4k)+\delta_{0s}\delta_{0k}+\delta_{0t}\delta_{0k} +\delta_{0u}\delta_{0k}+\delta_{0v}\delta_{0k}}\,,
\ee
which implies
\be\label{390h}
\left|\E_{bd} \, Q^{(\ell_1)}_{k} Q^{(\ell_2)}_{k} \right| \;\leq\; q^{2-\ell _1-\ell _2}N^{ (1/3  -2k  ) 
+C\e}p^{-6+2k}\,.
\ee
Inserting \eqref{388h} and \eqref{390h} into the second term of the left-hand side of \eqref{FF123}, and using the 
assumption $F$ as well as \eqref{pqrel}, we find
\begin{align}
&\mspace{-100mu} \E\left( F'(x^R)(x^S_0-x^R_0) + F'(x^R)(x^S_1-x^R_1) +
\frac12 F''(x^R)(x^S_0-x^R_0)^2\right)
\notag
\\
& =\; B +  \E F'(x^R)   Q_0^{(3)}   +  O\left(N^{-1/3+C\e}p^{-4}q^{-1}\right)
\notag
\\ \label{6755}
& =\; B +  \E F'(x^R)  \left( Q_0^{(0,3)}+ Q_0^{(3,0)} \right) +  O\left(N^{-1/3+C\e}p^{-4}q^{-1}\right)\,,
\end{align}
where we defined
\begin{align}
B  \;\deq &\;\; \E\left( \sum_{k=0, 1}  F'(x^R) \pB{Q_k^{(1)}+Q_k^{(2)}} +\frac12 F''(x^R) \pB{Q_0^{(1)}}^2  \right)
\notag \\ \label{Adef}
\;=&\;\;  \E\left( \sum_{k=0, 1}  F'(x^R) \E_{bd} \pB{Q_k^{(1)}+Q_k^{(2)}} +\frac12 F''(x^R) \E_{bd} \pB{Q_1^{(0)}}^2  
\right)\,.
\end{align}
Note that $B$ depends on $h_{bd}$ only through its expectation (which is zero) and on its
second moment. Thus, $B$ will be $B_N(m_2(v_{bd}), \law(Q))$ from \eqref{c3f}.

In order to estimate \eqref{6755}, it only remains to estimate $ \E F'(x^R)  Q_0^{(0,3)}$ and  $ \E F'(x^R)  
Q_0^{(3,0)}$. Using \eqref {387h}, \eqref{FF123}, \eqref{6755}, and \eqref{388h}, we have
\be
\absB{\E  \pb{F(x^S)- F(x^R)} -B} \;\leq\; N^{-1/3+C\e}p^{-3}q^{-1}\,,
\ee
which implies \eqref{c3} in the case $b=d$.

Let us therefore from now on assume $b \neq d$. Since we estimate $Q_0^{(3,0)}$ and  $Q_0^{(0,3)}$, this implies that 
$i,j,b,d$ are all distinct.
In order to enforce this condition in sums, it is convenient to introduce the indicator function $\chi \equiv 
\chi(i,j,b,d) \deq \indb{|\{i,j,b,d\}| \;=\; 4}$.

Recalling \eqref{defRMij}, we introduce the notation $R_{ij}^{(m, s)}$ to denote the sum of the terms in the definition
\eqref{defRMij}
of $R_{ij}^{(m )}$ in which the number of the  off-diagonal elements of $R$ is  $s$. For example,
\be
R_{ij}^{(3, 0)}=R_{ij}^{(3, 1)} \;=\;0\,, \qquad R_{ij}^{(3, 2)} \;=\; 
R_{ib}h_{bd}R_{dd}h_{db}R_{bb}h_{bd}R_{dj}+R_{id}h_{db}R_{bb}h_{bd}R_{dd}h_{db}R_{bj}\,.
\ee
Then in the case $\chi = 1$ we have
\be
R_{ij}^{(3)}=\sum_{s=2}^4R_{ij}^{(3, s)}
\ee
Now from the definition \eqref{defqst} we get
\be{\label{397b} }
Q^{(0,3)}_{0}  \;=\; \sum_{s=2}^4Q^{(0,3, s)}_{0} \where Q^{(0,3, s)}_{0} \;\deq\; \eta  \int_{E_1}^{E_2} \rd y  
\sum_{i,j} \chi \,
 R^{(0)}_{ij}\overline{R ^{(3,s)}_{ji}}\,,
\ee
and
\be{\label{397c} }
Q^{(3,0)}_{0} \;=\; \sum_{s=2}^4Q^{(3,0, s)}_{0} \where Q^{(3,0, s)}_{0} \;\deq\; \eta  \int_{E_1}^{E_2} \rd y  
\sum_{i,j}  \chi \,
 R ^{(3,s)}_{ij}\overline{R^{(0)}_{ji}}\,.
\ee

As above, it is easy to see that for $s \geq 3$ we have
\be
 \E F'(x^R)  Q_0^{(0,3,s) } \;\leq\; N^{-1/3+C\e}p^{-4}q^{-1}\,,
 \ee
which implies, using \eqref{6755},
 \be
 \left|\E  \pb{ F(x^S)- F(x^R)} -B\right| \;\leq\;   \E F'(x^R)  Q_0^{(0,3,2) }+ \E F'(x^R)  Q_0^{(3,0,2) }+ 
N^{-5/3+C\e} q^{-1}\,.
 \ee
By symmetry, it only remains to prove that
 \be
 \E F'(x^R)  Q_0^{(3,0,2) } \;\leq\; N^{-1/3+C\e}p^{-4}q^{-1}\,.
  \ee
Using the definition \eqref{397b} and the estimate \eqref{defOMR} to replace some diagonal resolvent matrix elements
with $m_{sc}$, we find
\begin{align}
\E F'(x^R)Q_0^{(3,0,2)}  & \;=\; \eta \int_{E_1}^{E_2} \rd y\,  \E F'(x^R)
\sum_{i,j}  \chi \, \Bigg[   R_{ib}h_{bd}R_{dd}h_{db}R_{bb}h_{bd}R_{dj}\overline{R_{ji}} +(b \leftrightarrow d)\Bigg]
\notag \\
& \;=\; \eta  \int_{E_1}^{E_2} \rd y \,  \E F'(x^R)
 \sum_{i,j} \chi\,\Bigg[m_{sc}^2 R_{ib} R_{dj}\overline{R_{ji}}\left(\E_{bd}|h_{bd}|^2h_{bd}\right)
(b\leftrightarrow d)\Bigg] 
\notag \\ \label{Q30}
&\mspace{40mu}
+ O(N^{-1/3 + C \e} p^{-4} q^{-1})\,,
 \end{align}
where we used the estimate $\absb{\E_{bd}|h_{bd}|^2h_{bd}}\leq \frac{C}{Nq}$ to control the errors for the replacement.
Combining \eqref{Q30} with \eqref{6755} and \eqref{FF123}, we therefore obtain
\begin{multline}\label{680}
\left|\E  [ F(x^S)- F(x^R) ] -B \right|
\;\leq\;  C N^{-1/3+C\e}p^{-4}q^{-1}
+Cq^{-1}N^{-1/3+C\e} \\
{}\times{}
\max_{y\in I_\e}\max_{i,j} \, \chi\,
\Big[
\big|\E F'(x^R)  R_{ij}\overline{R_{jb}R_{di}}\big|
+\big| \E F'(x^R) R_{ib}R_{dj}\overline{R_{ji}}\big| +(b\leftrightarrow d)\Big],
\end{multline}
where we used the trivial bounds on $F'$ and $\abs{m_{sc}}$,
and that every estimate is uniform in $y$.

In order to complete the proof of Lemma \ref{lemGamma}, we need to estimate the expectations in \eqref{680} by a better 
bound than the naive high-probability bound on the argument of $\E$.  This is accomplished in Lemma \ref{lem: 52} 
below.  From Lemma \ref{lem: 52} and \eqref{680} we get in the case $b \neq d$ that
\be\label{FF1232}
\absb{\E  [ F(x^S)- F(x^R) ] - B} \;\leq\; N^{-1/3+C\e}p^{-4}q^{-1}\,,
\ee
where $B$ was defined in \eqref{Adef}. This completes the proof of Lemma \ref{lemGamma}.
\end{proof}

\begin{lemma} \label{lem: 52} Under the assumptions of Lemma \ref{lemGamma}, in particular fixing $f = f_*$,
and assuming that $a, b, i, j$ are all distinct,
we have
\be\label{new628} \max_{y\in I_\e} \absb{\E F'(x^R)R_{ib}R_{dj}\overline{R_{ji}}(y+\ii\eta)} \;\leq\; Cp^{-4}\,.
\ee
The same estimate holds for the other three terms on the right-hand side of \eqref{680}.
\end{lemma}

In order to prove Lemma \ref{lem: 52}, we shall need the following result, which is crucial when estimating terms 
arising from the nonvanishing expectation of $\E a_{ij} = f N^{-1}$. Before stating it, we introduce some notation.

Recall that we set $A \equiv A_{\gamma - 1} = (a_{ij})$, where the matrix entries are given by $a_{ij} = h_{ij} + f / N$ 
and $\E h_{ij} = 0$. We denote by $A^{(b)}$ the matrix obtained from $A$ by setting all entries with index $b$ to zero, 
i.e.\ $(A^{(b)})_{ij} \deq \ind{i \neq b} \ind{j \neq b} a_{ij}$. If $Z \equiv Z(A)$ is a function of $A$, we define 
$Z^{(b)} \deq Z(A^{(b)})$. See also Definitions 5.2 and 3.3 in \cite{EKYY}. We also use the notation $\E_b$ to denote 
partial expectation with respect to all variables $(a_{1b}, \dots, a_{Nb})$ in the $b$-th column of $A$.

\begin{lemma} \label{non-centred Z-lemma}
For any fixed $i$ we have, \hp{\xi}{\nu},
\begin{equation*}
\absbb{\frac{1}{N} \sum_{k \neq i} \sum_{l \neq k} S_{il}^{(k)} h_{lk}} \;\leq\; p^{-2}\,.
\end{equation*}
\end{lemma}
\begin{proof}
The claim is an immediate consequence of Proposition 7.11 in \cite{EKYY} and the observation that, for $E \in 
I_\epsilon$, $\eta = N^{-2/3 - \epsilon}$, and $q \geq N^\phi$ we have
\begin{equation*}
(\log N)^C \pBB{\frac{1}{q} + \sqrt{\frac{\im m_{sc}}{N \eta}} + \frac{1}{N \eta}} \;\leq\; p^{-1}
\end{equation*}
for large enough $N$.
\end{proof}

Another ingredient necessary for the proof of Lemma \ref{lem: 52} is the following resolvent identity.

\begin{lemma} \label{lemma: 376}
Let $A = (a_{ij})$ be a square matrix and set $S = (S_{ij}) = (A - z)^{-1}$. Then for $i \neq j$ we have
\begin{align} \label{sq exp formula}
S_{ij} \;=\; - S_{ii} \sum_{k \neq i} a_{ik}  S_{kj}^{(i)}\,,\qquad S_{ij} \;=\; - S_{jj} \sum_{k \neq j}  
S_{ik}^{(j)}a_{kj}\,.
\end{align}
\end{lemma}
\begin{proof}
We prove the first identity in \eqref{sq exp formula}; the second one is proved analogously. We use the resolvent 
identity
\begin{equation} \label{SijSijk}
S_{ij} \;=\; S_{ij}^{(k)} + \frac{S_{ik} S_{kj}}{S_{kk}} \qquad \text{for } i,j \neq k
\end{equation}
from \cite{EKYY}, (3.8).
Without loss of generality we assume that $z = 0$. Then \eqref{SijSijk} and the identity $A S = \umat$ yield
\begin{equation*}
\sum_{k \neq i} a_{ik} S_{kj}^{(i)} \;=\; \sum_{k \neq i} a_{ik} S_{kj} - \sum_{k \neq i} a_{ik}
\frac{S_{ki} S_{ij}}{S_{ii}}
\;=\; - a_{ii} S_{ij} - \frac{S_{ij}}{S_{ii}} (1 - a_{ii} S_{ii}) \;=\; \frac{S_{ij}}{S_{ii}}\,. \qedhere
\end{equation*}
\end{proof}

Armed with Lemmas \ref{non-centred Z-lemma} and \ref{lemma: 376}, we may now prove Lemma \ref{lem: 52}.

\begin{proof}[Proof of Lemma \ref{lem: 52}]
With the relation between $R$ and $S$ in \eqref{relRS} and \eqref{51}, we find that \eqref{new628} is implied by 
\be\label{new629}
\absb{\E F'(x^S)S_{ib}S_{dj}\overline{S_{ji}}} \;\leq\; Cp^{-4}N^{C\e}, \ee under the assumption that $b,d,i,j$ are all 
distinct.
This replacement is only a technical convenience
when we apply a large deviation
estimate below.

Recalling the definition of $\Omega$  after \eqref{defOMS}, we get using \eqref{SijSijk}
\be\label{GiiGjiif}
 |S_{ij}-S^{(b)}_{ij}| \;=\; \Big | S_{ib}S_{bj}(S_{bb})^{-1} \Big | \;\le\; Cp^{-2} \qquad \text{in } \Omega\,.
\ee
This yields
\be\label{c5}
|x^{S} - ( x^S)^{(b)}| \;\le\; p^{-1}N^{ C \e} \qquad \text{in }\Omega\,.
\ee
Similarly, we have
\be\label{Sbi}
\left|S_{ib}S_{dj}\overline{S_{ji}}-S_{ib}S^{(b)}_{dj}\overline{S^{(b)}_{ji}}\right|
\;\leq\; C p^{-4} \qquad \text{in } \Omega\,.
\ee
Hence by assumption on $F$ we have
\be\label{632}
|\E F'(x^S)S_{ib} S_{dj} \ol{S_{ji}}| \;\leq\; \left|\E \pB{F' \pb{( x^S)^{(b)}}} 
S_{ib}S^{(b)}_{dj}\overline{S^{(b)}_{ji}}\right|+O\left(p^{-4}N^{ C\e}\right).
\ee
Since $(x^S)^{(b)}$ and $S_{dj}^{(b)} \ol{S_{ji}^{(b)}}$ are independent of the $b$-th row of $A$, we find from 
\eqref{632} that \eqref{new629}, and hence \eqref{new628}, is proved if we can show that
\begin{equation} \label{Eb S}
\E_b S_{ib} \;=\; O(p^{-2})
\end{equation}
for any fixed $i$ and $b$.

What remains therefore is to prove \eqref{Eb S}.
Using \eqref{defOMS} and \eqref{GiiGjiif} we find in $\Omega$ that
\be\label{SSS}
S_{bb} \;=\; m_{sc} + O(p^{-1})\,,
\qquad S_{ik}^{(b)}\;=\; O(p^{-1})\,.
\ee
Using $a_{kb} = h_{kb} + f / N$ we write
\begin{equation} \label{SmS1}
S_{ib} \;=\; - m_{sc} \sum_{k \neq b} S^{(b)}_{ik} \pbb{h_{kb} + \frac{f}{N}} - (S_{bb} - m_{sc}) \sum_{k \neq b} 
S^{(b)}_{ik} \pbb{h_{kb} + \frac{f}{N}}\,.
\end{equation}
By \eqref{SSS} and the large deviation estimate (3.15) in \cite{EKYY}, the second sum in \eqref{SmS1} is bounded by 
$O(p^{-1})$ \hp{\xi}{\nu}. Therefore, using \eqref{SSS} and $\E_b h_{kb} = 0$, we get
\begin{equation} \label{EbSibeq}
\E_b S_{ib} \;=\; \frac{- m_{sc} f}{N} \sum_{k \neq b} S_{ik}^{(b)} + O(p^{-2})
 \;=\; \frac{- m_{sc} f}{N} \sum_{k \neq i} S_{ik} + O(p^{-2})\,,
\end{equation}
where in the second step we used \eqref{SijSijk}.

In order to estimate the right-hand side of \eqref{EbSibeq}, we introduce the quantity
\begin{equation*}
X \;\deq\; \frac{1}{N} \sum_{k \neq i} \E_k S_{ik}\,.
\end{equation*}
Note that $X$ depends on the index $i$, which is omitted from the notation as it is fixed.
Using \eqref{sq exp formula}, \eqref{SSS}, and \eqref{SijSijk} as above, we find \hp{\xi}{\nu}
\begin{align*}
X &\;=\; \frac{- m_{sc}}{N} \sum_{k\neq i} \sum_{l \neq k} \E_k \, S_{il}^{(k)} \pbb{h_{lk} + \frac{f}{N}} + O(p^{-2})
\\
&\;=\;
\frac{- m_{sc} f}{N^2} \sum_{k \neq i} \sum_{l \neq k} S_{il}^{(k)} + O(p^{-2})
\\
&\;=\;
\frac{- m_{sc} f}{N} \sum_{l \neq i} S_{il} + O(p^{-2})
\\
&\;=\; -m_{sc} f \, X + O \pbb{\frac{1}{N} \sum_{l \neq i} \pb{S_{il} - \E_l S_{il}}} + O(p^{-2})\,.
\end{align*}
Now recall that the spectral parameter $z = E + \ii \eta$ satisfies $E \in I_\epsilon$ (see \eqref{def Ie}) and $\eta = 
N^{-2/3 - \epsilon}$.  Therefore \eqref{formula for m sc} implies that $m_{sc}(z) = -1 + o(1)$. Recalling that $f = 1 + 
\epsilon_0$, we therefore get, \hp{\xi}{\nu},
\begin{equation} \label{bound on X}
X \;=\; O \pbb{\frac{1}{N} \sum_{l \neq i} \pb{S_{il} - \E_l S_{il}}} + O(p^{-2})\,.
\end{equation}

We now return to \eqref{EbSibeq}, and estimate, \hp{\xi}{\nu}
\begin{align*}
\E_b S_{ib} \;=\; \frac{- m_{sc} f}{N} \sum_{k \neq i} S_{ik} + O(q^{-2})
\;=\; -m_{sc} f X + O\pbb{\frac{1}{N} \sum_{k \neq i} \pb{S_{ik} - \E_k S_{ik}}} + O(p^{-2})\,.
\end{align*}
Together with \eqref{bound on X} this yields, \hp{\xi}{\nu},
\begin{equation} \label{bound on Sib 1}
\E_b S_{ib} \;=\; O \pbb{\frac{1}{N} \sum_{k \neq i} \pb{S_{ik} - \E_k S_{ik}}} + O(p^{-2})\,.
\end{equation}
In order to estimate the quantity in parentheses, we abbreviate $\IE_k Z \deq Z - \E_k Z$ for any random variable $Z$ 
and write, using \eqref{sq exp formula},
\begin{align*}
\frac{1}{N} \sum_{k \neq i} \pb{S_{ik} - \E_k S_{ik}} &\;=\; \frac{-1}{N} \sum_{k \neq i} \sum_{l \neq k} \IE_k \, 
S_{kk} S_{il}^{(k)} a_{lk}
\\
&\;=\; \frac{-m_{sc}}{N} \sum_{k \neq i} \sum_{l \neq k} \IE_k \, S_{il}^{(k)} \pbb{h_{lk} + \frac{f}{N}} - \frac{1}{N} 
\sum_{k \neq i} \IE_k (S_{kk} - m_{sc}) \sum_{l \neq k} S_{il}^{(k)} \pbb{h_{lk} + \frac{f}{N}}\,.
\end{align*}
Using the large deviation estimate (3.15) in \cite{EKYY}, \eqref{SSS}, and the bound $\abs{h_{lk}} \leq p^{-1}$ which 
holds \hp{\xi}{\nu} (see Lemma 3.7 in \cite{EKYY}), we find that the second term is bounded by $O(p^{-2})$ 
\hp{\xi}{\nu}.  Thus we get
\begin{equation*}
\frac{1}{N} \sum_{k \neq i} \pb{S_{ik} - \E_k S_{ik}} \;=\; \frac{- m_{sc}}{N} \sum_{k \neq i} \sum_{l \neq k} 
S_{il}^{(k)} h_{lk} + O(p^{-2})
\end{equation*}
\hp{\xi}{\nu}.
Therefore \eqref{bound on Sib 1} and Lemma \ref{non-centred Z-lemma} imply \eqref{Eb S}, and the proof is complete.
\end{proof}

\section{Universality of generalized Wigner matrices with finite moments}\label{sec4}

This section is an application of our results to the problem of universality of generalized Wigner matrices (see 
Definition \ref{def gen Wigner} below) whose entries have heavy tails. We prove the bulk universality of generalized 
Wigner matrices under the assumption that the matrix entries have a finite $m$-th moment for some $m > 4$. We also prove 
the edge universality of Wigner matrices under the assumption that $m > 12$. (This lower bound can in fact be improved to $m \geq 7$; see Remark \ref{remark: m = 7} below.)
The Tracy-Widom law for the largest eigenvalue of Wigner matrices was first proved in \cite{Sosh} under a Gaussian decay 
assumption, and was proved later in \cite{Ruz,  TV2, EYYrigidity, Kho} under various weaker restrictions
on the distributions of the matrix elements. In particular, in \cite{Kho} the Tracy-Widom law was proved for entries with symmetric distribution and $m > 12$.
In \cite{J2} similar results were derived for complex Hermitian Gaussian divisible matrices, where the GUE component is of order one. For this case it is proved in \cite{J2} that bulk universality holds provided the entries of the Wigner component have finite second moments, and edge universality holds provided they have finite fourth moments.

\begin{definition} \label{def gen Wigner}
We call a Hermitian or real symmetric random matrix $H = (h_{ij})$ a \emph{generalized Wigner matrix} if the two 
following conditions hold. First, the family of upper-triangluar entries $( h_{ij} : i \leq j)$ is independent.  Second, 
we have
\begin{equation*}
\E h_{ij} \;=\; 0\,, \qquad \E \abs{h_{ij}}^2 \;=\; \sigma_{ij}^2 \,,
\end{equation*}
where the variances $\sigma_{ij}^2$ satisfy
\begin{equation*}
\sum_{j} \sigma_{ij}^2 \;=\; 1\,, \qquad C_- \leq \inf_{i,j} (N \sigma_{ij}^2) \;\leq\; \sup_{i,j} (N \sigma_{ij}^2) 
\;\leq\; C_+\,,
\end{equation*}
and $0<C_- \leq C_+ < \infty$ are constants independent of $N$.
\end{definition}

%

\begin{theorem}[Bulk universality] \label{theorem: bulk universality for Wigner}
Suppose that $H = (h_{ij})$ satisfies Definition \ref{def gen Wigner}. Let $m > 4$ and assume that for all $i$ and $j$ we have
\begin{equation} \label{assumptions for 4+e Wigner}
\E \absb{h_{ij} / \sigma_{ij}}^m \;\leq\; C_m\,,
\end{equation}
for some constant $C_m$, independent of $i$, $j$, and $N$.

Let $n \in \N$ and $O : \R^n \to \R$ be compactly supported and continuous.
Let $E$ satisfy $-2 < E < 2$ and let $\epsilon > 0$. Then for any sequence $b_N$ satisfying $N^{-1 + \epsilon} \leq b_N 
\leq \abs{\abs{E} - 2} / 2$ we have
\begin{multline*}
\lim_{N \to \infty} \int_{E - b_N}^{E + b_N} \frac{\dd E'}{2 b_N} \int \dd \alpha_1 \cdots \dd \alpha_n\, O(\alpha_1, 
\dots, \alpha_n)
\\
{}\times{} \frac{1}{\varrho_{sc}(E)^n}\pb{p_{N}^{(n)} - p_{{\rm G}, N}^{(n)}} \pbb{E' +
\frac{\alpha_1}{N\varrho_{sc}(E)}, \dots, E' + \frac{\alpha_n}{N\varrho_{sc}(E)}} \;=\; 0\,.
\end{multline*}
Here $\varrho_{sc}$ was defined in \eqref{def rho sc}, $p_N^{(n)}$ is the $n$-point marginal of the eigenvalue 
distribution of $H$, and $p^{(n)}_{{\rm G},N}$ the $n$-point marginal of an $N \times N$ GUE/GOE matrix.
\end{theorem}

\begin{theorem}[Edge universality] \label{thm: edge Wigner}
Suppose that $H^{\f v} = (h^{\f v}_{ij})$ and $H^{\f w} = (h^{\f w}_{ij})$ both satisfy Definition \ref{def gen Wigner}.  Assume that 
the entries of $H^{\f v}$ and $H^{\f w}$ all satisfy \eqref{assumptions for 4+e Wigner} for some $m > 12$, and that the 
two first two moments of the entries of $h^{\f v}_{ij}$ and $h^{\f w}_{ij}$ match:
\begin{equation*}
\E^{\f v} (\ol{h_{ij}^{\f v}})^l (h_{ij}^{\f v})^u \;=\;
\E^{\f w} (\ol{h_{ij}^{\f w}})^l (h_{ij}^{\f w})^u \for 0 \leq l + u \leq 2\,.
\end{equation*}
Then there is a $\delta > 0$ such that for any $s \in \R$ we have
\begin{equation}\label{tw for Wigner}
\P^{\f v} \pB{N^{2/3}(\lambda_N - 2) \leq s - N^{-\delta}} - N^{- \delta} \;\leq\; \P^{\f w} \pb{N^{2/3}(\lambda_N - 2) 
\leq s}
\;\leq\;
\P^{\f v} \pB{N^{2/3}(\lambda_N - 2) \leq s + N^{-\delta}} + N^{- \delta}\,.
\end{equation}
Here $\P^{\f v}$ and $\P^{\f w}$ denote the laws of the ensembles $H^{\f v}$ and $H^{\f w}$ respectively, and 
$\lambda_N$ denotes the largest eigenvalue of $H^{\f v}$ or $H^{\f w}$.
\end{theorem}

\begin{remark}
A similar result holds for the smallest eigenvalue $\lambda_1$. Moreover, a result analogous to \eqref{tw for Wigner} 
holds for the $n$-point joint distribution functions of the extreme eigenvalues. (See \cite{EYYrigidity}, Equation 
(2.40)).
\end{remark}

\begin{remark} \label{remark: m = 7}
With some additional effort, one may in fact improve the condition $m > 12$ in Theorem \ref{thm: edge Wigner} to $m \geq 
7$. The basic idea is to match seven instead of four moments in Lemma \ref{lemma: four moment matching for Wigner}, and 
to use the resolvent expansion method from Section \ref{sect: proof of green fn comp}. We omit further details.
\end{remark}

The rest of this section is devoted to the proof of Theorems \ref{theorem: bulk universality for Wigner} and \ref{thm: 
edge Wigner}.

\subsection{Truncation}
For definiteness, we focus on real symmetric matrices, but the following truncation argument applies trivially to complex 
Hermitian matrices by truncating the real and imaginary parts separately. To simplify the presentation, we consider 
Wigner matrices for which $\sigma_{ij} = N^{-1/2}$. The proof for the more general matrices from Definition \ref{def gen Wigner} is the same; see also Remark 2.4 in \cite{EKYY}.

We begin by noting that, without loss of generality, we may assume that the distributions of the entries of $H$ are absolutely 
continuous. Otherwise consider the matrix $H + \epsilon_N V$, where $V$ is a GUE/GOE matrix independent of $H$, and 
$(\epsilon_N)$ is a positive sequence that tends to zero arbitrarily fast. (Note that the following argument is 
insensitive to the size of $\epsilon_N$.)

Let $H \equiv H^{\f x} = (h^{\f x}_{ij})$ be a Wigner matrix whose entries are of the form $h^{\f x}_{ij} = N^{-1/2} 
x_{ij}$ for some $x_{ij}$. We assume that the family $(x_{ij} \col i \leq j)$ is independent, and that each $x_{ij}$ satisfies
\begin{equation*}
\E x_{ij} \;=\; 0\,, \qquad \E \abs{x_{ij}}^2 = 1\,.
\end{equation*}
Moreover, we assume that there is an $m > 4$ and a constant $C_m \geq 1$, independent of $i$, $j$, and $N$, such that
\begin{equation*}
\E \abs{x_{ij}}^m \;\leq\; C_m\,.
\end{equation*}
In a first step we construct a truncated Wigner matrix $H^{\f y}$ from $H^{\f x}$. This truncation is performed in the 
following lemma.

\begin{lemma} \label{lemma: truncation}
Fix $m > 2$ and let $X$ be a real random variable, with absolutely continuous law, satisfying
\begin{equation*}
\E X \;=\; 0\,, \qquad \E X^2 \;=\; 1\,, \qquad \E \abs{X}^m \;\leq\; C_m\,.
\end{equation*}
Let $\lambda > 0$.
Then there exists a real random variable $Y$ that satisfies
\begin{equation*}
\E Y \;=\; 0\,, \qquad \E Y^2 \;=\; 1\,, \qquad \abs{Y} \;\leq\; \lambda\,, \qquad \P(X \neq Y) \;\leq\; 2 C_m \lambda^{-m}\,.
\end{equation*}
\end{lemma}

\begin{proof}
We introduce the abbreviations
\begin{equation*}
P \;\deq\; \P(\abs{X} > \lambda) \,, \qquad E \;\deq\; \E \pb{X \, \ind{\abs{X} > \lambda}} \,, \qquad V \;\deq\; \E 
\pb{X^2 \, \ind{\abs{X} > \lambda}}\,.
\end{equation*}
Using the assumption on $X$, Markov's inequality, and H\"older's inequality, we find
\begin{equation} \label{basic bounds for cutoff}
P \;\leq\; C_m \lambda^{-m} \,, \qquad \abs{E} \;\leq\; C_m \lambda^{-m + 1} \,, \qquad V \;\leq\; C_m \lambda^{-m + 2}\,.
\end{equation}

The idea behind the construction of $Y$ is to cut out the tail $\abs{X} > \lambda$, to add appropriate Dirac weights at 
$\pm \lambda$, and to adjust the total probability by cutting out the values $X \in [-a, a]$, where $a$ is an 
appropriately chosen small nonnegative number. For any $t$ satisfying $0 \leq t \leq 1/2$, we choose a nonnegative 
number $a_t$ such that $\P (X \in [-a_t, a_t]) = t$. Note that since $X$ is absolutely continuous such a number $a_t$ 
exists and the map $t \to a_t$ is continuous. Moreover, using $\E X^2 = 1$ and Markov's inequality we find that $a_t 
\leq 2$ for $t \leq 1/2$.  We define the quantities
\begin{equation*}
e_t \;\deq\; \E \pb{X \ind{-a_t \leq X \leq a_t}}\,, \qquad
v_t \;\deq\; \E \pb{X^2 \ind{-a_t \leq X \leq a_t}}\,,
\end{equation*}
which satisfy the trivial bounds
\begin{equation}\label{trivial bounds for truncation}
\abs{e_t} \;\leq\; 2 t\,, \qquad v_t \;\leq\; 4 t\,.
\end{equation}

We shall remove the values $(-\infty, - \lambda) \cup [-a_t, a_t] \cup (\lambda, \infty)$ from the range of $X$, and 
replace them with Dirac weights at $\lambda$ and $-\lambda$ with respective probabilities $p$ and $q$. Thus we are led to 
the system
\begin{equation} \label{system for truncation}
p + q \;=\; P + t \,, \qquad p - q \;=\; \lambda^{-1}(E + e_t) \,, \qquad p + q \;=\; \lambda^{-2}(V + v_t)\,.
\end{equation}
In order to solve \eqref{system for truncation}, we abbreviate the right-hand sides of the equations in \eqref{system 
for truncation} by $\alpha(t)$, $\beta(t)$, and $\gamma(t)$ respectively.

In a first step, we solve $t$ from the equation $\alpha(t) = \gamma(t)$. To that end, we observe that $\alpha(0) \leq 
\gamma(0)$, as follows from the trivial inequality $V \geq \lambda^2 P$. Moreover, $\alpha(1/2) \gg \gamma(1/2)$, by 
\eqref{basic bounds for cutoff} and \eqref{trivial bounds for truncation}. Since $\alpha(t)$ and $\gamma(t)$ are 
continuous, the equation $\alpha(t) = \gamma(t)$ has a solution $t_0$. Moreover, \eqref{basic bounds for cutoff} and 
\eqref{trivial bounds for truncation} imply that $t_0 \leq C_m \lambda^{-m} + 4 \lambda^{-2} t_0$, from which we get that 
$t_0 \leq 2 C_m \lambda^{-m}$. For the following we fix $t \deq t_0$.

In a second step, we solve the equations $p + q = \alpha(t)$ and $p - q = \beta(t)$ to get
\begin{equation*}
p \;=\; \frac{\alpha(t) + \beta(t)}{2} \,, \qquad q \;=\; \frac{\alpha(t) - \beta(t)}{2}\,.
\end{equation*}
We now claim that $\abs{\beta(t)} \leq \alpha(t)$. Indeed, a simple application of Cauchy-Schwarz yields $\abs{\beta(t)} 
\leq (\alpha(t) + \gamma(t))/2 = \alpha(t)$. Hence $p$ and $q$ are nonnegative. Moreover, the bounds \eqref{basic bounds 
for cutoff} and \eqref{trivial bounds for truncation} yield
\begin{equation*}
p + q \;\leq\; 2 C_m \lambda^{-m}\,.
\end{equation*}
Thus we have proved that \eqref{system for truncation} has a solution $(p,q,t)$ satisfying
\begin{equation*}
0 \;\leq\; p,q,t\;\leq\; 2 C_m \lambda^{-m}\,.
\end{equation*}

Next, let $I \deq (-\infty, - \lambda) \cup [-a_t, a_t] \cup (\lambda, \infty)$. Thus, $\P(X \in I) = p + q$.  Partition 
$I = I_1 \cup I_2$ such that $\P(X \in I_1) = p$ and $\P(X \in I_2) = q$.  Then we define
\begin{equation*}
Y \;\deq\; X \ind{X \notin I} + \lambda \ind{X \in I_1} - \lambda \ind{X \in I_2}\,.
\end{equation*}
Recalling \eqref{system for truncation}, we find that $Y$ satisfies $\E Y = 0$ and $\E Y^2 = 1$.
Moreover,
\begin{equation*}
\P (X \neq Y) \;=\; \P(X \in I) \;=\; p + q \;\leq\; 2 C_m \lambda^{-m}\,.
\end{equation*}
This concludes the proof.
\end{proof}

Note that the variable $Y$ constructed in Lemma \ref{lemma: truncation} satisfies $\E \abs{Y}^m \leq 3 C_m$.
Let $\rho$ be some exponent satisfying $0 < \rho < 1/2$ and assume that $m > 4$.
Using Lemma \ref{lemma: truncation} with $\lambda \deq N^\rho$, we construct, for each $x_{ij}$, a random variable $y_{ij}$ such that the family 
$(y_{ij} \col i \leq j)$ is independent and
\begin{equation} \label{properties of y}
\E y_{ij} \;=\; 0\,, \qquad \E y_{ij}^2 \;=\; 1\,, \qquad \abs{y_{ij}} \;\leq\; N^\rho\,,
\qquad \P(x_{ij} \neq y_{ij}) \;\leq\; 2 C_m N^{-\rho m}\,, \qquad \E \abs{y_{ij}}^m \;\leq\; 3 C_m\,.
\end{equation}
We define the new matrix $H^{\f y} = (h^{\f y}_{ij})$ through $h^{\f y}_{ij} \deq N^{-1/2} y_{ij}$.
In particular, we have
\begin{equation} \label{moment conditions on y}
\E h^{\f y}_{ij} \;=\; 0\,, \qquad \E \abs{h^{\f y}_{ij}}^2 \;=\; \frac{1}{N} \,,\qquad
\E \abs{h_{ij}^{\f y}}^p \;\leq\; \frac{1}{N q^{p - 2}}\,,
\end{equation}
where we set
\begin{equation} \label{definition of q for Wigner}
q \;\deq\; N^{1/2 - \rho}\,.
\end{equation}
Thus, $H^{\f y}$ satisfies Definition \ref{definition of H}.

\subsection{Moment matching}
Next, we construct a third Wigner matrix, $H^{\f z} = (h_{ij}^{\f z})$, whose entries are of the form $h_{ij}^{\f z} = 
N^{-1/2} z_{ij}$. We require that $z_{ij}$ have uniformly subexponential decay, i.e.
\begin{equation} \label{Wigner conditions}
\E z_{ij} \;=\; 0\,, \qquad \E \abs{z_{ij}}^2 \;=\; 1\,, \qquad \P(\abs{z_{ij}} \geq \xi) \;\leq\; \theta^{-1} \me^{- 
\xi^\theta}\,,
\end{equation}
for some $\theta > 0$ independent of $i,$ $j$, and $N$. We choose $z_{ij}$ so as to match the first four moments of 
$y_{ij}$.

\begin{lemma} \label{lemma: four moment matching for Wigner}
Let $y_{ij}$ satisfy \eqref{properties of y} with some $m > 4$. Then there exists a $z_{ij}$ satisfying \eqref{Wigner conditions} such that $\E 
z_{ij}^l = \E y_{ij}^l$ for $l = 1, \dots , 4$.
\end{lemma}

\begin{proof}
In fact, using an explicit construction similar to the one used in the proof of Theorem \ref{theorem: bulk 
universality}, $z_{ij}$ can be chosen to be supported at only three points.  We omit further details.
%
%
%
\end{proof}

It was proved in \cite{EYYrigidity}, Section 2.1, that the statement of Theorem \ref{theorem: bulk universality for 
Wigner} holds if the entries of $H$ satisfy the subexponential decay condition \eqref{Wigner conditions}.  Theorem 
\ref{theorem: bulk universality for Wigner} will therefore follow if we can prove that, for $b_N$ and $O$ as in Theorem 
\ref{theorem: bulk universality for Wigner}, we have
\begin{multline} \label{bulk statistics}
\lim_{N \to \infty} \int_{E - b_N}^{E + b_N} \frac{\dd E'}{2 b_N} \int \dd \alpha_1 \cdots \dd \alpha_n\, O(\alpha_1, 
\dots, \alpha_n)
\\
{}\times{} \frac{1}{\varrho_{sc}(E)^n}\pb{p_{\f x, N}^{(n)} - p_{\f z, N}^{(n)}} \pbb{E' +
\frac{\alpha_1}{N\varrho_{sc}(E)}, \dots, E' + \frac{\alpha_n}{N\varrho_{sc}(E)}} \;=\; 0\,,
\end{multline}
where $p_{\f x, N}^{(n)}$ and $p_{\f z, N}^{(n)}$ are the $n$-point marginals of the eigenvalue distributions of $H^{\f 
x}$ and $H^{\f z}$, respectively.

Similarly, it was proved in \cite{EYYrigidity}, Section 2.2, that the statement of Theorem \ref{thm: edge Wigner} holds 
if the entries of $H^{\f v}$ and $H^{\f w}$ both satisfy the subexponential decay condition \eqref{Wigner conditions}.  
Thus, Theorem \ref{thm: edge Wigner} will follow if we can prove that
\begin{equation}\label{edge statistics}
\P^{\f x} \pB{N^{2/3}(\lambda_N - 2) \leq s - N^{-\delta}} - N^{- \delta} \;\leq\; \P^{\f z} \pb{N^{2/3}(\lambda_N - 2) 
\leq s}
\;\leq\;
\P^{\f x} \pB{N^{2/3}(\lambda_N - 2) \leq s + N^{-\delta}} + N^{- \delta}\,,
\end{equation}
for some $\delta > 0$. Here we use $\P^{\f x}$ and $\P^{\f z}$ to denote the laws of ensembles $H^{\f x}$ and $H^{\f z}$ 
respectively.

We shall prove both \eqref{bulk statistics} and \eqref{edge statistics} by first comparing $H^{\f x}$ to $H^{\f y}$ and 
then comparing $H^{\f y}$ to $H^{\f z}$. The first step is easy: from \eqref{properties of y} we get
\begin{equation} \label{W error 2}
\P (H^{\f x} \neq H^{\f y}) \;\leq\; 2 C_m N^{2 - \rho m}\,.
\end{equation}
Thus, \eqref{bulk statistics} and \eqref{edge statistics} hold with $\f z$ replaced by $\f y$ provided that
\begin{equation} \label{W error 2 condition}
\rho m \;>\; 2\,.
\end{equation}

\subsection{Comparison of $H^{\f y}$ and $H^{\f z}$ in the bulk}
In this section we prove that \eqref{bulk statistics} holds with $\f x$ replaced by $\f y$, and hence complete the proof 
of Theorem \ref{thm: edge Wigner}.

We compare the local spectral statistics of $H^{\f y}$ and $H^{\f z}$ using the Green function comparison method from 
\cite{EYY}, Section 8. The key additional input is the local semicircle theorem for sparse matrices, Theorem 
\ref{LSCTHM}. We merely sketch the differences to \cite{EYY}. As explained in \cite{EYY}, the $n$-point correlation 
functions $p_N^{(n)}$ can be expressed (up to an error $N^{-c}$) in terms of expectations of observables $F$, whose 
arguments are products of expressions of the form $m(z_i + \ii \eta)$ where $\eta \deq N^{-1 - \epsilon}$. We assume 
that the first five derivatives of $F$ are polynomially bounded, uniformly in $N$. Using the local semicircle law for 
sparse matrices, Theorem \ref{LSCTHM}, we may control the Green function matrix elements down to scales $N^{-1 - 
\epsilon}$, uniformly in $E$.  (Note that in \cite{EYY}, the spectral edge had to be excluded since the bounds derived 
there were unstable near the edge, unlike our bound \eqref{Gij estimate}.) This allows us to compare the local 
eigenvalue statistics of the matrix ensembles at scales $N^{-1 - \epsilon}$, which is sufficiently accurate for both 
Theorems \ref{theorem: bulk universality for Wigner} and \ref{thm: edge Wigner}.

We use the telescopic summation and the Lindeberg replacement argument from \cite{EYY}, Chapter 8, whose notations we 
take over without further comment; see also Section \ref{sect: proof of green fn comp}. A resolvent expansion yields
\begin{equation*}
S \;=\; R - N^{-1/2} RVR + N^{-1} (RV)^2 R - N^{-3/2} (RV)^3 R + N^{-2} (RV)^4 R - N^{-5/2} (RV)^5 S\,.
\end{equation*}
Next, note that, by \eqref{moment conditions on y} and \eqref{Wigner conditions}, both ensembles $H^{\f y}$ and $H^{\f 
z}$ satisfy Definition \ref{definition of H} with $q$ defined in \eqref{definition of q for Wigner}.
Hence we may invoke Theorem \ref{LSCTHM} with $f = 0$, and in particular \eqref{Gij estimate}, for the matrices $H^{\f 
y}$ and $H^{\f z}$. Choosing $\eta = N^{-1 - \epsilon}$ for some $\epsilon > 0$, we therefore get
\begin{equation*}
\abs{R_{ij}(E + \ii \eta)} \;\leq\; N^{2 \epsilon} \,, \qquad \abs{S_{ij}(E + \ii \eta)} \;\leq\; N^{2 \epsilon}
\end{equation*}
\hp{\xi}{\nu}.

Consider now the difference $\E^{\f y} - \E^{\f z}$ applied to some fixed observable $F$ depending on traces (normalized 
by $N^{-1}$) of resolvents and whose derivatives have at most polynomial growth.
Since the first four moments of the entries of $H^{\f y}$ and $H^{\f z}$ coincide by Lemma \ref{lemma: four moment 
matching for Wigner}, the error in one step of the telescopic summation is bounded by the expectation of the rest term 
in the resolvent expansion, i.e.
\begin{equation*}
N^{C \epsilon} N^{-5/2} \E (RV)^5 S \;\leq\; N^{-5/2 + C \epsilon} \max_{a,b} \E \abs{V_{ab}}^5 \;\leq\; N^{-5/2 + C 
\epsilon} C_m N^\rho\,,
\end{equation*}
where in the last step we used \eqref{properties of y}. The first factor $N^{C \epsilon}$ comes from the polynomially 
bounded derivatives of $F$. Summing up all $O(N^2)$ terms of the telescopic sum, we find that the difference $\E^{\f 
y} - \E^{\f z}$ applied to $F$ is bounded by
\begin{equation} \label{W error 1}
N^{-1/2 + C \epsilon + \rho}\,.
\end{equation}
Combining \eqref{W error 1} and \eqref{W error 2}, we find that both \eqref{bulk statistics} follows provided that
\begin{equation} \label{final condition for 4 moments}
-\frac{1}{2} + C \epsilon + \rho \;<\; 0 \,, \qquad \rho m \;>\; 2\,.
\end{equation}
Since $m > 4$ is fixed, choosing first $1/2 - \rho$ small enough and then $\epsilon$ small enough yields \eqref{final 
condition for 4 moments}. This completes the proof of Theorem \ref{theorem: bulk universality for Wigner}.

\subsection{Comparison of $H^{\f y}$ and $H^{\f z}$ at the edge} \label{Hy Hz edge}
In order to prove \eqref{tw for Wigner} under the assumption $m > 12$, we may invoke Proposition \ref{twthm}, which 
implies that \eqref{edge statistics} holds with $\f x$ replaced by $\f y$, provided that $\phi = 1/2 - \rho > 1/3$, 
i.e.\ $\rho < 1/6$. Together with the condition \eqref{W error 2 condition}, this implies that \eqref{tw for Wigner} 
holds if $m > 12$.

\appendix

\section{Regularization of the Dyson Brownian motion} \label{appendix: DBM}

In this appendix we sketch a simple
regularization argument needed to prove two results
concerning the Dyson Brownian motion (DBM). This argument can be used as a substitute for earlier, more involved, proofs
given in Appendices A and B of \cite{ESYY}
on the existence of the dynamics restricted
to the subdomain $\Sigma_N \deq \{ \bx \col x_1<x_2<\cdots<x_N\}$,
and on the applicability of the Bakry-Emery method. The argument presented in this section is also more probabilistic in 
nature than the earlier proofs of \cite{ESYY}.

For applications in Section \ref{section: DBM} of this paper, some minor adjustments to the argument below are needed to 
incorporate the separate treatment
of the largest eigenvalue.  These modifications  are straightforward, and we shall only sketch the argument for the 
standard DBM.

\begin{theorem} \label{theorem: flow2}
Fix $n \geq 1$ and let $\f m = (m_1, \dots, m_n) \in \N^n$ be an
 increasing family of indices. Let $G \col \R^n \to \R$ be a 
continuous function of compact support and set
\begin{align*}
\cal G_{i, \f m}(\f x) \;\deq\; G \pb{N(x_i - x_{i + m_1}), N(x_{i + m_1} - x_{i + m_2}),
 \dots, N(x_{i + m_{n - 1}} - x_{i + m_n})}\,.
\end{align*}
Let $\gamma_1, \dots, \gamma_{N }$ denote the classical locations of the  eigenvalues
and set
\begin{equation}
Q \;\deq\; \sup_{t \in [t_0, \tau]} \sum_{i = 1}^{N } \int (x_i - \gamma_i)^2 f_t \, \dd \mu\, 
\end{equation}
Choose an $\epsilon > 0$.
Then for any $\rho$ satisfying $0 < \rho < 1$, and setting $\tau = N^{-\rho}$, there exists a $\bar \tau \in [\tau/2, 
\tau]$ such that, for any $J \subset \{1, 2, \dots, N - m_n - 1\}$, we have
\begin{align} \label{local ergodicity of DBM 2}
\absbb{\int \frac{1}{\abs{J}} \sum_{i \in J} \cal G_{i, \f m} \, f_{\bar \tau} \, \dd \mu - \int \frac{1}{\abs{J}} 
\sum_{i \in J} \cal G_{i, \f m} \, \dd \mu  } \;\leq\; C N^\epsilon \sqrt{\frac{N Q + 1
}{\abs{J}  \tau }}
\end{align}
for all $N \geq N_0(\rho)$.
Here $\mu=\mu^{(N )}$ is the equilibrium measure of the $N$ eigenvalues of the GOE.
\end{theorem}

Define $\mu_\beta(\f x) = C  e^{ - N \beta \cal H(\f x) }$ as in \eqref{GOE measure} and \eqref{GOEH}, but introducing a 
parameter $\beta$ so that $\mu_\beta$ is the equilibrium measure of the usual $\beta$-ensemble which is invariant under 
the ($\beta$-dependent) DBM. We remark that Theorem \ref{theorem: flow2}  holds for all $\beta \ge 1$ with an identical 
proof.  The following lemma holds more generally for  $\beta > 0$.

Let $\om \deq C \mu_\beta e^{ - N \sum_j U_j(x_j)}$, where $U_j$ is a $C^2$-function satisfying
\be
\min_j  U_j''(x) \;\ge\; \tau^{-1}
\ee
for some $\tau < 1$. For the following lemma we recall the definition \eqref{DBMD} of the Dirichlet form.

\begin{lemma} \label{411} Let $\beta>0$ and $q \in H^1(\dd   \omega) $ be a  probability density with respect to
 $\om$. 
 Then for  any $\beta > 0$  and  any $J \subset \{1, 2, \dots, N - m_n - 1\}$ and 
any $t > 0$ we have
\begin{align}\label{443}
\absbb{\int \frac{1}{\abs{J}} \sum_{i \in J} \cal G_{i, \f m} \, q \, \dd   \omega - \int \frac{1}{\abs{J}} \sum_{i 
\in J} \cal G_{i, \f m} \, \dd   \omega} \;\leq\; C \sqrt{\frac{D_{ \omega }(\sqrt{q}) \, t}{\abs{J}}} + C 
\sqrt{S_{ \omega}(q)} \, \me^{- c t / \tau}\,.
\end{align}
\end{lemma}

Recall that the DBM is defined via the stochastic differential equation
\begin{align} \label{DBMbeta}
\dd x_i \;=\; \frac{\dd B_i}{\sqrt{N}} +
 \beta \pbb{-\frac{1}{4} x_i + \frac{1}{2N} \sum_{j \neq i} \frac{1}{x_i - x_j}} \dd t 
\for i = 1, \dots, N\,,
\end{align}
where $B_1, \dots, B_N$ is a family of independent standard Brownian motions.
It was proved in \cite{AGZ}, Lemma 4.3.3, that there is a unique strong solution to \eqref{DBMbeta} for all $\beta \ge 
1$.

For any $\delta>0$
define the extension $ \mu_\beta^\delta$  of  the measure $\mu_\beta$ from  $\Sigma_N$
 to $\R^N$ by replacing  the singular logarithm with a $C^2$-function. To that end, we introduce the approximation 
parameter $\delta > 0$ and define, as in Section \ref{section: DBM}, for
$\bx \in \R^{N}$,
\begin{align*}
   \cal H_\delta (\f x) \;\deq\; \sum_i \frac{1}{4} x_i^2 - \frac{1}{N} \sum_{i < j} \log_\delta
 (x_j - x_i) \,
\end{align*}
where we set
\begin{align*}
\log_\delta(x) \;\deq\; \ind{x \geq \delta} \log x + \ind{x < \delta} \pbb{\log \delta + \frac{x - \delta}{\delta} -
\frac{1}{2 \delta^2} (x - \delta)^2}\,.
\end{align*}
It is easy to check that $\log_\delta \in C^2(\R)$, is concave, and satisfies
\begin{align*}
\lim_{\delta \to 0} \log_\delta(x) \;=\;
\begin{cases}
\log x &\text{if } x > 0
\\
-\infty &\text{if } x \leq 0\,.
\end{cases}
\end{align*}
Furthermore, we have the lower bound 
\begin{align*}
\partial_x^2 \log_\delta(x) \;\ge\;
\begin{cases}
- \frac 1 {x^2} &\text{if } x > \delta 
\\
-\frac 1 {\delta^2}  &\text{if } x \leq \delta\,.
\end{cases}
\end{align*}
Similarly, we can extend the measure $\omega$ to $\R^N$ by setting $\om^\delta = C e^{ - N \sum_j U_j(x_j)} 
\mu_\beta^\delta$.

\begin{lemma} \label{41}
Let $q \in L^\infty(\dd   \omega^\delta) $ be a $C^2$ probability density. Then for  $\delta \le 1/N$, $\beta > 0$  and  any $J \subset \{1, 2, \dots, N - m_n - 1\}$ and 
any $t > 0$ we have
\begin{align}\label{44}
\absbb{\int \frac{1}{\abs{J}} \sum_{i \in J} \cal G_{i, \f m} \, q \, \dd   \omega^\delta - \int \frac{1}{\abs{J}} \sum_{i 
\in J} \cal G_{i, \f m} \, \dd   \omega^\delta} \;\leq\; C \sqrt{\frac{D_{ \omega^\delta }(\sqrt{q}) \, t}{\abs{J}}} + C 
\sqrt{S_{ \omega^\delta}(q)} \, \me^{- c t / \tau}\,.
\end{align}
\end{lemma}
\begin{proof}
The proof of Theorem 4.3 in \cite{ESYY} applies with merely cosmetic changes; now however the dynamics is defined on 
$\R^N$ instead of $\Sigma_N$, so that complications arising from the boundary are absent. The condition $\delta \le 1/N$ 
is needed since we use the singularity of $\partial_x^2 \log x $ to generate a factor $1/N^2$ in the regime $ x \le C 
/N$ in the proof.
\end{proof}

\begin{proof}[Proof of Lemma~\ref{411}]
Suppose that  $q$ is a probability density in $\Sigma_N$ with respect to $\om$.
We extend $q$ to be zero outside $\Sigma$ and let $q_\e\in C^2$ be any regularization
 of $q$ on $\R^N$ that converges to $q$ in $H^1(\om)$.
 Then there is a constant $C_{\e, \delta}$ such that $ q_\e^\delta \deq C_{\e, \delta} \, q_\e $ is a probability 
density with respect to $\om_\delta$.  Thus \eqref{44} holds with $q$ replaced by $q_\e^\delta$.  Taking the limit 
$\delta \to 0$ and then $\e \to 0$, we have
\begin{multline}\label{441}
\absBB{\int \frac{1}{\abs{J}} \sum_{i \in J} \cal G_{i, \f m} \, q \, \dd   \omega  - \int \frac{1}{\abs{J}} \sum_{i \in 
J} \cal G_{i, \f m} \, \dd   \omega}
\\
\leq\; C  \lim_{\e \to 0}  \lim_{\delta \to 0} \sqrt{\frac{D_{ \omega^\delta }(\sqrt{q_\e^\delta}) \, t}{\abs{J}}} + C 
\lim_{\e \to 0}  \lim_{\delta \to 0}  \sqrt{S_{ \omega^\delta}(q^\delta_\e)} \, \me^{- c t / \tau}\,.
\end{multline}
Notice that $\om_\delta \to {\bf 1}(\Sigma_N) \om$ weakly as $\delta \to 0$. Thus
\[
\lim_{\e \to 0}  \lim_{\delta \to 0}  D_{ \omega^\delta }(\sqrt{q_\e^\delta}) \;=\; \lim_{\e \to 0}  D_{ \omega 
}(\sqrt{q_\e})  \;=\;  D_{ \omega }(\sqrt{q})
\]
provided that $q \in H^1(\omega)$. This proves Lemma \ref{411}.  Notice that the proof did  not use the existence of 
DBM; instead, it used the existence of the regularized DBM.
\end{proof}

\begin{proof}[Proof of Theorem \ref{theorem: flow2}]
Write
\[
\int \frac{1}{\abs{J}} \sum_{i \in J} \cal G_{i, \f m} \, f_{t} \, \dd \mu
\;=\;  \E^ {f_{0}   \mu}  \E^{\bx_0} \frac{1}{\abs{J}} \sum_{i \in J} \cal G_{i, \f m} (\bx(t) )\,.
\]
Here $\E^{\bx_0}$ denotes expectation with respect to the law
of the DBM $(\f x(t))_t$ starting from $\bx_0$, and $\E^{f_0 \mu}$ denotes expectation of $\f x_0$ with respect to the 
measure $f_{0}   \mu$. Let $\E_\delta$ denote expectation with respect to the regularized DBM. Then we have
\[
  \E^{\bx_0} \frac{1}{\abs{J}} \sum_{i \in J} \cal G_{i, \f m} (\bx(t) )
  \;=\; \lim_{\delta \to 0}   \E^{\bx_0}_\delta  \frac{1}{\abs{J}}
 \sum_{i \in J} \cal G_{i, \f m} (\bx(t) ),
\]
where we have used the existence of a strong solution to the DBM (see \cite{AGZ}, Lemma 4.3.3)
and that the dynamics remains in $\Sigma_N$ almost surely. Hence \[
\int \frac{1}{\abs{J}} \sum_{i \in J} \cal G_{i, \f m} \, f_{t} \, \dd \mu
\;=\;  \lim_{\delta \to 0}  \int \frac{1}{\abs{J}}
 \sum_{i \in J} \cal G_{i, \f m} \, f_{t}^\delta \, \dd \mu^\delta,
\]
where $f_{t}^\delta $ is the solution to the regularized DBM at the time $t$ with
 initial data $f_0 \mu/\mu^\delta$. Using that \eqref{local ergodicity of DBM 2} holds for the regularized dynamics, and 
taking the limit $\delta \to 0$, we complete the proof.
\end{proof}


%

%

\end{document}